\documentclass[a4paper,11pt]{article}
\usepackage[utf8]{inputenc}
\usepackage[T1]{fontenc}
\usepackage{amsmath, amsthm, amssymb}
\usepackage{physics} 
\usepackage{braket} 
\usepackage{faktor} 
\usepackage{tikz-cd} 
\usepackage{graphicx} 
\usepackage{enumitem}
\usepackage{comment}
\usepackage[toc,page]{appendix} 

\usepackage{xcolor}      
\usepackage{soul}        

\newcommand{\mres}{\mathbin{\vrule height 1.6ex depth 0pt width 0.13ex\vrule height 0.13ex depth 0pt width 0.8ex}}

\theoremstyle{plain}
\newtheorem{thm}{Theorem}[section]
\newtheorem{prop}[thm]{Proposition}
\newtheorem{cor}[thm]{Corollary}
\newtheorem{lem}[thm]{Lemma}
\newtheorem*{thm*}{Theorem}
\newtheorem*{prop*}{Proposition}
\newtheorem*{cor*}{Corollary}
\newtheorem*{lem*}{Lemma}

\theoremstyle{definition}
\newtheorem{defi}[thm]{Definition}
\newtheorem{exa}[thm]{Example}
\newtheorem*{defi*}{Definition}

\theoremstyle{remark}
\newtheorem{rmk}{Remark}[section]
\newtheorem*{rmk*}{Remark}

\newcommand{\R}{\mathbf{R}}
\newcommand{\Z}{\mathbf{Z}}
\newcommand{\N}{\mathbf{N}}
\newcommand{\spt}{\mathrm{spt}}
\newcommand{\HH}{\mathcal{H}}
\newcommand{\LL}{\mathcal{L}}
\newcommand{\II}{\mathcal{I}}
\newcommand{\IL}{\Lambda}
\newcommand{\GT}{\varphi}
\newcommand{\GL}{M}
\newcommand{\ka}{\kappa_0}
\newcommand{\kb}{\kappa_1}

\title{Weak Limits of Quasiminimizing Sequences}
\author{Camille Labourie}
\date{}

\begin{document}

\maketitle

\begin{abstract}
    We show that the weak limit of a quasiminimizing sequence is a quasiminimal set. This generalizes the notion of weak limit of a minimizing sequences introduced by De Lellis, De Philippis, De Rosa, Ghiraldin and Maggi. This result is also analogous to the limiting theorem of David in local Hausdorff convergence. The proof is based on the construction of suitable deformations and is not limited to the ambient space $\R^n \setminus \Gamma$, where $\Gamma$ is the boundary. We deduce a direct method to solve various Plateau problems, even minimizing the intersection of competitors with the boundary. Furthermore, we propose a structure to build Federer--Fleming projections as well as a new estimate on the choice of the projection centers.
\end{abstract}

\tableofcontents

\section{Introduction}
\subsection{Background}
Inspired by soap films, the Plateau problem is an old problem which consists in minimizing the area of a surface spanning a boundary. It admits many mathematical formulations corresponding to different classes of surfaces \emph{spanning the boundary} (also called \emph{competitors}) and different \emph{areas} to minimize. 

\begin{figure}[ht]
    \begin{center}
        \includegraphics[width=.4\linewidth]{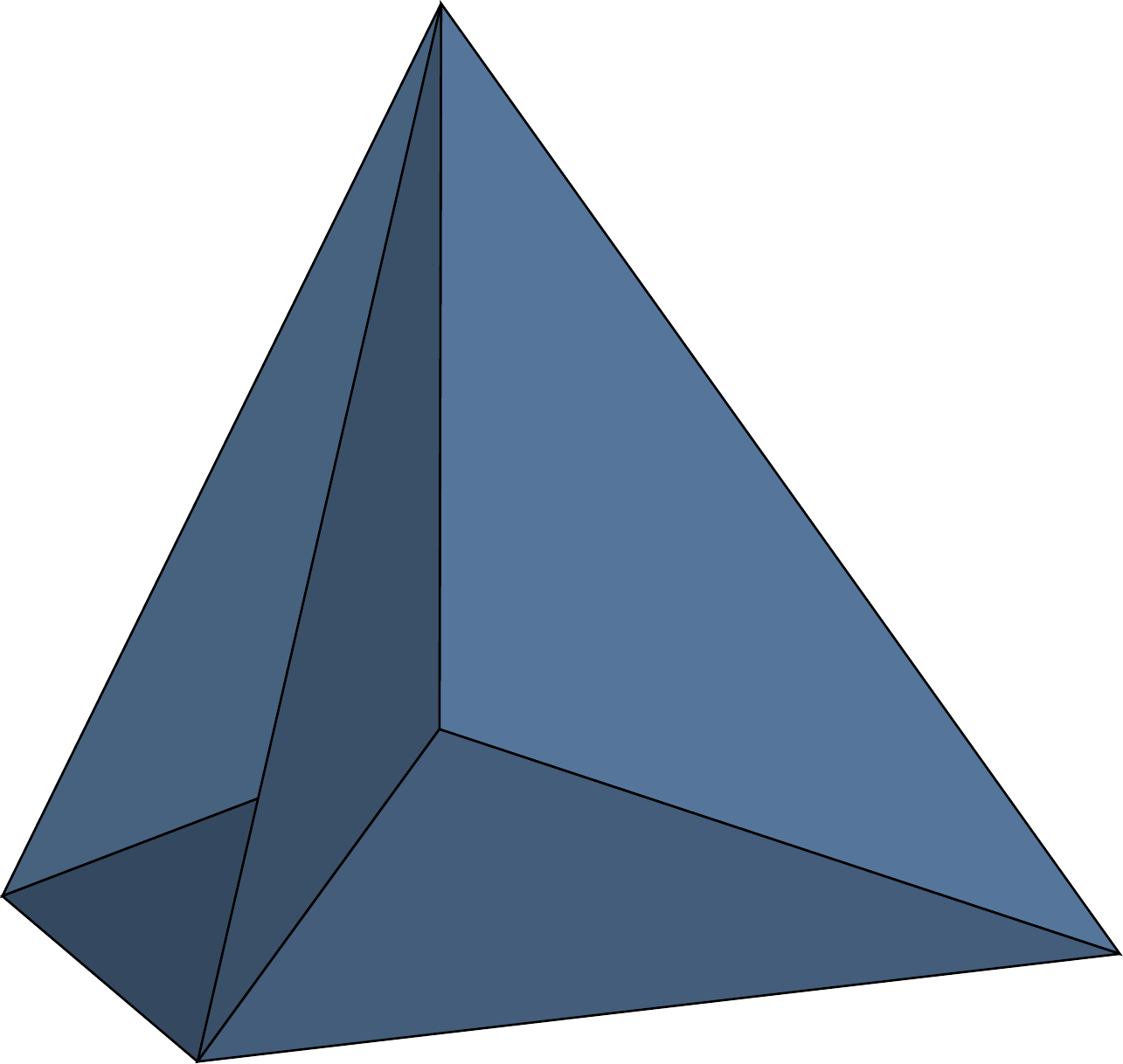}
        \qquad \qquad
        \includegraphics[width=.4\linewidth]{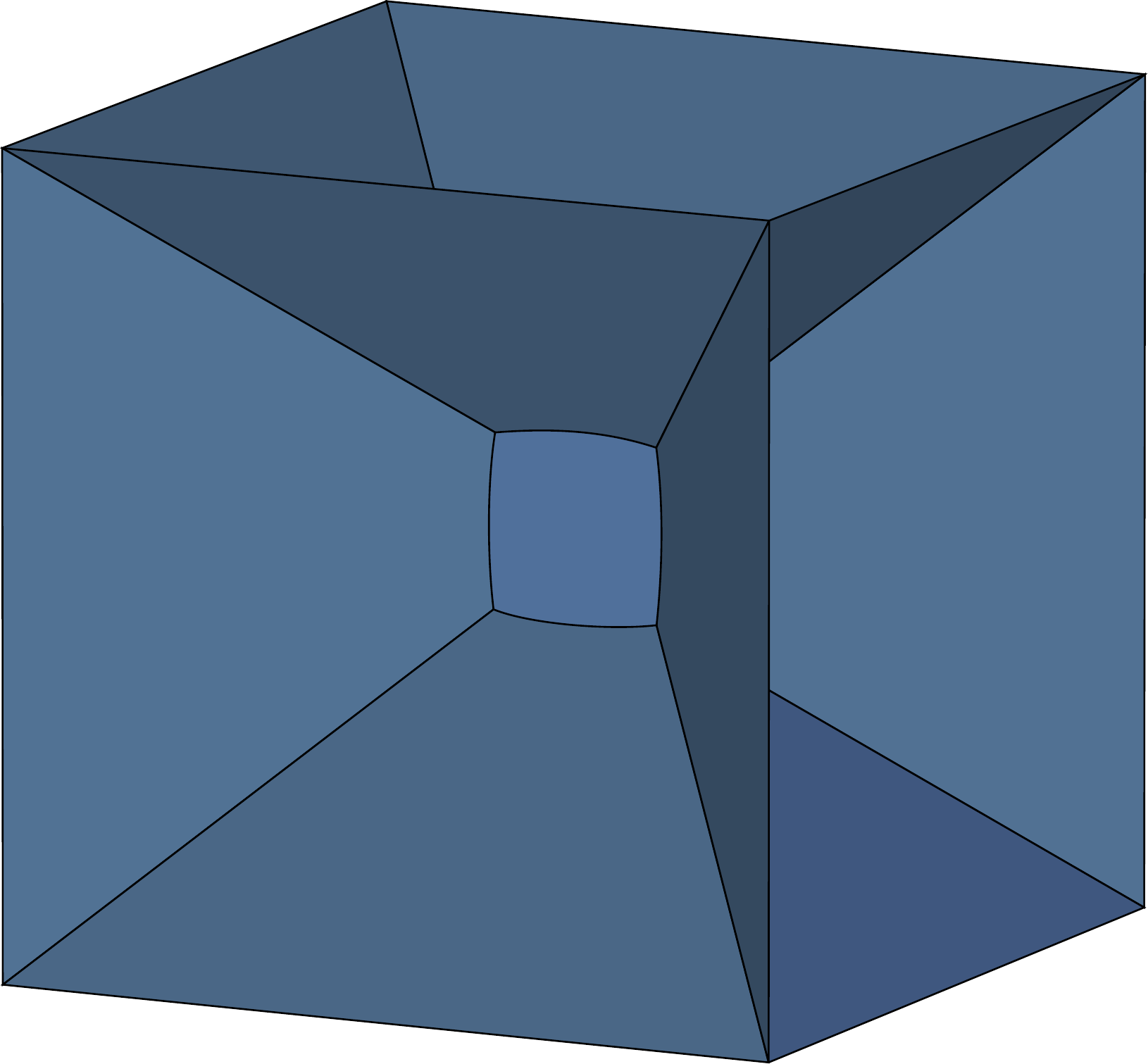}
        \caption{Soap films spanning the skeleton of a tetrahedron (left) and the skeleton of a cube (right).}
    \end{center}
\end{figure}

The usual strategy of existence is the direct method of the calculus of variation, that is, taking a limit of a minimizing sequence. It is in general difficult to have both a compactness principle on the class of competitors and a lower semicontinuity principle on the area. We present a few approaches below to motivate our main results (we also advise \cite{SteinConf}).

We are interested in the formulations that allow to describe the singularities of soap films. We work in the Euclidean space $\R^n$. We denote by $\Gamma$ a compact in $\R^n$ (the boundary) and we minimize the Hausdorff measure $\HH^d$ ($d=1,\ldots,n$) of competitors. It is possible to replace the Hausdorff measure by integrals of elliptic integrands. We present three class of competitors: the homological competitors of Reifenberg \cite{Reifenberg}, the competitors of Harrison--Pugh \cite{HaPu} and the sliding competitors of David \cite{Sliding}. These formulations do not specify the topological structure of competitors.

\begin{figure}[ht]
    \centering
    \begin{minipage}{0.46\linewidth}
        \centering
        \includegraphics[width=0.8\linewidth]{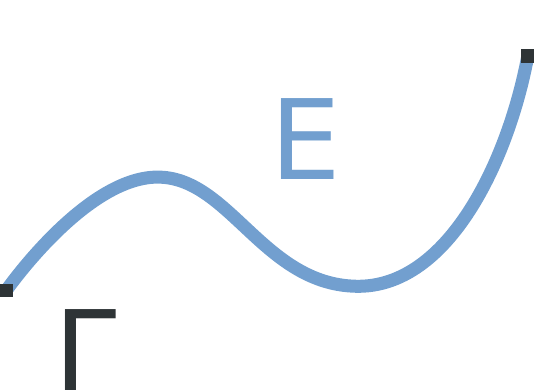}
        \caption{The competitor has a \emph{fixed boundary}.}
    \end{minipage}
    \qquad
    \begin{minipage}{0.46\linewidth}
        \centering
        \includegraphics[width=0.8\linewidth]{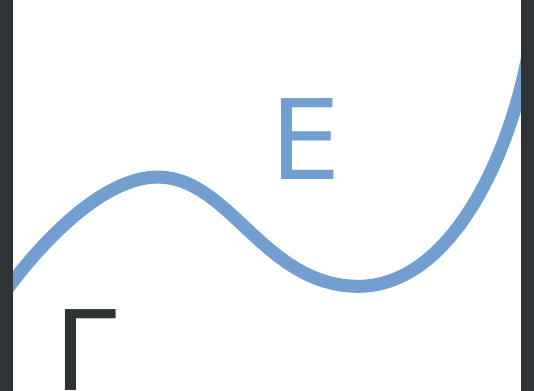}
        \caption{The competitor has a \emph{sliding boundary}. The intersection $E \cap \Gamma$ may not be negligible. One can minimize the whole set $E$ or only $E \setminus \Gamma$.}
    \end{minipage}
\end{figure}

Reifenberg minimizes the compact sets which fill the cycles of dimension $(d-1)$ of $\Gamma$. We give a slight generalization of Reifenberg original definition to allow sliding boundaries.
\begin{defi}[Reifenberg competitor]
    A competitor is a compact set $E \subset \R^n$ such that the morphism induced by inclusion,
    \[
        \begin{tikzcd}
            H_{d-1}(\Gamma) \arrow[r] & H_{d-1}(E \cup \Gamma),
        \end{tikzcd}
    \]
    is zero. One can also require that $E$ cancels only a fixed subgroup of $H_{d-1}(\Gamma)$.
\end{defi}

\begin{figure}[ht]\label{RH}
    \begin{center}
        \includegraphics[width=0.6\linewidth]{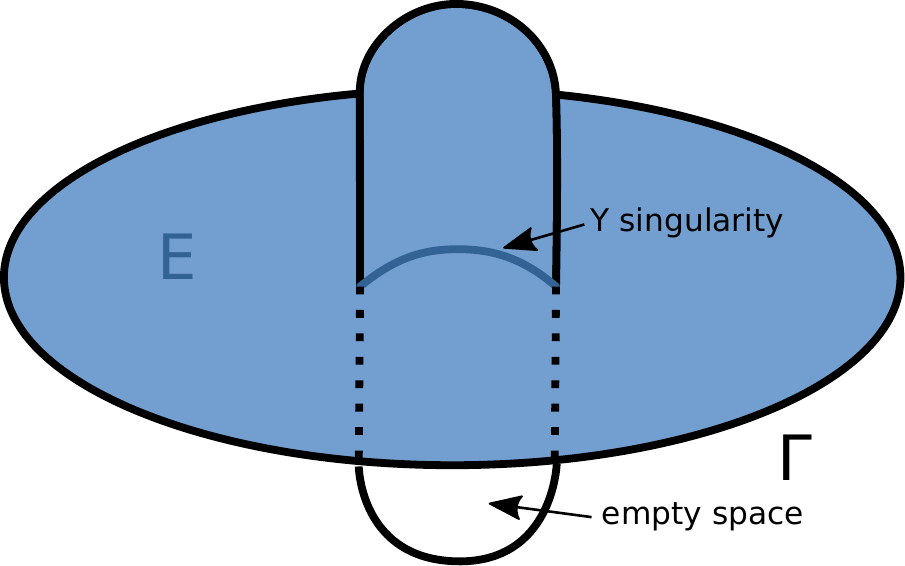}
        \caption{A minimizer of Harrison--Pugh which is not a Reifenberg minimizer.}\label{figure_reifenberg}
    \end{center}
\end{figure}

Harrison and Pugh minimize the compact sets which intersect the links of $\Gamma$.
\begin{defi}[Harrison--Pugh competitor]
    A competitor is a compact set $E \subset \R^n$ such that for any embedding
    \begin{equation*}
        \gamma \colon \mathbf{S}^{n-d} \to \R^n \setminus \Gamma
    \end{equation*}
    of the sphere $\mathbf{S}^{n-d}$ into $\R^n$, we have $\gamma \cap E \ne \emptyset$. One can also restrict the embeddings to a subclass which is closed under homotopy.
\end{defi}

David aims to minimize a fixed compact $E_0$ under the action of continuous deformations $f$ which preserve the boundary over time (\emph{sliding boundary}). This evokes a soap film sliding along a tube for example.
\begin{defi}[Sliding competitor]
    Fix $E_0$ a compact of $\R^n$ such that $\HH^d(E_0) < \infty$ (an initial surface spanning $\Gamma$). A \emph{sliding deforation} is a Lipschitz function $f\colon \R^n \to \R^n$ such that there exists a continuous homotopy $F\colon I \times \R^n \to \R^n$ satisfying the following conditions
    \begin{align*}
&F_0 = \mathrm{id},\ F_1 = f,\\
&\forall t \in I,\ F_t(\Gamma) \subset \Gamma.
    \end{align*}
    The sliding competitors of $E_0$ are the set of the form $f(E_0)$ where $f$ is a sliding deformation.
\end{defi}

The minimizers of Reifenberg, Harrison--Pugh and David satisfy an intermediate property: their area is minimal under sliding deformations. This motivates the notion of \emph{minimal sets}. They are closed and $\HH^d$ locally finite sets $E \subset \R^n$ such that for every sliding deformation $f$ of $E$ in a small ball,
\begin{equation*}
    \HH^d(E \cap W) \leq \HH^d(f(E \cap W)),
\end{equation*}
where
\begin{equation*}
    W = \set{x \in \R^n | \phi(x) \ne x}.
\end{equation*}
This property echoes the elasticity and stability of soap films under (small) deformations. We underline that the deformation has two constraints: the ball where the deformation takes place and the sliding boundary. \emph{Quasiminimal sets} are sets whose area cannot be decreased below a fixed percentage (modulo a lower-order term):
\begin{equation*}
    \HH^d(E \cap W) \leq \kappa \HH^d(\phi(E \cap W)) + h \mathrm{diam}(B)^d.
\end{equation*}
Almgren has studied such sets under the name \emph{restricted sets} but without sliding boundary and without lower-order term \cite[Definition II.1]{A2}. David and Semmes introduced quasiminimal sets as a minor variation of restricted sets in \cite[Definition 1.9]{DS}. David added the lower error term in \cite[Definition 2.10]{Holder} and the sliding boundary in \cite[Definition 2.3]{Sliding}.

The quasiminimal sets have weak regularity properties: the local Ahlfors-regularity and the rectifiability. They enjoy interesting limiting properties. Say that $(E_k)$ is a sequence of quasiminimal sets of uniform parameters. Assume that it converges in local Hausdorff distance to a limit set $E$. Then for all open set $V$,
\begin{equation}
    \HH^d(E \cap V) \leq \liminf_k \HH^d(E_k \cap V)
\end{equation}
and for all compact set $K$,
\begin{equation}
    \limsup_k \HH^d(E_k \cap K) \leq (1 + Ch) \kappa \HH^d(E \cap K).
\end{equation}
In addition, $E$ is a quasiminimal set (with the same parameters). David proved it a first time in \cite{Limits} for minimal sets without sliding boundary, a second time in \cite[Section 3]{Holder} for quasiminimal sets without sliding boundary and a last time in \cite[Theorem 10.8]{Sliding} for quasiminimal sets with a sliding boundary. Another proof is given by Fang in \cite{FangConvergence}.

We finally discuss some strategies to solve Plateau problem. Let us say that we are given a class of competitors which is closed under deformation. The first strategy is the following.
\begin{enumerate}
    \item Consider a minimizing sequence of competitors $(E_k)$, all included in a fixed compact.
    \item Extract a subsequence so that $(E_k)$ converge to a limit set $E$ in Hausdorff distance.
    \item It might be necessary to build an alternative minimizing sequence to make sure that $\HH^d(E_\infty) \leq \liminf_k \HH^d(E_k)$. Indeed, the area is not lower semicontinuous with respect to Hausdorff distance convergence. There might be thin parts whose Hausdorff measure goes to $0$ but which become more and more dense so that the limit set is too large.
    \item Check whether the limit set $E$ is a competitor.
\end{enumerate}
This is the original strategy of Reifenberg (\cite{Reifenberg},1960). Reifenberg works with the \v{C}ech homology theory because its continuity property ensures that if the set $(E_k)$ are Reifenberg competitors, then $E$ is also a Reifenberg competitor. Reifenberg builds an alternative minimizing sequence by cutting the undesirable parts and patching the holes. He uses the Exactness axiom to check that the resulting set is a competitor. However, for this axiom to holds true, the coefficient group of the \v{C}ech theory needs to be compact. This excludes the natural case of $\Z$.

This is also the strategy of Fang in \cite[Theorem 1.1]{FangReifenberg}. Fang builds an alternative minimizing sequence composed of quasiminimal sets (with uniform parameters). He relies on a complex construction of Feuvrier \cite{Feuvrier}. This allows him to solve the Reifenberg problem without restriction on the coefficient group.

In \cite{A1}, Almgren presents another strategy for solving the Reifenberg problem which relies on the theory of currents, flat chains and integral varifolds. Fang and Kolasinksi clarify and axiomatize this approach in \cite[Theorem 3.20]{FangAlmgren}.

In \cite{I1} and \cite{I2}, De Lellis, De Philippis, De Rosa, Ghiraldin and Maggi introduce a new scheme. 
\begin{enumerate}
    \item Consider a minimizing sequence of competitors $(E_k)$, all included in a fixed compact.
    \item Extract a subsequence so that the Radon measures $(\HH^d \mres E_k)$ converge to a limit measure $\mu$.
    \item Show that $\mu$ has a special structure:
        \begin{equation*}
            \mu = \HH^d \mres E
        \end{equation*}
        where $E$ is the support of $\mu$. This implies the lower semicontinuity of the area $\HH^d(E) \leq \liminf_k \HH^d(E_k)$. In addition, one can show that $E$ is minimal under deformation.
    \item Check whether the limit set $E$ is a competitor.
\end{enumerate}
The authors also replace the Hausdorff measure $\HH^d$ by integrals of elliptic integrands in \cite{I4} and \cite{I5}. A key step of the proof is to show that $E$ is $\HH^d$ rectifiable. Their proof uses the Preiss's rectifiability theorem in the first articles and an extension of Allard's rectifiability theorem (\cite{I3}) in the last articles. It has also been reproved without relying on the Preiss theorem or the Allard rectifiability theorem in \cite[Section 6, Theorem 6.7]{DRK}. This strategy allows to solve the problem of Reifenberg and the problem of Harrison--Pugh.  

Until now, these schemes have only been established in the ambient space $\R^n \setminus \Gamma$. The convergence takes place in $\R^n \setminus \Gamma$ and the intersection of competitors with $\Gamma$ is not minimized.

These ideas do not allow to solve the sliding problem directly because the class of sliding competitors is not closed under convergence in Hausdorff distance or convergence in Radon measure.

\subsection{Main Results}
\textbf{Notation} For $x \in \R^n$ and $r > 0$, $B(x,r)$ is the open ball of center $x$ and radius $r$. For $h \geq 0$, the notation $h B(x,r)$ means $B(x,hr)$.  The interval $[0,1]$ is denoted by the capital letter $I$. Given a set $E \subset \R^n$ and a function $F\colon I \times E \to \R^n$, the notation $F_t$ means $F(t,\cdot)$. Given two sets $A, B \subset \R^n$, the notation $A \subset \subset B$ means that there exists a compact set $K \subset \R^n$ such that $A \subset K \subset B$. The terms ``a closed set $A \subset B$'' or ``a closed subset $A$ of $B$'' mean that $A$ is relatively closed.

Our ambient space is an open set $X$ of $\R^n$. We fix an integer $1 \leq d \leq n$. The symbol $\HH^d$ denotes the Hausdorff measure of dimension $d$.  To each ball $B(x,r) \subset X$, we associate a scale $s \in [0,\infty]$. If $X = \R^n$, the scale is given by the radius but if $X \ne \R^n$, we would like to express that the scale goes to $\infty$ when $r \to \mathrm{d}(x,X^c)$. For example, we can choose the parameter $s \in [0,\infty]$ such that $r = r_s$ where
\begin{equation*}\label{r_s}
    r_s(x) = \min\set{\tfrac{s}{1+s}\mathrm{d}(x,X^c),s}.
\end{equation*}
The ball $B(x,r)$ is of scale $\leq s$ if and only if $r \leq r_s(x)$. When there is no ambiguity, we write $r_s$ instead of $r_s(x)$. In practice, we do not use the formula but only two properties.
\begin{enumerate}
    \item For $s \in [0,\infty]$ and $\lambda > 0$, there exists a scale $t \in [0,\infty]$ such that for all $x \in X$, $r_t(x) \leq \lambda r_s(x)$. 
    \item For all $s \in [0,\infty]$, the application $x \mapsto r_s(x)$ is $1$-Lipschitz in $X$.
\end{enumerate}
For the first one, it suffices to observe that
\begin{equation*}
    \tfrac{s}{1+s} \min \set{\mathrm{d}(x,X^c),1} \leq r_s(x) \leq s \min \set{\mathrm{d}(x,X^c),1}.
\end{equation*}
For the second one, we set $r = \abs{x - y}$ and we observe that $\mathrm{d}(y,X^c) \leq \mathrm{d}(x,X^c) + r$ so
\begin{align*}
    r_s(y)  &\leq \min\set{\tfrac{s}{1+s}\left(\mathrm{d}(x,X^c) + r\right),s}\\
            &\leq \min\set{\tfrac{s}{1+s}\mathrm{d}(x,X^c) + r,s+r}\\
            &\leq r_s(x) + r.
\end{align*}

Here are our two main results. We fix a closed subset $\Gamma$ of $X$ (the boundary) and a Borel measure $\II$ in $X$ (the area). We omit the assumptions on $\Gamma$ and $\II$ but they are specified in subsections \ref{sub_whitney} and \ref{sub_energy}. The complete statements are Theorem \ref{thm_limit} and Corollary \ref{cor_direct}. These statements involve the notion of \emph{sliding deformations} (Definition \ref{defi_sliding} and Remark \ref{rmk_global_sliding}).
\begin{thm*}[Limiting theorem]
    Fix $\kappa \geq 1$, $h > 0$ and a scale $s \in ]0,\infty]$. Assume that $h$ is small enough (depending on $n$, $\Gamma$, $\II$). Let $(E_i)$ be a sequence of closed, $\HH^d$ locally finite subsets of $X$ satisfying the following conditions:
    \begin{enumerate}[label=(\roman*)]
        \item the sequence of Radon measures $(\II \mres E_i)$ has a weak limit $\mu$ in $X$;
        \item for all open balls $B$ of scale $\leq s$ in $X$, there exists a sequence $(\varepsilon_i) \to 0$ such that for all global sliding deformations $f$ in $B$,
            \begin{equation}
                \II(E_i \cap W_f) \leq \kappa \II(f(E_i \cap W_f)) + h \II(E_i \cap hB) + \varepsilon_i
            \end{equation}
    \end{enumerate}
    Then we have
    \begin{equation}
        \II \mres E \leq \mu \leq \ka \II \mres E,
    \end{equation}
    where $E = \spt(\mu)$ and $\ka = \kappa + h$. The set $E$ is $(\kappa,\ka h,s)$-quasiminimal with respect to $\II$ in $X$ and in particular, $\HH^d$ rectifiable.
\end{thm*}
This theorem says that the weak limit of a quasiminimizing sequence is a quasiminimal set. It is analogous to the limiting theorem of David for sequences of quasiminimal sets which converge in local Hausdorff distance. A limit in local Hausdorff distance is a set by assumption, whereas we do not know the structure of $\mu$ a priori. David's theorem does not allow a small term $\varepsilon_i$ going to $0$ so it cannot be applied directly to minimizing sequences.

Our theorem also generalizes the strategy De Lellis, De Philippis, De Rosa, Ghiraldin and Maggi. It deals with more general sequences thanks to the factor $\kappa$ and the error term $h \II(E_i \cap hB)$. The convergence can take place in any open set $X$ of $\R^n$ and not only $X = \R^n \setminus \Gamma$.

Our proof does not rely on the Preiss rectifiability theorem or the extension of Allard's rectifiability theorem. It is similar to the proof of David and is based upon the construction of relevant sliding deformations.

We deduce the direct method developed by De Lellis, De Philippis, De Rosa, Ghiraldin and Maggi but we are also able to minimize the intersection of competitors with the boundary.
\begin{cor*}[Direct method]
    Fix $\mathcal{C}$ a class of closed subsets of $X$ such that
    \begin{equation*}
        m = \inf \set{\II(E) | E \in \mathcal{C}} < \infty
    \end{equation*}
    and assume that for all $E \in \mathcal{C}$, for all sliding deformations $f$ of $E$ in $X$,
    \begin{equation*}
        m \leq \II(f(E)).
    \end{equation*}
    Let $(E_k)$ be a minimizing sequence for $\II$ in $\mathcal{C}$. Up to a subsequence, there exists a coral\footnote{A set $E \subset X$ is coral in $X$ if $E$ is the support of $\HH^d \mres E$ in $X$. Equivalently, $E$ is closed in $X$ and for all $x \in E$ and for all $r > 0$, $\HH^d(E \cap B(x,r)) > 0$.} minimal set $E$ with respect to $\II$ in $X$ such that
    \begin{equation*}
        \II \mres E_k \rightharpoonup \II \mres E.
    \end{equation*}
    where the arrow $\rightharpoonup$ denotes the weak convergence of Radon measures in $X$. In particular, $\II(E) \leq m$.
\end{cor*}

\subsection{Main Definitions}
\subsubsection{Sliding Deformations and Quasiminimal Sets}
We fix a closed subset $\Gamma$ of $X$ (the boundary).
\begin{defi}[Sliding deformation along a boundary]\label{defi_sliding}
    Let $E$ be a closed, $\HH^d$ locally finite subset of $X$. A \emph{sliding deformation} of $E$ in an open set $U \subset X$ is a Lipschitz map $f\colon E \to X$ such that there exists a continuous homotopy $F\colon I \times E \to X$ satisfying the following conditions:
    \begin{subequations}
        \begin{align}
            &   F_0 = \mathrm{id}\\
            &   F_1 = f\\
            &   \forall t \in I,\ F_t(E \cap \Gamma) \subset \Gamma\\
            &   \forall t \in I,\ F_t(E \cap U) \subset U\\
            &   \forall t \in I,\ F_t = \mathrm{id} \ \text{in} \ E \setminus K,
        \end{align}
    \end{subequations}
    where $K$ is some compact subset of $E \cap U$. Alternatively, the last axiom can be stated as
    \begin{equation}
        \Set{x \in E | \exists \, t \in I, \ F_t(x) \ne x} \subset \subset E \cap U.
    \end{equation}
\end{defi}

Quasiminimal sets are sets for which sliding deformations cannot decrease the area below a fixed percentage $\kappa^{-1}$ (and modulo an error of small size). The topological constraint preventing the collapse might come from $U$ as $F_t(E \cap U) \subset U$ and $F_t = \mathrm{id}$ in $X \setminus U$ or from the boundary as $F_t(E \cap \Gamma) \subset \Gamma$.
\begin{defi}[Quasiminimal sets]\label{defi_quasi}
    Let $E$ be a closed, $\HH^d$ locally finite subset of $X$. Let $\mathcal{P}=(\kappa,h,s)$ be a triple of parameters composed of $\kappa \geq 1$, $h \geq 0$ and a scale $s \in ]0,\infty]$. We say that $E$ is \emph{$\mathcal{P}$-quasiminimal} in $X$ if for all open balls $B \subset X$ of scale $\leq s$ and for all sliding deformations $f$ of $E$ in $B$, we have
    \begin{equation}\label{quasi_ineq}
        \HH^d(W_f) \leq \kappa \HH^d(f(W_f)) + h \HH^d(E \cap hB),
    \end{equation}
    where $W_f = \set{x \in E | f(x) \ne x}$. We say that $E$ is almost-minimal in the case $\kappa = 1$. We say that $E$ is a minimal set in the case $\mathcal{P}=(1,0,\infty)$. 
\end{defi}
We can also replace the Hausdorff measure by an admissible energy $\II$ (Definition \ref{defi_energy}). The inequality (\ref{quasi_ineq}) is then replaced by
\begin{equation}
    \II(W_f) \leq \kappa \II(f(W_f)) + h \II(E \cap hB)
\end{equation}
and we say that $E$ is $(\kappa,h,s)$-quasiminimal with respect to $\II$ in $X$. Our definition of admissible energies says that there exist $\IL \geq 1$ such that $\IL^{-1} \HH^d \leq \II \leq \IL \HH^d$ so $E$ would be also $(\IL^2\kappa, \IL^2 h,s)$-quasiminimal in the usual sense.

The factor $\kappa$ includes Lipschitz graphs among quasiminimal sets. The error term $h \HH^d(E \cap hB)$ is a lower-order term which broadens the class of functionals that are to be minimized. One could also consider an error term of the form $h \mathrm{diam}(B)^d$. In practice, $h$ is assumed small enough depending on $n$, $\Gamma$ and $\II$. The quasiminimal sets for $\kappa > 1$ have weak regularity properties: the local Ahlfors-regularity (Proposition \ref{prop_density}) and the rectifiability (Proposition \ref{prop_rect}). The situation is different for almost-minimal sets because they look more and more like minimal sets at small scales. Their density is nearly monotone and their blow-ups are minimal cones.

By deforming a minimal set with a one-parameter family of diffeomorphisms, we get that the \emph{first variation of the area} is null. This means that the \emph{mean curvature} is null in the sense of varifolds. However, the previous definition allows to deform a set with a non-injective function and this imposes more constraints on the singularities. Let us take for example a cone of dimension $d=1$ centered at $0$. The first variation is null if and only if the center of mass of the cone is $0$. But the cone is a minimal set if and only if it is a line or three half-lines that make $120^\circ$ angles at the origin.

\begin{rmk}[Replacing $W_f$ by $B$]\label{rmk_B}
    Let us consider an open set $W$ such that $W_f \subset W$. Then, (\ref{quasi_ineq}) implies
    \begin{equation}
        \HH^d(E \cap W) \leq (\kappa+1) \HH^d(f(E \cap W)) + h \HH^d(E \cap hB).
    \end{equation}
    Taking $W = B$ and $h \leq \frac{1}{2}$, we obtain
    \begin{equation}\label{weak_ineq}
        \HH^d(E \cap B) \leq 2(\kappa+1) \HH^d(f(E \cap B)).
    \end{equation}
    This inequality is handier to use if one does not mind a bad scalar factor. For example, (\ref{weak_ineq}) is fine to prove the local Ahlfors-regularity and the rectifiability. But it would be too rough for certain properties of almost-minimal sets. Almost-minimal sets could be alternatively defined by
    \begin{equation}\label{weak_almost}
        \HH^d(E \cap B) \leq \HH^d(f(E \cap B)) + h \HH^d(E \cap hB).
    \end{equation}
    It is clear that (\ref{weak_almost}) implies
    \begin{equation}
        \HH^d(E \cap W_f) \leq \HH^d(f(E \cap W_f)) + h \HH^d(E \cap hB).
    \end{equation}
    and according to \cite[Proposition 4.10]{Holder} or \cite[Proposition 20.9]{Sliding}, these inequalities define in fact the same almost-minimal sets.
\end{rmk}
\begin{rmk}[Global sliding deformations]\label{rmk_global_sliding}
    In Definition \ref{defi_quasi}, we test the quasiminimality against sliding deformations defined \emph{only} on $E$. But we could also work with deformations defined everywhere in the ambient space $X$. For an open set $U \subset X$, a \emph{global sliding deformation in $U$} is a sliding deformation where $E$ has been replaced by $X$. In this case,
    \begin{equation}
        W_f = \set{x \in X | f(x) \ne x}.
    \end{equation}
    We think that sliding deformations defined only on $E$ are more natural than global sliding deformations. On the other hand, global sliding deformations are handier for our limiting theorem (Theorem \ref{thm_limit}). In Appendix \ref{appendix_toolbox},  we prove that if $\Gamma$ is a Lipschitz neighborhood retract (Definition \ref{retract_definition}), these notions induce the same quasiminimal sets.
\end{rmk}
\begin{rmk}\label{rmk_core}
    The \emph{core} of $E$ of $X$ is defined as the support of $\HH^d \mres E$ in $X$ and is denoted by $E^*$. It can be characterized as
    \begin{equation}
        E^* = \set{x \in X | \forall r > 0,\ \HH^d(E \cap B(x,r)) > 0}.
    \end{equation}
    By definition of the support, $E^*$ is closed and $\HH^d(E \setminus E^*) = 0$. As $E$ is also closed, we have furthermore $E^* \subset E$. Assuming that $\Gamma$ is a Lipschitz neighborhood retract, the set $E$ is $\mathcal{P}$-quasiminimal if and only if $E^*$ is $\mathcal{P}$-quasiminimal. It suffices to check the quasiminimality with respect to global sliding deformations; in this case $\HH^d(E \cap W_f) = \HH^d(E^* \cap W_f)$ and $\HH^d(f(E \cap W_f)) = \HH^d(f(E^* \cap W_f))$.
\end{rmk}

\subsubsection{The Boundary}\label{sub_whitney}
The set $\Gamma$ which plays the role of the boundary is first of all a closed subset of $X$. We present three additional properties. We cannot build much sliding deformations without assuming that the boundary is at least a Lipschitz neighborhood retract in $X$.
\begin{defi}\label{retract_definition}
    A \emph{Lipschitz neighborhood retract} of $X$ is a closed subset $\Gamma \subset X$ for which there exists an open set $O \subset X$ containing $\Gamma$ and a Lipschitz map $r\colon O \to \Gamma$ such that $r = \mathrm{id}$ on $\Gamma$.
\end{defi}
The construction of relevant deformations (Lemma \ref{technical_lemma}) will require $\Gamma$ to be locally diffeomorphic to a cone.
\begin{defi}\label{cone_definition}
    We say that a closed subset $\Gamma \subset X$ is locally diffeomorphic to a cone if for all $x \in \Gamma$, there exists $R > 0$ such that $\overline{B}(x,R) \subset X$, an open set $O \subset \R^n$ containing $x$, a $C^1$ diffeomorphism $\GT\colon B(x,R) \to O$ and a closed cone $S \subset \R^n$ centered at $x$ such that $\GT(x)=x$ and $\GT(\Gamma \cap B(x,R)) = S \cap O$.
\end{defi}
The local Ahlfors-regularity and the rectifiability of quasiminimal sets are proved with a special deformation called the Federer--Fleming projection. For this deformation to preserve $\Gamma$, we require that $\Gamma$ looks like a piece of grid in balls of uniform scales. For example, $\Gamma$ can be an union of faces of varying dimensions of a Whitney decomposition of $X$ or the image of such an union by a bilipschitz diffeomorphism. 

For $k \in \N$, we denote by $\mathcal{E}_n(k)$ the grid of dyadic cells of sidelength $2^{-k}$ in $\R^n$. This is the set of all faces of the form
\begin{equation}
    \prod_{i=1}^n [p_i,p_i + 2^{-k}\alpha_i],
\end{equation}
where $p \in 2^{-k} \Z^n$ and $\alpha \in \set{0,1}^n$. Given a set $S$ of cells, we define the support of $S$ as $\abs{S} = \bigcup \set{A | A \in S}$.
\begin{defi}[Whitney subset]\label{defi_whitney}
    Let $X$ be an open set of $\R^n$. A Whitney subset of $X$ is a closed subset $\Gamma \subset X$ such that
    \begin{enumerate}[label=(\roman*)]
        \item $\Gamma$ is a Lipschitz neighborhood retract in $X$;
        \item $\Gamma$ is locally diffeomorphic to a cone;
        \item there exists a constant $\GL \geq 1$ and a scale $t > 0$ such that for all $x \in \Gamma$, there exists an open set $O \subset \R^n$, a $\GL$-bilipschitz map $\GT\colon B(x,r_t) \to O$, an integer $k \in \N$ and a subset $S \subset \mathcal{E}_n(k)$ such that $2^{-k} \geq \GL^{-1} r_t$ and $\GT(\Gamma \cap B(x,r_t)) = O \cap \abs{S}$.
    \end{enumerate}
\end{defi}
The definition includes smooth compact submanifolds embedded in $X$. It is convenient to refine the third axiom in the following way.
\begin{enumerate}[label=(\roman**)]
    \setcounter{enumi}{2}
\item there exists a constant $\GL \geq 1$ and a scale $t > 0$ such that for all $x \in X$, for all $0 \leq r \leq r_t$, there exists an open set $O \subset \R^n$, a $\GL$-bilipschitz map $\GT\colon B(x,r) \to O$, an integer $k \in \N$ and a subset $S \subset \mathcal{E}_n(k)$ such that $\GT(\Gamma \cap B(x,r)) = O \cap \abs{S}$ and
    \begin{equation}\label{boundary}
        \GT(B(x,\GL^{-1} r)) \subset B(0,2^{-k-1}) \subset B(0,2^{-k} \sqrt{n}) \subset O.
    \end{equation}
\end{enumerate}
The point of the two balls $B(0,2^{-k-1})$ and $B(0,2^{-k} \sqrt{n})$ is to frame the cubes $[-2^{-k+1},2^{-k+1}]^n$ and $[-2^{-k}, 2^{-k}]^n$. We observe that $\GL^{-1} r \leq \GL 2^{-k-1}$ and $2^{-k}\sqrt{n} \leq \GL r$ so $r \simeq 2^{-k}$.
\begin{proof}
    Let $\GL \geq 1$ and $t > 0$ be the constants of (iii). We want first of all to extend this axiom to all points $x \in X$. Let $u > 0$ be a scale such that $r_u \leq \tfrac{1}{2}r_t$. Let $x \in X$. If $B(x,r_u) \cap \Gamma = \emptyset$, then (iii) holds trivially for the ball $B(x,r_u)$. Otherwise, there exists $x^* \in \Gamma$ such that $\abs{x - x^*} < r_u$ and then there exists an open set $O \subset \R^n$, a $\GL$-bilipschitz map $\GT\colon B(x^*,r_t) \to O$, an integer $k \in \N$ and a subset $S \subset \mathcal{E}_n(k)$ such that $2^{-k} \geq \GL^{-1} r_t$ and $\GT(\Gamma \cap B(x^*,r_t)) = O \cap \abs{S}$. By restriction, (iii) holds true for $(x,r_u)$ in place of $(x^*,r_t)$ (the image set $O$ is restricted accordingly).

    Let $y= \GT(x)$ and $0 \leq r < \GL^{-1} r_u$. We are going to justify that $B(y,\GL^{-1} r) \subset T(B(x,r))$. As $\GT^{-1}$ is $\GL$-Lipschitz in $O$, we have
    \begin{equation}
        \GT^{-1}(O \cap B(y,r)) \subset B(x,\GL r) \subset \overline{B}(x,\GL r) \subset B(x,r_u).
    \end{equation}
    and thus
    \begin{equation}
        O \cap B(y,r) \subset \GT(\overline{B}(x,\GL r)) \subset O.
    \end{equation}
    The set
    \begin{equation}
        O \cap B(y,r) = \GT(\overline{B}(x,\GL r)) \cap B(y,r).
    \end{equation}
    is open and closed in $B(y,r)$. We conclude $O \cap B(y,r) = B(y,r)$ by connectedness. This proves our claim.

    We are going to use the maximum norm $\abs{\, \cdot \,}_\infty$ and the corresponding open (cubic) balls $U$. For $l \in \mathbf{N}$, there exists $p \in 2^{-l-1} \Z^n$ such that $\abs{y - p}_{\infty} \leq 2^{-l-2}$. According to the triangular inequality
    \begin{equation}
        U(y,2^{-l-2}) \subset U(p,2^{-l-1}) \subset \overline{U}(p,2^{-l}) \subset U(y,2^{-l+1}).
    \end{equation}
    We replace the extremities by Euclidean balls
    \begin{equation}\label{range}
        B(y,2^{-l-2}) \subset U(p,2^{-l-1}) \subset \overline{U}(p,2^{-l}) \subset B(y,2^{-l+1} \sqrt{n}).
    \end{equation}
    We want to choose the smallest $l$ such that $B(y,2^{-l+1} \sqrt{n}) \subset \GT(B(x,r))$. We take $l \in \Z$ such that
    \begin{equation}
        (4 \sqrt{n} \GL)^{-1} r \leq 2^{-l} \leq (2 \sqrt{n} \GL)^{-1} r.
    \end{equation}
    Note that $l \geq k$ because $2^{-k} \geq \GL^{-1} r > 2^{-l}$. We have
    \begin{equation}
        B(y,2^{-l+1}\sqrt{n})   \subset B(y,\GL^{-1}r)                  \subset \GT(B(x,r))
    \end{equation}
    and for $\GL_* = 16 \sqrt{n} \GL^2$,
    \begin{equation}
        \GT(B(x,\GL_*^{-1} r)) \subset B(y, (16 \sqrt{n} \GL)^{-1} r) \subset B(y,2^{-l-2})
    \end{equation}
    We postcompose $\GT$ with a translation to assume $p = 0$. This proves (\ref{boundary}). As $l \geq k$, the grid $\mathcal{E}_n(l)$ is a refinement of $\mathcal{E}_n(k)$ and there exists $S' \subset \mathcal{E}_n(l)$ such that $\abs{S} = \abs{S'}$.
\end{proof}

\subsubsection{The Area}\label{sub_energy}
The next definition comes from \cite[Definition 25.3, Remark 25.87]{Sliding}. It is a slight generalization of \cite[Definition 1.6(2)]{A1} and \cite[Definition IV.1(7)]{A2}. After the statement, we will give a few explanations and compare it to other definitions.
\begin{defi}[Admissible energy]\label{defi_energy}
    An \emph{admissible energy} in $X$ is a Borel regular measure $\II$ in $X$ which satisfies the following axioms:
    \begin{enumerate}[label=(\roman*)]
        \item There exists $\IL \geq 1$ such that $\IL^{-1} \HH^d \leq \II \leq \IL \HH^d$.
        \item Let $x \in X$, let a $d$-plane $V$ passing through $x$, let a $C^1$ map $f\colon V \to \mathbf{R}^n$ be such that $f(x) = x$ and $Df(x)$ is the inclusion map $\overrightarrow{V} \hookrightarrow \R^n$. Then
            \begin{equation}
                \lim_{r \to 0} \frac{\II(f(V \cap B(x,r)))}{\II(V \cap B(x,r))} = 1.
            \end{equation}
        \item For each $x \in X$, there exists $R > 0$ and $\varepsilon \colon \mathopen{]}0,R\mathclose{]} \to \R^+$ such that $\overline{B}(x,R) \subset X$, $\lim_{r \to 0} \varepsilon(r) = 0$ and
            \begin{equation}\label{ellipticity_david}
                \II(V \cap B(x,r)) \leq \II(S \cap B(x,r)) + \varepsilon(r)r^d
            \end{equation}
            whenever $0 < r \leq R$, $V$ is a $d$-plane passing through $x$ and $S \subset \overline{B}(x,r)$ is a compact $\HH^d$ finite set which cannot be mapped into $V \cap \partial B(x,r)$ by a Lipschitz mapping $\psi\colon \overline{B}(x,r) \to \overline{B}(x,r)$ such that $\psi = \mathrm{id}$ on $V \cap \partial B(x,r)$.
    \end{enumerate}
\end{defi}
We use the second axiom to say that if $E \subset X$ is $\HH^d$ measurable, $\HH^d$ locally finite and $\HH^d$ rectifiable, then for $\HH^d$-a.e. $x \in E$,
\begin{equation}
    \lim_{r \to 0} \frac{\II(E \cap B(x,r))}{\II(V \cap B(x,r))} = 1.
\end{equation}
A proof is given at the end of this subsection. The third axiom is important to establish the lower semicontinuity. Let us say that $(E_k)$ is a minimizing sequence of competitors such that $\II \mres E_k \to \mu$ where $\mu$ is a $d$-rectifiable Radon measure. The third axiom is the main argument to show that
\begin{equation}
    \II \mres E_\infty \leq \mu,
\end{equation}
where $E_\infty = \spt(\mu)$. This yields in particular $\II(E_\infty) \leq \lim_k \II(E_k)$. The reader can find an example below Definition 25.3 in \cite{Sliding} where $\II(E_\infty)$ is too large.

The second axiom is perhaps too weak for a certain step in the proof of the limiting theorem (Theorem \ref{thm_limit}). We suggest a stronger variant below. That being said, we do not need this variant for the direct method (Corollary \ref{cor_direct}) because minimizing sequences have additional properties. See also Remark \ref{rmk_limit}
\begin{defi}\label{defi_energy2}
    We say that an admissible energy $\II$ in $X$ is \emph{induced by a continuous integrand} if there exists a continuous function $i\colon X \times G(d,n) \to \mathopen{]}0,\infty\mathclose{[}$ such that for all $\HH^d$ measurable, $\HH^d$ finite, $\HH^d$ rectifiable set $E \subset X$,
    \begin{equation}
        \II(E) = \int_E \! i(x,T_xE) \, \mathrm{d}\HH^d(x)
    \end{equation}
    where $T_xE$ is the linear tangent plane of $E$ at $x$. 
\end{defi}

Here is an example. We consider two Borel functions,
\begin{equation}
    i\colon X \times G(d,n) \to \mathopen{]}0,\infty\mathclose{[} \ \text{and} \ j\colon X \to \mathopen{]}0,\infty\mathclose{[}
\end{equation}
called the \emph{integrands} and we define the corresponding energy by the formula
\begin{equation}
    \II(S) = \int_{S_r} i(y,T_yS) \, \mathrm{d}\HH^d(y) + \int_{S_u} j(y) \, \mathrm{d}\HH^d(y)
\end{equation}
where $S \subset X$ is a Borel $\HH^d$ finite set and $S_r$, $S_u$ are its $d$-rectifiable and purely $d$-unrectifiable parts. We assume $\IL^{-1} \leq i,j \leq \IL$ so that $\IL^{-1} \HH^d \leq \II \leq \IL \HH^d$ on such sets $S$. We define of course $\II(S) = \infty$ on Borel sets $S \subset X$ of infinite $\HH^d$ measure and we extend $\II$ on all subsets of $X$ by Borel regularity. Next, we assume that $i$ is continuous on $X \times G(d,n)$ and we show that this implies the second axiom. Let $x$, $V$ and $f$ be as in (ii). To shorten a bit the notations, we assume $x = 0$ and we denote $i(0,V)$ by $i_0$. For $r > 0$, we denote $B(0,r)$ by $B_r$ and $f(V \cap B(0,r))$ by $S_r$. Fix $\varepsilon > 0$. According to the inverse function theorem, there exists $R > 0$ such that $f$ induces a $(1+\varepsilon)$-bilipschitz diffeomorphism from $V \cap B_R$ to $S = S_R$. Note that the function $y \mapsto (y,T_y S)$ is continuous on $S$ and that at $y = 0$, $(y,T_y S) = (0,V)$. As $i$ is continuous, there exists $R' > 0$ such that for all $y \in S \cap B_{R'}$
\begin{equation}
    (1 + \varepsilon)^{-1} i_0 \leq i(y,T_y S) \leq (1 + \varepsilon) i_0.
\end{equation}
If $r > 0$ is small enough so that $B_r \subset B_R \cap f^{-1}(B_{R'})$, we have $S_r \subset S \cap B_{R'}$ and thus
\begin{equation}
    (1 + \varepsilon)^{-1} i_0 \leq \frac{\II(S_r)}{\HH^d(S_r)} \leq (1 + \varepsilon) i_0.
\end{equation}
The function $f$ is $(1+\varepsilon)$-bilipschitz on $V \cap B_r$ so this simplifies to 
\begin{equation}
    (1 + \varepsilon)^{d-1} i_0 \leq \frac{\II(S_r)}{\HH^d(V \cap B_r)} \leq (1 + \varepsilon)^{d+1} i_0.
\end{equation}
We deduce that
\begin{equation}
    \lim_{r \to 0} \frac{\II(S_r)}{\HH^d(V \cap B_r)} = i_0
\end{equation}
and by the same reasoning,
\begin{equation}
    \lim_{r \to 0} \frac{\II(V \cap B_r)}{\HH^d(V \cap B_r)} = i_0.
\end{equation}
The second axiom follows. It is more difficult to construct integrands that yield the third axiom. A first example is when $\II = \HH^d$ or when $i$ depends on $x$ alone with $i = j$ (we will detail this later). Let us try to relate our integrands to other definitions. The energy introduced by Almgren in \cite[Definition 1.6(2)]{A1} is like above with $X = \R^n$ and $i$ of class $C^3$ but (iii) is replaced by a stronger condition called \emph{ellipticity bound}. We present this condition in more detail. For $x \in \R^n$, let
\begin{equation}
    \II^x(S) = \int_{S_r} i(x,T_yS) \, \mathrm{d}\HH^d(y) + j(x) \HH^d(S_u)
\end{equation}
be defined on Borel $\HH^d$ finite sets $S \subset \R^n$. Almgren requires that there exists a continuous function $c\colon \R^n \to ]0,\infty[$ such that for all $x \in \R^n$,
\begin{multline}\label{ellipticity_almgren}
    \II^x(S \cap B(x,r)) - \II^x(V \cap B(x,r)) \\\geq c(x) \left[\HH^d(S \cap B(x,r)) - \HH^d(V \cap B(x,r))\right]
\end{multline}
whenever $r > 0$, $V$ is a $d$-plane passing through $x$ and $S \subset \overline{B}(x,r)$ is a compact, $\HH^d$ rectifiable, $\HH^d$ finite set which spans $V \cap \partial B(x,r)$. In \cite[Definition IV.1(7)]{A2}, Almgren restricts this definition to the functions $c$ which are positive constants. The energy introduced by De Lellis, De Philippis, De Rosa, Ghiraldin and Maggi in \cite[Definitions 1.3, 1.5 and Remark 1.8]{I5} is also like above with $X = \R^n$ and $i$ of class $C^1$ but (iii) is replaced by a new axiom called \emph{atomic condition}. This axiom characterizes the integrands for which Allard's rectifiability theorem holds (\cite{I3}). In codimension one, it is equivalent to a convexity condition on the integrand. In the general case, \cite{DRK} shows that the atomic condition implies an ellipticity condition of the form
\begin{multline}\label{ellipticity_weak}
    \left[\HH^d(S \cap B(x,r)) - \HH^d(V \cap B(x,r))\right] > 0 \\\implies \left[\II^x(S \cap B(x,r)) - \II^x(V \cap B(x,r))\right] > 0
\end{multline}
whenever $r > 0$, $V$ is a $d$-plane passing through $x$ and $S \subset \overline{B}(x,r)$ is a compact, $\HH^d$ finite set which spans $V \cap \partial B(x,r)$.

In this paragraph, we compare (iii) to Almgren’s ellipticity bound. We allow $d$-dimensional sets $S$ which have a purely unrectifiable part because we allow competitors which have a purely unrectifiable part in the direct method (Corollary \ref{cor_direct}). Almgren only tests the ellipticity bound against rectifiable sets but \cite[Corollary 5.13]{DRK} justifies that this is not a loss of generality. Next, we show that an inequality such as (\ref{ellipticity_almgren}) implies (\ref{ellipticity_david}). We work at the point $x = 0$ and we fix a constant $c > 0$. Let $r > 0$ (to be chosen small enough), let $V$ be a $d$-plane passing through $0$, let $S \subset \overline{B}(0,r)$ be a compact $\HH^d$ finite set such that $V \cap \overline{B}(0,r) \subset p_V(S)$ and let us assume that
\begin{multline}\label{S_assumptions}
    \II^0(S \cap B(0,r)) - \II^0(V \cap B(0,r)) \\\geq c \left[\HH^d(S \cap B(0,r)) - \HH^d(V \cap B(0,r))\right].
\end{multline}
Note that $\HH^d(V \cap B(0,r)) \leq \HH^d(S \cap B(0,r))$ because $V \cap \overline{B}(0,r) \subset p_V(S)$ and $p_V$ is $1$-Lipschitz. We define
\begin{multline}
    \varepsilon(r) = \sup\Set{\abs{i(y,W) - i (0,W)} | y \in B(0,r),\ W \in G(d,n)}\\+ \sup\Set{\abs{j(y) - j(0)} | y \in B(0,r)}.
\end{multline}
The function $i$ is uniformly continuous on the compact set $\overline{B}(0,1) \times G(d,n)$ and $j$ as well on $\overline{B}(0,1)$ so $\lim_{r \to 0} \varepsilon(r) = 0$. We assume $r$ small enough so that $\varepsilon(r) \leq c$. According to the definition of $\varepsilon(r)$,
\begin{equation}
    \abs{\II(S \cap B(0,r)) - \II^0(S \cap B(0,r))} \leq \varepsilon(r) \HH^d(S \cap B(0,r))
\end{equation}
and similarly
\begin{equation}
    \abs{\II(V \cap B(0,r)) - \II^0(V \cap B(0,r))} \leq \varepsilon(r) \HH^d(V \cap B(0,r)).
\end{equation}
Plugging this into (\ref{S_assumptions}), we get
\begin{align}
    \begin{split}
        &\II(S \cap B(0,r)) - \II(V \cap B(0,r))\\
        &\geq c \left[\HH^d(S \cap B(0,r)) - \HH^d(V \cap B(0,r))\right]\\
        &\qquad - \varepsilon(r) \HH^d(S \cap B(0,r)) - \varepsilon(r) \HH^d(V \cap B(0,r))
    \end{split}\\
    \begin{split}
        &\geq (c - \varepsilon(r))\left[\HH^d(S \cap B(0,r)) - \HH^d(V \cap B(0,r))\right]\\
        &\qquad - 2 \varepsilon(r) \HH^d(V \cap B(0,r))
    \end{split}\\
        &\geq - 2 \varepsilon(r) \HH^d(V \cap B(0,r))\\
        &\geq - 2 \varepsilon(r) \omega_d r^d
\end{align}
where $\omega_d$ is the $\HH^d$ measure of the $d$-dimensional unit disk. This is (\ref{ellipticity_david}) as promised. We should not forget to mention that if $i$ depends on $x$ alone with $i = j$, then we have $\II^0 = i(0) \HH^d$ so (\ref{S_assumptions}) is clearly verified.

\begin{lem}
    Let $\II$ be an admissible energy in $X$. Let $E \subset X$ be $\HH^d$ measurable, $\HH^d$ locally finite and $\HH^d$ rectifiable. Then for $H^d$-ae. $x \in E$, 
    \begin{equation}
        \lim_{r \to 0} \frac{\II(E \cap B(x,r))}{\II(V \cap B(x,r))} = 1
    \end{equation}
\end{lem}
\begin{proof}
    Let $M$ be an embedded submanifold of $\R^n$. We show that for all $x \in M \cap X$,
    \begin{equation}\label{density_manifold}
        \lim_{r \to 0} \frac{\II(M \cap B(x,r))}{\II(V \cap B(x,r))} = 1.
    \end{equation}
    Fix $x \in M \cap X$ and let $V$ be the affine tangent plane of $M$ at $x$. There exists a radius $R > 0$, an open set $O$ containing $x$, a $C^1$ diffeomorphism $f \colon B(x,R) \to O$ such that $f(x) = x$, $Df(x) = \mathrm{id}$ and $f(V \cap B(x,R)) = M \cap O$. Let us fix $\lambda > 1$. We observe that for small $r > 0$, we have
    \begin{equation}
        f(V \cap B(x,\lambda^{-1}r) \subset M \cap B(x,r) \subset f(V \cap B(x,\lambda r)).
    \end{equation}
    We also observe that there exists a constant $C \geq 1$ (depending on $n$, $\IL$) such that for small $r > 0$,
    \begin{align}
        \abs{\frac{\II(V \cap B(x,\lambda r))}{\II(V \cap B(x,r))} - 1}   &\leq \frac{\II(V \cap B(x,\lambda r) \setminus B(x,r))}{\II(V \cap B(x,r))}\\
                                                                          &\leq C \frac{\HH^d(V \cap B(x,\lambda r) \setminus B(x,r))}{\HH^d(V \cap B(x,r))}\\
                                                                        &\leq C (\lambda^d - 1).
    \end{align}
    We combine these observations to deduce that for small $r > 0$,
    \begin{equation}
        \frac{\II(M \cap B(x,r))}{\II(V \cap B(x,r))} \leq C(\lambda) \frac{\II(f(V \cap B(x,\lambda r))}{\II(V \cap B(x,\lambda r))}
    \end{equation}
    where $C(\lambda) \to 1$ when $\lambda \to 1$. It follows that
    \begin{equation}
        \limsup_{r \to 0} \frac{\II(M \cap B(x,r)}{\II(V \cap B(x,r)} \leq C(\lambda)
    \end{equation}
    and since $\lambda > 1$ is arbitrary, we get
    \begin{equation}
        \limsup_{r \to 0} \frac{\II(M \cap B(x,r))}{\II(V \cap B(x,r))} \leq 1.
    \end{equation}
    We can work similarly with $\liminf_{r \to 0}$.

    Now, we deal with $E$. It suffices to show that for $\HH^d$-a.e. $x \in E$, there exists an embedded submanifold $M \subset \R^n$ such that
    \begin{equation}
        \lim_{r \to 0} r^{-d} \HH^d(B(x,r) \cap (E \Delta M)) = 0.
    \end{equation}
    According to \cite[Theorem 15.21]{Mattila} or \cite[Theorem 3.2.29]{Federer}, there exists a sequence of embedded submanifolds $(M_i) \subset \R^n$ such that
    \begin{equation}
        \HH^d(E \setminus \bigcup_i M_i) = 0.
    \end{equation}
    We also recall \cite[Theorem 6.2(2)]{Mattila}. If $A \subset \R^n$ is $\HH^d$ measurable and $\HH^d$ locally finite, then for $\HH^d$-a.e. $x \in \R^n \setminus A$, $\lim_{r \to 0} r^{-d} \HH^d(B(x,r) \cap A) = 0$.
    Let us fix an index $i$. We have for $\HH^d$-a.e. $x \in E$,
    \begin{equation}
        \lim_{r \to 0} r^{-d} \HH^d(B(x,r) \cap M_i \setminus E) = 0
    \end{equation}
    and for $\HH^d$-a.e. $x \in M_i$,
    \begin{equation}
        \lim_{r \to 0} r^{-d} \HH^d(B(x,r) \cap E \setminus M_i) = 0
    \end{equation}
    whence for $\HH^d$-a.e. $x \in E \cap M_i$,
    \begin{equation}
        \lim_{r \to 0} r^{-d} \HH^d(B(x,r) \cap (E \Delta M_i)) = 0.
    \end{equation}
    Our claim follows.
\end{proof}

\section{Federer--Fleming Projection}\label{rigid_boundaries}
\subsection{Complexes}
The Federer--Fleming projection for sets was introduced in \cite{DS} by David and Semmes following the ideas of the Federer--Fleming projection for currents. This technique sends a $d$-dimensional set $E$ in the $d$-dimensional skeleton of a grid of cubes while controlling the measure of the image. In each cube, we choose a center of projection which is not in the closure of $E$. We perform a radial projection to send $E$ in the $(n-1$)-dimensional faces. Then we repeat the process in each $(n-1)$-dimensional face to project $E$ in the $(n-2)$-dimensional faces, etc. One usually stops when $E$ is sent in the $d$-skeleton because it is no longer possible to ensure that there exists a center of projection away from $E$. A simple idea to estimate the measure of the image is to estimate the Lipschitz constants of the radial projections. However, we may not have a good control over this Lipschitz constant. David and Semmes proved that in average among centers of projection, the measure of the image is multiplied by a constant that depends only on $n$. 

The goal of this subsection is to axiomatize the structure which supports Federer--Fleming projections. We consider a set $K$ made of faces of varying dimensions in which we intend to perform radial projections. In order to glue and compose radial projections, the inclusion relation $\subset$ between the faces of $K$ should be compatible with the topology in some sense. We adopt a formalism close the one of $CW$-complex , except that the boundary of a cell may not be covered by other subcells. In Fig. \ref{grille} for example, we choose not to make radial projection on the external edges so they are excluded from $K$.
\begin{figure}[ht]
    \begin{center}
        \setlength{\fboxsep}{0pt}
        \fbox{\includegraphics[width=.8\linewidth]{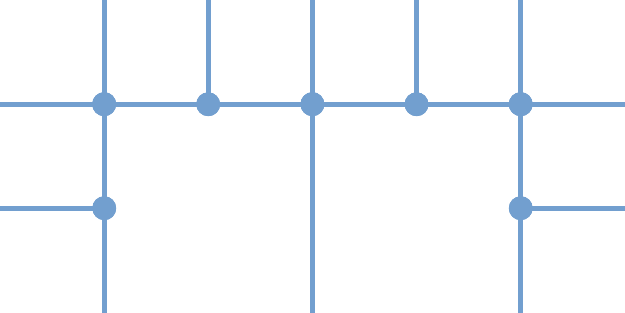}}
        \caption{A set of cells composed of $2$-cells (white squares), $1$-cells (blue edges) and $0$-cells (blue vertices). The associated Federer--Fleming projection consists in making a radial projection in each white square and then in each blue edge (we take the convention that a radial projection in a $0$-cell is the identity map).}
        \label{grille}
    \end{center}
\end{figure}

We define a $\emph{cell}$ as a face of cubes of $\R^n$. We could take a more general definition but it is easier to build complexes that way. The \emph{dimension} of a cell $A$ is the dimension of its affine span. Its \emph{interior} $\mathrm{int}(A)$ and \emph{boundary} $\partial A$ are the topological interior and topological boundary relative to its affine span. For example, if $A$ is $0$-dimensional, then $\mathrm{int}(A) = A$ and $\partial A = \emptyset$. If $A$ is a segment $[x,y]$ with $x \ne y$, then $\mathrm{int}(A) = ]x,y[$ and $\partial A = \set{x,y}$. Given a set of cells $K$, we define the \emph{support} of $K$ by
\begin{equation}
    \abs{K} = \bigcup \set{A | A \in K}.
\end{equation}
For an integer $i = 0, \ldots, n$, the subset of $K$ composed of the $i$-dimensional cells is denoted by $K^i$.

We are going to motive the axioms. Say that $K$ is a set of cells in $\R^n$ and we want to glue radial projections in its $n$-cells, then in its $(n-1)$-cells, etc. It is natural to ask that the cells have disjoint interiors because we want the gluing of our radial projections to be well defined. It is more problematic to ensure that the gluing is continuous. At each step, we have to glue a family of continuous maps (radial projections)
\begin{equation}
    \phi_A\colon A \to A
\end{equation}
indexed by $\set{A \in K^d}$ (for some $d = 0, \ldots,n$), with the identity map on
\begin{equation}
    \abs{K} \setminus \bigcup \set{\mathrm{int}(A) | A \in K,\ \mathrm{dim} \, A \geq d}.
\end{equation}
We use of a classical result which says that a pasting of continuous maps on two closed domains is a continuous map. This also works for a locally finite family of closed domains\footnote{In a topological space $X$, a family of sets $(A_i)$ is locally finite provided that for every $x \in X$, there exists a neighborhood $U$ of $x$ such that $\set{i | A_i \cap U \ne \emptyset}$ is finite. As a consequence, if the sets $A_i$ are closed, their union $\bigcup A_i$ is closed in $X$.}. We require the cells to constitute a locally finite family in $\abs{K}$ so that we can glue continuously the maps $(\phi_A)$. Next, we want the set
\begin{equation}
    \abs{K} \setminus \bigcup \set{\mathrm{int}(A) | A \in K,\ \mathrm{dim} \, A \geq d}
\end{equation}
to be closed. The corresponding axiom is more precise: we require that for every $A \in K$, the set $\bigcup \set{\mathrm{int}(B) | B \in K,\ A \subset B}$ is a neighborhood of $\mathrm{int}(A)$.

\begin{defi}[Complex]\label{def_complex}
    A \emph{complex} $K$ is a set of cells such that
    \begin{enumerate}[label=(\roman*)]
        \item the cells interior $\set{\mathrm{int}(A) | A \in K}$ are mutually disjoint;
        \item every $x \in \abs{K}$ admits a relative neighborhood in $\abs{K}$ which meets a finite number of cells $A \in K$;
        \item for every cell $A \in K$, the set
            \begin{equation}
                V_A = \bigcup \set{\mathrm{int}(B) | B \in K \ \text{contains} \ A}
            \end{equation}
            is a relative neighborhood of $\mathrm{int}(A)$ in $\abs{K}$.
    \end{enumerate}
\end{defi}
These are expected properties of CW-complexes but for us the boundary of a cell may not be covered by other cells. These axioms are enough for building a continuous Federer--Fleming projection but for a Lipschitz projection, we also need to quantify how wide is $V_A$ around $\mathrm{int}(A)$. In the next section, we will work a more specific class of complexes that we call $n$-complexes.

We present a few basic properties of such a complex $K$. By the local finiteness axiom, $K$ is at most countable. For each $A \in K$, the set $V_A$ is relatively open in $\abs{K}$. Indeed for each cell $B \in K$ containing $A$, $V_B \subset V_A$ so $V_A$ is also a neighborhood of $\mathrm{int}(B)$. Finally, we are going to see that a non-empty intersection of the form $\mathrm{int}(A) \cap B$ is meaningful.
\begin{lem}\label{complex_lemma}
    Let $K$ be a complex.
    \begin{itemize}
        \item For $A, B \in K$, $\mathrm{int}(A) \cap B \ne \emptyset$ implies $A \subset B$.
        \item For $A, B \in K$ such that $\mathrm{dim} \, A > \mathrm{dim} \, B$, we have $\mathrm{int}(A) \cap B = \emptyset$.
        \item For $A, B \in K$ of the same dimension, $\mathrm{int}(A) \cap B \ne \emptyset$ implies $A = B$.
    \end{itemize}
\end{lem}
\begin{proof}
    Let us prove the first point. The set $V_A$ is a relative open set of $\abs{K}$ which meets $B = \overline{\mathrm{int}(B)}$ so it also meets $\mathrm{int}(B)$. As the cells of $K$ have disjoint interiors, we deduce that $B$ contains $A$. The second point is now obvious. Next, we assume that $A$ and $B$ have the same dimension and we prove that $A = B$. According to the inclusion $A \subset B$ and a dimension argument, the affines spaces $\mathrm{aff}(A)$ and $\mathrm{aff}(B)$ are equals. As $\mathrm{int}(A)$ is relatively open in $\mathrm{aff}(B)$, we deduce $\mathrm{int}(A) \subset \mathrm{int}(B)$. Then $A = B$ because the cells have disjoint interiors.
\end{proof}


We introduce a natural notion subspace.
\begin{defi}[Subcomplex]\label{subcomplex_definition}
    Let $K$ be a complex. A \emph{subcomplex} of $K$ is a subset $L \subset K$ such that for all $A \in L$ and $B \in K$, the inclusion $A \subset B$ implies $B \in L$. The \emph{rigid open set} induced by $L$ is the set
    \begin{equation}
        U(L) = \bigcup \Set{\mathrm{int}(A) | A \in L}.
    \end{equation}
\end{defi}
If $L$ is a subcomplex of $K$, then $L$ is a also a complex. Moreover, $U(L)$ is relatively is relatively open in $\abs{K}$ by the second axiom of Definition \ref{def_complex}. We introduce a natural relationship between complexes.
\begin{defi}[Subordinate complex]
    We say that a complex $L$ is \emph{subordinate} to a complex $K$ when for every $A \in L$, there exists $B \in K$ (necessary unique) such that
    \begin{equation}
        \mathrm{int}(A) \subset \mathrm{int}(B).
    \end{equation}
    This relation is denoted by $L \preceq K$.
\end{defi}
See Fig. \ref{figure_subordinate}. The notion of subordinate complex differs from the notion of subcomplex because $L$ may not be subset of $K$.
\begin{figure}[ht!]
    \begin{center}
        \setlength{\fboxsep}{0pt}
        \fbox{\includegraphics[width=.8\linewidth]{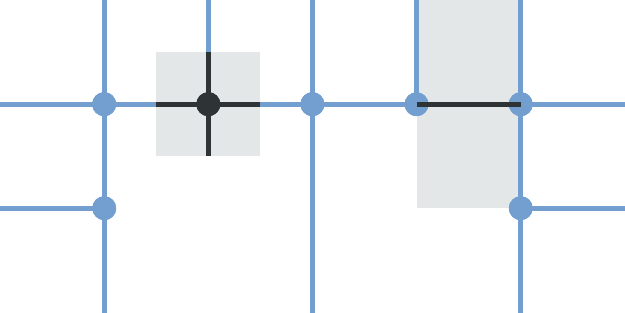}}
        \caption{A complex (white squares, blue edges, blue vertices) and two subordinate complexes (gray squares, black edges, black vertices).}\label{figure_subordinate}
    \end{center}
\end{figure}

\subsection{Pasting Complexes} 
We are going to present a way of building complexes by combining pieces of grids. The result of such a procedure is called a $n$-complex. As an example, we will see that the Whitney decomposition of an open set is a $n$-complex. 
\begin{defi}[Canonical $n$-complex]
    The \emph{canonical $n$-complex} of $\R^n$ is
    \begin{equation}
        E_n = \Set{\prod_{i = 1}^n [0,\alpha_i] | \alpha \in \set{-1,0,1}^n}.
    \end{equation}
    See Fig. \ref{figure_En}. We call \emph{$n$-charts} the image of $E_n$ by a similarity of $\R^n$ (translations, isometries, homothetys and their compositions). 
\end{defi}
\begin{figure}[ht!]
    \begin{center}
        \setlength{\fboxsep}{0pt}
        \fbox{\includegraphics[width=.35\linewidth]{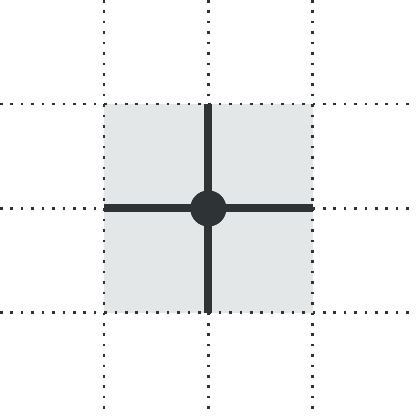}}
        \caption{The canonical grid of the plane is represented in dotted lines. The complex $E_2$ is made of the four gray squares, the four black edges and the black vertice.}\label{figure_En}
    \end{center}
\end{figure}
We justify that $E_n$ is a complex. We are going to see that for all $A \in E_n$, the set $V_A$ is open in $\R^n$ (and not only relatively open in $\abs{E_n}$). Moreover, if $A$ is not $0$-dimensional, there exists a constant $\kappa \geq 1$ (depending only on $n$) such that the set
\begin{equation}
    V_A(\kappa) = \Set{x \in \R^n | \mathrm{d}(x, A) < \kappa^{-1} \mathrm{d}(x, \partial A)}
\end{equation}
is included in $V_A$. The set $V_A(\kappa)$ is an open neighborhood of $\mathrm{int}(A)$ and the parameter $\kappa$ quantifies how wide it is.
\begin{proof}
    Consider $\alpha \in \set{-1,0,1}^n$ and the corresponding cell $A = \prod [0,\alpha_i]$. The interior of $A$ is
    \begin{equation}
        \mathrm{int}(A) = \prod_{i=1}^n (0,\alpha_i),
    \end{equation}
    where $(0,\alpha_i) = ]0,\alpha_i[$ if $\alpha_i \ne 0$ and $(0,\alpha_i) = \set{0}$ otherwise. Observe that different $\alpha, \beta$ induce cells $\prod [0,\alpha]$, $\prod [0,\beta]$ which have disjoint interiors. Thus, the cells of $E_n$ have disjoint interiors. Next, we are going to see that for all $A \in E_n$, the set $V_A$ is open in $\R^n$. Without loss of generality, we can work with $A = [0,1]^d \times \set{0}^{n-d}$ for some $d = 0,\ldots,n$ and then we see that
    \begin{equation}\label{V_int}
        V_A = \mathopen{]}0,1\mathclose{[}^d \times \mathopen{]}-1,1\mathclose{[}^{n-d}.
    \end{equation}
    Now, we assume $d \geq 1$ and we show that there exists $\kappa \geq 1$ (depending only on $n$) such that
    \begin{equation}
        \Set{x \in \R^n | \mathrm{d}(x, A) < \kappa^{-1} \mathrm{d}(x, \partial A)} \subset V_A.
    \end{equation}
    We proceed by contraposition. We fix $x \in \R^n \setminus V_A$. We have either $x_i \notin ]0,1[$ for some $i=1,\ldots,d$ or $x_i \notin ]-1,1[$ for some $i = d+1,\ldots,n$. In the first case, the distance $\mathrm{d}(x,A)$ is attained on $\partial A$ so $\mathrm{d}(x,A) = \mathrm{d}(x,\partial A)$. In the second case we have $\mathrm{d}(x,A) \geq 1$. The triangular inequality yields
    \begin{align}
        \mathrm{d}(x,\partial A)    &   \leq \mathrm{d}(x,A) + \mathrm{diam}(A)\\
                                    &   \leq \mathrm{d}(x,A) + \sqrt{d},
    \end{align}
    so
    \begin{equation}
        \mathrm{d}(x,\partial A) \leq (1 + \sqrt{d}) \mathrm{d}(x,A).
    \end{equation}
    In both cases, one has $\mathrm{d}(x,\partial A) \leq \kappa \mathrm{d}(x,A)$, where $\kappa$ depends only on $n$. Finally, it is clear that $E_n$ is locally finite because it has $3^n$ cells.
\end{proof}
We detail how to combine a family of $n$-charts $(K_\alpha)$. This principle is close to the notion of \emph{direct limit} in algebra (see the remark below Definition \ref{defi_system}). We paste the family by "identifying" the cells $A$,$B$ for which there exists another cell $C$ such that $\mathrm{int}(A), \mathrm{int}(B) \subset \mathrm{int}(C)$. We need a few assumptions to ensure that each equivalent class has a maximal cell and that the collection of maximal cells is locally finite. The result of such a procedure is called a $n$-complex. 
\begin{defi}\label{defi_system}
    A \emph{system of $n$-charts} is a family of $n$-charts $(K_\alpha)$ such that
    \begin{enumerate}[label=(\roman*)]
        \item for all cells $A, B \in \bigcup K_\alpha$
            \begin{equation}
                \mathrm{int}(A) \cap \mathrm{int}(B) \ne \emptyset \implies \mathrm{int}(A) \subset \mathrm{int}(B) \ \text{or}\ \mathrm{int}(B) \subset \mathrm{int}(A);
            \end{equation}
        \item for every $x \in \bigcup \abs{K_\alpha}$, there exists a relative neighborhood $V$ of $x$ in $\bigcup \abs{K_\alpha}$ and a finite set $S \subset \bigcup K_\alpha$ such that whenever a cell $A \in \bigcup K_\alpha$ meets $V$, there exists $B \in S$ such that $\mathrm{int}(A) \subset \mathrm{int}(B)$.
    \end{enumerate}
    We define a \emph{maximal cell of the system} as a cell $A \in \bigcup K_\alpha$ such that for all $B \in \bigcup K_\alpha$,
    \begin{equation}
        \mathrm{int}(A) \cap \mathrm{int}(B) \ne \emptyset \implies \mathrm{int}(B) \subset \mathrm{int}(A).
    \end{equation}
    The collection of all maximal cells is called the \emph{limit of the system $(K_\alpha)$}. It is a complex $K$ such that $\abs{K} = \bigcup \abs{K_\alpha}$. Moreover, it has the following intrinsic properties.
    \begin{enumerate}[label=(\roman*)]
        \item For all $A \in K$, the set $V_A$ is open in $\R^n$.
        \item There exists a constant $\kappa \geq 1$ (depending only on $n$) such that for all $A \in K$ which is not $0$-dimensional, the set
            \begin{equation}
                V_A(\kappa) = \set{x \in \R^n | \mathrm{d}(x,\partial A) < \kappa^{-1} \mathrm{d}(x,A)}
            \end{equation}
            is included in $V_A$.
        \item Every cell $A \in K$ is included in at most $3^n$ cells $B \in K$.
    \end{enumerate}
\end{defi}
We prove our claims below the remark. 
\begin{rmk}\label{rmk_system}
    The goal of this remark is to precise the analogy between a system of $n$-charts and the notion of \emph{direct system} in algebra. Let the index set of $(K_\alpha)$ be denoted by $M$. Equip $M$ with the following quasi-order: $\alpha \leq \beta$ if $K_\alpha \preceq K_\beta$. When $\alpha \leq \beta$, there exists a natural function
    \begin{equation}
        \pi_\alpha^\beta\colon K_\alpha \to K_\beta
    \end{equation}
    which associates to each $A \in K_\alpha$, the unique $B \in K_\beta$ such that $\mathrm{int}(A) \subset \mathrm{int}(B)$. It is clear that
    \begin{equation}
        \pi_\alpha^\alpha = \mathrm{identity}
    \end{equation}
    and for $\alpha \leq \beta \leq \gamma$ in $M$,
    \begin{equation}
        \pi_\beta^\gamma \pi_\alpha^\beta = \pi_\alpha^\gamma.
    \end{equation}
    Given $\alpha$ and $\beta$, the first axiom of our definition allows to paste $K_\alpha$ and $K_\beta$. The result is a $n$-complex $K$ such that $K_\alpha, K_\beta \preceq K$ and which can be added to the system. In this way, $M$ becomes a directed set. Then we say that two cells $A \in K_\alpha$ and $B \in K_\beta$ are equivalent if there exists some $\gamma \geq \alpha, \beta$ such that $\pi_\alpha^\gamma(A) = \pi_\beta^\gamma(B)$. In algebra, the \emph{direct limit} of such system would be the set of all equivalent classes. Our definition take a few shortcuts because the abstract equivalent classes are replaced by a collection of maximal cells. We could not avoid a long axiom to say that these maximal cells are locally finite.
\end{rmk}
\begin{proof}
    It is straightforward from the definition that the cells of $K$ have disjoint interiors. Moreover, the second axiom implies that $K$ is locally finite in $\bigcup \abs{K_\alpha}$. We are going to show that for every $A \in \bigcup K_\alpha$, there exists $B \in K$ such that
    \begin{equation}
        \mathrm{int}(A) \subset \mathrm{int}(B).
    \end{equation}
    Such an inclusion implies the inclusion of the closures $A \subset B$ and it will follow $\abs{K} = \bigcup \abs{K_\alpha}$. We introduce
    \begin{equation}
        S = \set{B \in \bigcup K_\alpha | \mathrm{int}(A) \subset \mathrm{int}(B)}
    \end{equation}
    and we search a biggest element in the set
    \begin{equation}\label{S_img}
        \set{\mathrm{int}(B) | B \in S}.
    \end{equation}
    The first axiom implies that the set (\ref{S_img}) is totally ordered.  We apply the second axiom at a fixed point $x_0 \in \mathrm{int}(A)$ and we obtain a finite subset $T \subset \bigcup K_\alpha$ such that for all $B \in S$, there exists $C \in T$ such that $\mathrm{int}(B) \subset \mathrm{int}(C)$. In particular, $C \in S$. Therefore, a biggest element of (\ref{S_img}) can be searched for in a finite subset:
    \begin{equation}
        \set{\mathrm{int}(C) | C \in S \cap T}
    \end{equation}
    We conclude that there exists $B \in S$ such that $\mathrm{int}(B)$ is the biggest element of (\ref{S_img}). Then we justify that $B$ is a maximal cell. For all $C \in \bigcup K_\alpha$, we have either $\mathrm{int}(C) \subset \mathrm{int}(B)$ or $\mathrm{int}(B) \subset \mathrm{int}(C)$. In the second case, $C \in S$ so we have in fact $\mathrm{int}(B) = \mathrm{int}(C)$. In all cases, $\mathrm{int}(C) \subset \mathrm{int}(B)$.

    Next, we fix $A \in K$ and we show that the set
    \begin{equation}
        V_A = \bigcup \set{\mathrm{int}(B) | B \in K \ \text{contains} \ A}
    \end{equation}
    is open in $\R^n$. There exists an index $\alpha$ such that $A \in K_\alpha$ and since $K_\alpha$ is similar to $E_n$, the set
    \begin{equation}\label{V_alpha}
        W_A = \set{\mathrm{int}(C) | C \in K_\alpha \ \text{contains} \ A}
    \end{equation}
    is open in $\R^n$. For every $C \in K_\alpha$ containing $A$, there exists $B \in K$ such that $\mathrm{int}(C) \subset \mathrm{int}(B)$ and in particular $B$ contains $A$. We deduce that $W_A \subset V_A$. This shows that $V_A$ contains an open neighborhood of $\mathrm{int}(A)$ in $\R^n$. Furthermore, for all $B \in K$ containing $A$, we have $V_B \subset V_A$ so $V_A$ also contains an open neighborhood of $\mathrm{int}(B)$ in $\R^n$. This proves that $V_A$ is open in $\R^n$. If $A$ is not $0$-dimensional, the constant $\kappa$ that we had found for $E_n$ fulfills
    \begin{equation}
        V_A(\kappa) \subset W_A \subset V_A.
    \end{equation}
    Now, we want to estimate the number of cells $B \in K$ that contains $A$. Let $B \in K$ containing $A$. The set $W_A$ contains $\mathrm{int}(A)$ so it meets $B$. As $W_A$ is open in $\R^n$ and $B = \overline{\mathrm{int}(B)}$, $W_A$ also meets $\mathrm{int}(B)$. We deduce that there exists $C \in K_\alpha$ containing $A$ such that $\mathrm{int}(C) \cap \mathrm{int}(B) \ne \emptyset$. In fact, there exists $B' \in K$ such that $\mathrm{int}(C) \subset \mathrm{int}(B')$ and since the cells of $K$ have disjoint interior, we necessarily have $B' = B$. There are at most $3^n$ cells $C \in K_\alpha$ and since there can be only one $B \in K$ such that $\mathrm{int}(C) \subset \mathrm{int}(B)$; we deduce that there are at most $3^n$ cells $B \in K$ containing $A$.
\end{proof}

\begin{exa}\label{grid_example}
    Remember
    \begin{equation}
        E_n = \Set{\prod_{i = 1}^n [0,\alpha_i] | \alpha \in \set{-1,0,1}^n}.
    \end{equation}
    The family of translations of $E_n$ by $\Z^n$,
    \begin{equation}
        \Set{p + E_n | p \in \Z^n},
    \end{equation}
    is a system of $n$-charts. Its limit is the canonical grid of $\R^n$. We could have defined this grid directly as the set of all cells of the form
    \begin{equation}
        p + \prod_{i=1}^n [0,\alpha_i],
    \end{equation}
    where $p \in \Z^n$ and $\alpha \in \set{0,1}^n$.
\end{exa}

\begin{exa}\label{whitney_example}
    Let $X$ an open set of $\R^n$. We build a $n$-complex filling $X$ and which is analogous to a Whitney decomposition. For $k \in \N$ and $p \in 2^{-k} \Z^n$, we introduce the dyadic chart of center $p$ and sidelength $2^{-k}$:
    \begin{equation}
        E_n(p,k) = p + 2^{-k} E_n,
    \end{equation}
    i.e.
    \begin{equation}
        E_n(p,k) = \Set{\prod_{i=1}^n [p_i,p_i + 2^{-k} \alpha_i] | \alpha \in \set{-1,0,1}^n}.
    \end{equation}
    Then we consider the family
    \begin{equation}\label{dyadic_system}
        \Set{E_n(p,k) | k \in \N,\ p \in 2^{-k} \Z^n,\ \abs{E_n(p,k)} \subset X}.
    \end{equation}
    All the cells induced by this family are dyadic cells. We rely on the following property: for all dyadic cell $A$ of sidelength $2^{-k}$, for all dyadic cell $B$ of sidelength $2^{-l}$ with $l \leq k$, the intersection $\mathrm{int}(A) \cap \mathrm{int}(B) \ne \emptyset$ implies $\mathrm{int}(A) \subset \mathrm{int}(B)$. It is then clear that the first axiom of Definition \ref{defi_system} is satisfied. Let us focus on the second axiom. Instead of working with the Euclidean norm, we work with the maximum norm $\abs{\, \cdot \,}_\infty$, the corresponding distance $d_\infty$ and the corresponding open (cubic) balls $U$. We fix $x \in X$ and we define
    \begin{equation}
        r = \min\set{1, \mathrm{d}_\infty(x, X^c)}.
    \end{equation}
    Let $k \in \N$ be such that $2^{-k+1} \leq r \leq 2^{-k+2}$. There exists $p \in 2^{-k} \Z^n$ such that $\abs{x - p}_\infty < 2^{-k-1}$. Then, the triangular inequality gives
    \begin{equation}
        U(x, 2^{-k-1}) \subset U(p, 2^{-k}) \subset \overline{U}(p, 2^{-k}) \subset U(x,2^{-k+1}).
    \end{equation}
    so
    \begin{equation}
        U(x, 2^{-3}r) \subset U(E_n(p,k)) \subset \abs{E_n(p,k)} \subset U(x,r).
    \end{equation}
    According to the definition of $r$, $E_n(p,k)$ belongs to the family (\ref{dyadic_system}). This proves that the family fills $X$. Let $A$ be a cell of the system which meets $U(x, 2^{-3}r)$. The interior $\mathrm{int}(A)$ also meets $U(x,2^{-3}r)$ because $U(x,2^{-3}r)$ is an open set and $A = \overline{\mathrm{int}(A)}$. Then $\mathrm{int}(A)$ meets the interior of a cell $B \in E_n(p,k)$. Either the sidelength of $A$ is $\geq 2^{-k}$, either it is $\leq 2^{-k}$ and $\mathrm{int}(A) \subset \mathrm{int}(B)$. In both cases, there exists a dyadic cell $C$ of sidelength $\geq 2^{-k}$ such that $\mathrm{int}(A) \subset \mathrm{int}(C)$. However, there exists only a finite number of dyadic cells of sidelength $\geq 2^{-k}$ which meet $U(x,2^{-3}r)$. We deduce the second axiom of Definition \ref{defi_system}.
\end{exa}

\subsection{Estimates}\label{subsection_FF}
We recall the Federer--Fleming projection with an example. Let us say $K$ is the set of dyadic cells of sidelength $2^{-k}$ which are included in $Q = [-1,1]^n$ but not in $\partial Q$. Let us say that $E$ is a $1$-dimensional subset of $Q$. We perform a radial projection in each cell $A \in K^n$, then in each cell $A \in K^{n-1}$, ... until the cells $A \in K^2$. This operation sends $E$ to
\begin{equation}
    \abs{K} \setminus \bigcup \set{\mathrm{int}(A) | \mathrm{dim} \, A > 1} = \partial Q \cup \abs{K^1}.
\end{equation}
Note that $E$ is not entirely sent in the $1$-skeleton because we do not make radial projections on $\partial Q$. We denote by $\phi$ the function obtained by composing these radial projections. Throughout the construction, we can choose the centers of projections so as to control the measure of $\phi(E)$. Such procedure was already developed in \cite{DS}. In this subsection, we state this projection in the language of complexes and we present a new estimate.

We recall some notions about the Grassmannian spaces. Let $E$ be an euclidean space of dimension $n$ and $0 \leq d \leq n$ an integer. The Grassmannian $G(d,E)$ is the set of all $d$-linear planes of $X$. A linear plane $V$ can be represented by its orthogonal projection $p_V$. Thus, the operator norm induces a distance on $G(d,E)$. The Haar measure $\gamma_{d,E}$ is the unique Radon measure on $G(d,E)$ whose total mass is $1$ and which is invariant under the action of the orthogonal group $O(E)$ (see \cite[Chapter 3]{Mattila} for existence and unicity). In the case $E = \R^n$, the Grassmannian is denoted by $G(d,n)$ and the Haar measure by $\gamma_{n,d}$.

It is usually not defined that way but $\gamma_{d,n}$ coincides with the Hausdorff measure of dimension $m=d(n-d)$, up to a multiplicative constant. Indeed, $G(d,n)$ is non-empty compact manifold of dimension $m$ so $\sigma = \HH^m(G(d,n))$ is finite and positive. Then $\sigma^{-1} \HH^m$ is a Radon measure on $G(d,n)$ whose total mass is $1$ and which is invariant under the isometries of $G(d,n)$. In particular, it is invariant under the action of $O(n)$. One can justify similarly that for all $A \subset G(1,n+1)$
\begin{equation}
    \gamma_{1,n+1}(A) = \sigma_n^{-1} \HH^n(\set{x \in \mathbf{S}^n | \exists L \in A,\ x \in L})
\end{equation}
where $\sigma_n = \HH^n(\mathbf{S}^n)$. If there is no ambiguity, we write $\mathrm{d}V$ in the integrals instead of $\mathrm{d}\gamma_{d,n}(V)$.

\begin{lem*}[Federer--Fleming projection]
    Let $K$ be $n$-complex. Let $0 \leq d < n$ be an integer and let $E$ be Borel subset of $\abs{K}$ such that $\HH^{d+1}(\abs{K} \cap \overline{E}) = 0$. Then there exists a locally Lipschitz function $\phi\colon \abs{K} \to \abs{K}$ satisfying the following properties:
    \begin{enumerate}[label=(\roman*)]
        \item for all $A \in K$, $\phi(A) \subset A$;
        \item $\phi = \mathrm{id}$ in $\abs{K} \setminus \bigcup \set{\mathrm{int}(A) | \mathrm{dim} \, A > d}$;
        \item there exists a relative open set $O \subset \abs{K}$ such that $E \subset O$ and
            \begin{equation}
                \phi(O) \subset \abs{K} \setminus \bigcup \set{\mathrm{int}(A) | \mathrm{dim} \, A > d};
            \end{equation}
        \item for all $A \in K$,
            \begin{equation}\label{FF_estimate1}
                \HH^d(\phi(A \cap E)) \leq C \HH^d(A \cap E);
            \end{equation}
        \item for all $A \in K^d$,
            \begin{equation}\label{FF_estimate2}
                \HH^d(A \cap \phi(E)) \leq C \int_{G(d,n)} \HH^d(p_V(V_A \cap E)) \, \mathrm{d}V,
            \end{equation}
    \end{enumerate}
    where $C \geq 1$ is a constant that depends only on $n$.
\end{lem*}
The construction itself is deferred in Appendix \ref{appendix_FF} (Lemma \ref{lem_FF}) because this is not an original procedure and the reader can skip it. The goal of this subsection is to explain the estimate (\ref{FF_estimate1}) and the new estimate (\ref{FF_estimate2}).

David and Semmes have shown in \cite[Lemma 3.22]{DS} that, in average among centers of radial projections, the measure of the image is not much larger than the original set. We state their lemma without proof below.
\begin{lem}\label{lem_DS1}
    Let $Q$ be a cube of $\R^n$. Let $E$ be a Borel subset of $Q$ such that $\LL^n(\overline{E}) = 0$. Then
    \begin{equation}\label{DS1}
        \mathrm{diam}(Q)^{-n} \int^*_{\frac{1}{2} Q} \HH^d(\psi_x(E)) \, \mathrm{d}x \leq C\HH^d(E),
    \end{equation}
    where $\psi_x$ is the radial projection onto $\partial Q$ centered at $x$ and $C$ is a constant that depends only on $n$.
\end{lem}
Combining (\ref{DS1}) with the Markov inequality yields that for $\lambda > 0$,
\begin{equation}\label{DS1_Markov}
    \LL^n(\set{x \in \tfrac{1}{2}Q \setminus \overline{E} | \HH^d(\psi_x(E)) > \lambda \HH^d(E)}) \leq C \lambda^{-1} \mathrm{diam}(Q)^n
\end{equation}
We can choose $\lambda$ big enough (depending on $n$) so that the set at the left hand side has a low density in $Q$. In particular, we can find $x$ for which $\HH^d(\psi_x(E)) \leq \lambda \HH^d(E)$. We are going to develop a similar lemma where $\HH^d$ is replaced by a certain gauge $\zeta^d$.

We explain first why (\ref{DS1}) is not optimal for $d=n-1$. Let $B$ be a ball centered in $0$ of radius $R > 0$. Let $E$ be a $(n-1)$ dimensional closed subset of $B$. For $x \in \tfrac{1}{2} B \setminus E$, the image $\psi_x(E)$ is included in $\partial B$ so we have trivially
\begin{equation}\label{DS0}
    R^{-n} \int_{\frac{1}{2} B} \HH^{n-1}(\psi_x(E)) \, \mathrm{d}x \leq C R^{n-1}
\end{equation}
If $E$ has a high density in $B$, then $\HH^{n-1}(E) \gg R^{n-1}$ so (\ref{DS1}) is much worse than the trivial estimate (\ref{DS0}). In fact, the proof of Lemma \ref{lem_DS1} does not take into account that $\psi_x(E)$ is non-injective. David and Semmes bound $\HH^{n-1}(\psi_x(E))$ by the total variation,
\begin{equation}
    \HH^{n-1}(\psi_x(E)) \leq \int_{\psi_x(E)} \! \HH^0(E \cap \psi_x^{-1}(y)) \, \mathrm{d}\HH^{n-1}(y),
\end{equation}
and then they apply the Fubini theorem. In our approach, we bound
\begin{equation}
    \HH^{n-1}(\psi_x(E)) \leq C R^{n-1} \gamma_{1,n}(\set{L | (x + L) \cap E \ne \emptyset}
\end{equation}
and then we apply the Fubini theorem. This yields (the details are in Lemma \ref{lem_DS2}),
\begin{equation}\label{DS20}
    R^{-n} \int_{\frac{1}{2} B} \HH^{n-1}(\psi_x(E)) \, \mathrm{d}x \leq C \int_{G(n-1,n)} \HH^{n-1}(p_V(E)) \, \mathrm{d}V.
\end{equation}
As orthogonal projections are $1$-Lipschitz, one can see that (\ref{DS20}) is better than both (\ref{DS1}) and (\ref{DS0}). Unfortunately, it is not true for $0 \leq d < n-1$ that
\begin{equation}
    R^{-n}\int_{\frac{1}{2} B} \! \HH^d(\psi_x(E)) \, \mathrm{d}x \leq C \int_{G(d,n)} \HH^d(p_V(E)) \, \mathrm{d}V.
\end{equation}
We have to average a different quantity.

We define the gauge $\zeta^d$ on Borel subsets of $\R^n$ by
\begin{equation}
    \zeta^d(E) = \int_{G(d,n)} \HH^d(p_V(E)) \, \mathrm{d}V.
\end{equation}
This is a basic way to estimate the size of $E$ but it is not a Borel measure. The function $\zeta^d$ is known to induce the integral-geometric measure via Carathéodory's construction (\cite[2.10.5]{Federer}). The function $\zeta^d$ also cancels the purely unrectifiable part of $E$ thanks to the Besicovitch--Federer projection theorem (\cite[Theorem 18.1]{Mattila}).

For a cell $A$, we define the restriction of this gauge to $A$
\begin{equation}
    \zeta^d \mres A(E) = \int_{G(\mathrm{aff}(A),d)} \HH^d(p_V(A \cap E)) \, \mathrm{d}V,
\end{equation}
where $\mathrm{aff}(A)$ is the affine span of $A$ and $G(\mathrm{aff}(A),d)$ is the set of all $d$-linear planes of $\mathrm{aff}(A)$ centered at an arbitrary point. The author does not know whether there exists a constant $C > 0$ (depending on $n$, $d$) such that
\begin{equation}
    \zeta^d \mres A(E) = C \zeta^d(A \cap E).
\end{equation}
but at least the next lemma implies that there exists a constant $C \geq 1$ (depending on $n$) such that
\begin{equation}\label{zeta_mres}
    \zeta^d \mres A(E) \leq C \zeta^d(A \cap E).
\end{equation}
We justify inequality (\ref{zeta_mres}) in Remark \ref{rmk_zeta} below the proof of the next lemma. When $A$ is $d$-dimensional, it is clear that $\zeta^d \mres A = \HH^d$. At the end of the Federer--Fleming procedure we thus have a simplification: for $A \in K^d$, $\zeta^d \mres A(\phi(E)) = \HH^d(A \cap \phi(E))$. For a cube $Q$, we find convenient to define
\begin{equation}
    \zeta^d \mres \partial Q = \sum_A \zeta^d \mres A
\end{equation}
where the sum is indexed by the $(n-1)$-faces $A$ of $Q$.

\begin{lem}\label{lem_DS2}
    Let $Q$ be a cube of $\R^n$. Let $E$ be a Borel subset of $Q$ such that $\LL^n(\overline{E}) = 0$. Then
    \begin{equation}\label{DS2}
        \mathrm{diam}(Q)^{-n} \int^{*}_{\frac{1}{2} Q} \zeta^d \mres \partial Q(\psi_x(E)) \, \mathrm{d}x \leq C \zeta^d(E),
    \end{equation}
    where $\psi_x$ is the radial projection onto $\partial Q$ centered at $x$ and $C$ is a constant that depends only on $n$.
\end{lem}
\begin{proof}
    The letter $C$ plays the role of a constant $\geq 1$ that depends on $n$. Its value can increase from one line to another (but a finite number of times). We denote by $R$ the diameter of $Q$. The principle is the following: we show that for all $(n-1)$-face $A$ of $Q$ and for all $x \in \frac{1}{2} Q \setminus E$,
    \begin{equation}\label{jauge-1}
        \zeta^d \mres A(\psi_x(E)) \leq C R^d \gamma_{n-d,n}(\set{W | (x + W) \cap E \ne \emptyset})
    \end{equation}
    and then we apply the Fubini theorem to get
    \begin{equation}\label{jauge-2}
        R^{-n} \int_Q \gamma_{n-d,n}(\set{W | (x + W)  \cap E \ne \emptyset}) \, \mathrm{d}x \leq C R^{-d} \zeta^d(E).
    \end{equation}

    \emph{Step 1.} Let $x \in \tfrac{1}{2} Q \setminus E$. We make a translation to assume that $x = 0$ (but the cube may not be centered at the origin). We denote by $\psi$ the radial projection onto $\partial Q$ centered at $0$. We fix a $(n-1)$ face $A$ of $Q$. We aim to show
    \begin{equation}
        \zeta^d \mres A(\psi(E)) \leq C R^d \gamma_{n-d,n}(\set{W | W \cap E \ne \emptyset}).
    \end{equation}
    Let $H$ be the affine span of $A$. It suffices to show that for all Borel set $F \subset H \cap B(0,R)$,
    \begin{equation}\label{zeta_step1}
        \int_{G(d,H)} \HH^d(p_V(F)) \, \mathrm{d}V \leq C R^d \gamma_{n-d,n}(\set{W | W \cap F \ne \emptyset})
    \end{equation}
    and to apply this inequality with $F = A \cap \phi(E)$. We apply the Disintegration Formula (Proposition \ref{grassSum}) for $n-d = 1 + m$ where $m=n-d-1$,
    \begin{multline}
        \gamma_{n-d,n}(\set{W | W \cap F \ne \emptyset})\\ = \int_{G(1,n)} \gamma_{m,L^\perp}(\set{W | (L + W) \cap F \ne \emptyset}) \, \mathrm{d}L.
    \end{multline}
    In particular,
    \begin{multline}
        \gamma_{n-d,n}(\set{W | W \cap F \ne \emptyset}) \\ \geq \int_{B(L_0,\alpha)} \gamma_{m,L^\perp}(\set{W | (L + W) \cap F \ne \emptyset}) \, \mathrm{d}L
    \end{multline}
    where $L_0$ be the vector line orthogonal to $H$, $\alpha \in \mathopen{]}0,1\mathclose{[}$ is a constant that we will specify later and $B(L_0,\alpha)$ is the ball of center $L_0$ and radius $\alpha$ in $G(1,n)$. For $L \in G(1,n)$ such that $\mathrm{d}(L,L_0) \leq \alpha$, we show that
    \begin{multline}\label{grass_goal1}
        \gamma_{m,L^\perp}(\set{W | (L + W) \cap F \ne \emptyset}) \\\geq C(\alpha)^{-1} \gamma_{m,L_0^\perp}(\set{W | (L + W) \cap F \ne \emptyset}),
    \end{multline}
    where $C(\alpha)$ is a constant $\geq 1$ that depends on $n$ and $\alpha$. We define $u\colon L_0^\perp \to L^\perp$ to be the orthogonal projection onto $L^\perp$. Since $\mathrm{d}(L,L_0) \leq \alpha < 1$, the inequality (\ref{grassGraph}) in Appendix \ref{appendix_grass} says that $u$ is an isomorphism from $L_0^\perp$ to $L^\perp$ and $\norm*{u} \norm*{u^{-1}} \leq C(\alpha)$. According to Lemma \ref{lem_grass_iso}, $u$ induces a $C(\alpha)$-bilipschitz and one-to-one correspondence from $G(L_0^\perp,d)$ onto $G(L^\perp,d)$. We observe finally that for all set $W \subset L_0^\perp$, we have $L + u(W) = L + W$ because for $x \in L_0^\perp$, $x - u(x) \in L$. We deduce (\ref{grass_goal1}) by the action of Lipschitz functions on Hausdorff measures. We now have
    \begin{multline}
        \gamma_{n-d,n}(\set{W | W \cap F \ne \emptyset}) \\\geq C(\alpha)^{-1} \int_{G(L_0,\alpha)} \gamma_{m,L_0^\perp}(\set{W | (L + W) \cap F \ne \emptyset}) \, \mathrm{d}L
    \end{multline}
    and the right-hand side allows an application of Fubini,
    \begin{multline}
        \gamma_{n-d,n}(\set{W | W \cap F \ne \emptyset}) \\\geq C(\alpha)^{-1} \int_{G(m,L_0^\perp)} \gamma_{1,n}(\set{L \in B(L_0,\alpha) | (L + W) \cap F \ne \emptyset}) \, \mathrm{d}W.
    \end{multline}
    It is convenient at this point to make the change of variable $W = V^\perp$ (the orthogonal complement is taken in $L_0^\perp$) where $V \in G(d,L_0^\perp)$,
    \begin{multline}
        \gamma_{n-d,n}(\set{W | W \cap F \ne \emptyset}) \\\geq C(\alpha)^{-1} \int_{G(d,L_0^\perp)} \gamma_{1,n}(\set{L \in B(L_0,\alpha) | (L + V^\perp) \cap F \ne \emptyset}) \, \mathrm{d}V.
    \end{multline}
    Notice that for $V \in G(d,L_0^\perp)$ and $L \in B(L_0,1)$, we have $(L + V^\perp) \cap F \ne \emptyset$ if and only if $(x_L + V^\perp) \cap F \ne \emptyset$ where $x_L$ is the intersection of $L$ with $H$. We apply Lemma \ref{grassParam} and obtain that for some $\alpha \in ]0,1[$ (depending on $n$) and for all $V \in G(d,L_0^\perp)$,
    \begin{multline}
        \HH^{n-1}(\set{x \in H \cap B(0,2R) | (x + V^\perp) \cap F \ne \emptyset}) \\\leq C R^{n-1} \gamma_{1,n}(\set{L \in B(L_0,\alpha) | (L + V^\perp) \cap F \ne \emptyset}).
    \end{multline}
    In addition, we decompose the Lebesgue measure as a product to see that
    \begin{equation}
        \HH^{n-1}(\set{x \in H \cap B(0,2R) | (x + V^\perp) \cap F \ne \emptyset}) \geq C^{-1} \HH^d(p_V(F)).
    \end{equation}
    We have proved (\ref{zeta_step1}).

    \emph{Step 2.} We apply the Fubini theorem,
    \begin{multline}
        \int_Q \gamma_{n,n-d}(\set{W | (x + W) \cap E \ne \emptyset}) \, \mathrm{d}x \\= \int_{G(n-d,n)} \LL^n(\set{x \in Q | (x + W) \cap E \ne \emptyset}) \, \mathrm{d}W
    \end{multline}
    and we make the change of variable $W = V^\perp$ where $V \in G(d,n)$,
    \begin{multline}
        \int_Q \gamma_{n,n-d}(\set{W | (x + W) \cap E \ne \emptyset}) \, \mathrm{d}x \\= \int_{G(d,n)} \LL^n(\set{x \in Q | (x + V^\perp) \cap E \ne \emptyset}) \, \mathrm{d}V.
    \end{multline}
    Then we decompose the Lebesgue measure as a product to see that
    \begin{equation}
        \LL^n(\set{x \in Q | (x + V^\perp) \cap E \ne \emptyset}) \leq C R^{n-d} \HH^d(p_V(E)).
    \end{equation}
\end{proof}
\begin{rmk}\label{rmk_zeta}
    We justify that there exists a constant $C \geq 1$ (depending on $n$) such that for all cell $A \subset \R^n$, for all Borel set $E \subset A$,
    \begin{equation}\label{zeta_goal}
        \zeta^d \mres A(E) \leq C \zeta^d(E).
    \end{equation}
    First of all, we assume that $A$ is $(n-1)$-dimensional. We can consider that $A$ is a face of some cube $Q$ and then by the previous lemma,
    \begin{equation}\label{re-DS2}
        \mathrm{diam}(Q)^{-n} \int_{\frac{1}{2} Q}^* \zeta^d \mres A (\psi_x(E)) \, \mathrm{d}x \leq C \zeta^d(E).
    \end{equation}
    However, for all $x \in \frac{1}{2} Q$, we have $\psi_x = \mathrm{id}$ on $A$ so (\ref{re-DS2}) simplifies to (\ref{zeta_goal}). We deduce the general case by induction.
\end{rmk}

\section{Properties of Quasiminimal Sets}\label{mini_properties}
\subsection{Local Ahlfors-Regularity and Rectifiability}
We recall we introduce a few notions. For $k \in \N$, we denote by $\mathcal{E}_n(k)$ the grid of dyadic cells of sidelength $2^{-k}$ in $\R^n$. This is the set of all faces of the form
\begin{equation}
    p + 2^{-k} \prod_{i=1}^n [0,\alpha_i],
\end{equation}
where $p \in 2^{-k} \Z^n$ and $\alpha \in \set{0,1}^n$. The canonical $n$-complex of $\R^n$ is
\begin{equation}
    E_n = \Set{\prod_{i = 1}^n [0,\alpha_i] | \alpha \in \set{-1,0,1}^n}.
\end{equation}
For $k \in \N$ and $p \in 2^{-k} \Z^n$, the dyadic chart of center $p$ and sidelength $2^{-k}$ is
\begin{equation}
    E_n(p,k) = p + 2^{-k} E_n.
\end{equation}

We fix a Whitney subset $\Gamma$ of $X$ and an admissible energy $\II$. We will only need a limited number of their properties in this subsection.
\begin{enumerate}
    \item There exists a scale $t > 0$ and a constant $\GL \geq 1$ such that for all $x \in \Gamma$, for all $0 \leq r \leq r_t$, there exists an open set $O \subset \R^n$, a $\GL$-bilipschitz map $\GT\colon B(x,r) \to O$, an integer $k \in \N$ and a subset $S \subset \mathcal{E}_n(k)$ such that $\GT(\Gamma \cap B(x,r)) = O \cap \abs{S}$ and
        \begin{equation}
            \GT(B(x,\GL^{-1} r)) \subset B(0,2^{-k-1}) \subset B(0,2^{-k} \sqrt{n}) \subset O.
        \end{equation}
    \item There exists $\IL \geq 1$ such that 
        \begin{equation}
            \IL^{-1} \HH^d \leq \II \leq \IL \HH^d
        \end{equation}
\end{enumerate}
We will use the letters $t$, $\GL$, $\IL$ without reminders in the statements and the proofs of this section.

We recall that the core $E^*$ of a closet set $E \subset X$ is the support of $\HH^d \mres E$ in $X$. It can be characterized as
\begin{equation}
    E^* = \set{x \in X | \forall r > 0,\ \HH^d(E \cap B(x,r)) > 0}.
\end{equation}

The following proposition adapts \cite[Proposition 4.74]{Sliding} to our definition of quasiminimal sets and boundaries.
\begin{prop}[Local Ahlfors-regularity]\label{prop_density}
    Fix $\kappa \geq 1$, $h > 0$ and a scale $s \in ]0,\infty]$. Assume that $h$ is small enough (depending on $n$, $\GL$, $\IL$). Let $E$ be a $(\kappa,h,s)$-quasiminimal set with respect to $\II$ in $X$. There exists $C > 1$ (depending on $n$, $\kappa$, $\GL$, $\IL$) and a scale $u > 0$ (depending on $n$, $s$, $t$) such that for all $x \in E^*$, for all $0 < r \leq r_{u}(x)$,
    \begin{equation}\label{ahlfors-regularity}
        C^{-1} r^d \leq \HH^d(E \cap B(x,r)) \leq C r^d.
    \end{equation}
\end{prop}
\begin{proof}
    The letter $C$ plays the role of a constant $\geq 1$ that depends on $n$, $\kappa$, $\GL$, $\IL$. Its value can increase from one line to another (but a finite number of times). We assume, without loss of generality, that $E = E^*$.

    We are going to reduce the statement of the proposition to a simpler situation. These simplifications preserve the quasiminimal condition except for the parameters $\kappa$, $h$ which are worsened by a multiplicative factor $\GL^2 \IL^2$. 

    We fix $x \in E$ and $0 < r \leq \min \set{r_s(x),r_t(x)}$. For all sliding deformation $f$ of $E$ in $B(x,r)$,
    \begin{equation}
        \II(E \cap W_f) \leq \kappa \II(f(E \cap W_f)) + h \II(E \cap B(x,hr)).
    \end{equation}
    As $\IL^{-1} \HH^d \leq \II \leq \IL \HH^d$, this implies
    \begin{equation}
        \HH^d(E \cap W_f) \leq \kappa \IL^2 \HH^d(f(E \cap W_f)) + h \IL^2 \HH^d(E \cap B(x,hr)).
    \end{equation}
    There exists an open set $O \subset \R^n$, a $\GL$-bilipschitz map $\GT\colon B(x,r) \to O$, an integer $k \in \N$ and a subset $S \subset \mathcal{E}_n(k)$ such that $\GT(\Gamma \cap B(x,r)) = O \cap \abs{S}$ and
    \begin{equation}
        \GT(B(x,\GL^{-1} r)) \subset B(0,2^{-k-1}) \subset B(0,2^{-k} \sqrt{n}) \subset O.
    \end{equation}
    We are going to see that the image $E'= \GT(E \cap B(x,r))$ satisfies a quasiminimality condition along $\Gamma' = \GT(\Gamma \cap B(x,r))$ in $B = B(0,2^{-k}\sqrt{n})$. Let $g$ be a sliding deformation of $E'$ along $\Gamma'$ in $B$. The function $f = T^{-1} \circ g \circ T$ is a sliding deformation of $E$ along $\Gamma$ in $B(x,r)$ so
    \begin{equation}
        \HH^d(E \cap W_f) \leq \kappa \IL^2 \HH^d(f(E \cap W_f)) + h \IL^2\HH^d(E \cap B(x,hr)).
    \end{equation}
    We assume $h \leq \GL^{-1}$ so that $B(x,hr) \subset B(x,\GL^{-1}r)$. As $\GT$, $\GT^{-1}$ are $\GL$-Lipschitz and $\GT(B(x,\GL^{-1} r)) \subset B(0,2^{-k-1})$, it follows that
    \begin{multline}\label{F_quasi}
        \HH^d(E' \cap W_g) \leq \kappa \GL^2 \IL^2 \HH^d(g(E' \cap W_g))\\+ h \GL^2 \IL^2\HH^d(E' \cap B(0,2^{-k-1})).
    \end{multline}
    We can also scale the set $E'$ so as to assume $k = 0$ (this does not alter $\kappa$, $h$).

    We conclude that it suffices to solve the following problem. Let the ball $B(0,\sqrt{n})$ be denoted by $B$. We assume that $E$ is a $\HH^d$ finite closed subset of $B$ and
    \begin{enumerate}
        \item $\HH^d(E \cap B(0,\tfrac{1}{2})) > 0$;
        \item there exists $S \subset \mathcal{E}_n(1)$ such that $\Gamma \cap B = \abs{S} \cap B$;
        \item for all sliding deformations $f$ of $E$ along $\Gamma$ in $B$,
            \begin{equation}\label{E_quasi}
                \HH^d(E \cap W_f) \leq \kappa \HH^d(f(E \cap W_f)) + h \HH^d(E \cap B(0,\tfrac{1}{2})).
            \end{equation}
    \end{enumerate}
    Then we show that
    \begin{align}
        \HH^d(E \cap \mathopen{]}-\tfrac{1}{2},\tfrac{1}{2}\mathclose{[}^n)    &\leq Cr^d\\
        \HH^d(E \cap \mathopen{]}-1,1\mathclose{[}^n))                         &\geq Cr^d.
    \end{align}

    \emph{Step 1. We prove that} 
    \begin{equation}\label{alfors_step1}
        \HH^d(E \cap \mathopen{]}-\tfrac{1}{2},\tfrac{1}{2}\mathclose{[}^n) \leq C.
    \end{equation}
    The sliding deformations will be Federer--Fleming projections of $E$ in finite $n$-complexes $K \preceq E_n$. The condition $K \preceq E_n$ ensures that a Federer--Fleming projection of $U(K) \cap E$ in $K$ induces a global sliding deformation in $B = B(0,\sqrt{n})$. Let us justify this claim. Let $\phi$ be a Federer--Fleming projection of $U(K) \cap E$ in $K$ (as in Lemma \ref{lem_FF}). We know that $\phi$ is locally Lipschitz in $\abs{K}$ but since $K$ is finite, $\phi$ is in fact Lipschitz in $\abs{K}$. We justify that $\phi$ can be extended as a Lipschitz function $\phi\colon \R^n \to \R^n$ by $\phi = \mathrm{id}$ in $\R^n \setminus \abs{K}$. By definition of $n$-complexes, $U(K)$ is an open set of $\R^n$ included in $\abs{K}$ so $\partial \abs{K} \subset \abs{K} \setminus U(K)$. For $x \in \R^n \setminus \abs{K}$ and $y \in \abs{K}$, the segment $[x,y]$ meets $\partial \abs{K}$ at a point $z$ and $\phi(z) = z$ because $z \in \abs{K} \setminus U(K)$. As a consequence, we have
    \begin{align}
        \abs{\phi(y) - \phi(x)} &   \leq \abs{\phi(y) - \phi(z)} + \abs{\phi(z) - \phi(x)}\\
                                &   \leq L\abs{y - z} +  \abs{x - z}\\
                                &   \leq (L + 1) \abs{x - y},
    \end{align}
    where $L$ is the Lipschitz constant of $\phi$ in $\abs{K}$. Finally, we justify that the homotopy
    \begin{equation}
        \phi_t = (1-t) \mathrm{id} + t\phi
    \end{equation}
    preserves $\Gamma$. As $\phi = \mathrm{id}$ outside $]-1,1[^n$ and there exists $S \subset \mathcal{E}_n(1)$ such that $\Gamma \cap [-1,1]^n = \abs{S} \cap [-1,1]^n$, it suffices to check that $\phi$ preserves the cells of $E_n$. This is where the condition $K \preceq E_n$ comes into play. Let $A \in E_n$ and let $x \in A$. We are going to show that $\phi(x) \in A$. Either $x \notin U(K)$ and $\phi(x) = x$ or $x \in U(K)$ and there exists $B \in K$ such that $x \in \mathrm{int}(B)$. In the last case, we have $B \subset A$ because $K \preceq E_n$. Then by the properties of Federer--Fleming projections, $\phi(x) \in B \subset A$. Now, we proceed to the construction.

    We fix a parameter $0 < \mu < 1$ which is close enough to $1$ (this will be precised later). We aim to apply the Federer--Fleming projection in a sequence of complexes $(K_k)_{k \in \N}$ whose supports are of the form
    \begin{equation}
        \abs{K_k} = (1 - \mu^k) [-1,1]^n.
    \end{equation}
    We also want the complexes $K_k$ to be composed of dyadic cells so that $K_k \preceq E_n$. The supports of $(K_k)_{k \in \N}$ will be in fact $(\tfrac{1}{2} + \sum_{i < k} 2^{-q(i)}) [-1,1]^n$ where $(q(k))_{k \in \N}$ is a sequence of non-negative integers such that, in some sense,
    \begin{equation}
        \frac{1}{2} - \sum_{i = 0}^{k-1} 2^{-q(i)} \sim \mu^k.
    \end{equation}
    We define $(q(k))_k$ by induction. Assuming that $q(0),\ldots,q(k-1)$ have been built and $\tfrac{1}{2} - \sum_{i < k} 2^{-q(i)} > 0$, we define $q(k)$ as the smallest non-negative integer such that
    \begin{equation}\label{q1.1}
        \frac{\frac{1}{2} - \sum_{i \leq k} 2^{-q(i)}}{\frac{1}{2} - \sum_{i < k} 2^{-q(i)}} \geq \mu.
    \end{equation}
    We present the main properties of this sequence, some of which will only be useful in step 2. It is straightforward that $q(k) \geq 2$ otherwise the numerator of (\ref{q1.1}) is non-positive. We rewrite (\ref{q1.1}) as $q(k)$ being the smallest non-negative integer such that
    \begin{equation}\label{q1.2}
        \frac{2^{-q(k)}}{\frac{1}{2} - \sum_{i < k} 2^{-q(i)}} \leq 1 - \mu.
    \end{equation}
    We can deduce that $(q(k))$ is strictly non-decreasing. Indeed, by definition of $q(k+1)$ we have,
    \begin{equation}
        \frac{2^{-q(k+1)}}{\frac{1}{2} - \sum_{i < k} 2^{-q(i)}} < \frac{2^{-q(k+1)}}{\frac{1}{2} - \sum_{i < k+1} 2^{-q(i)}} \leq 1 - \mu.
    \end{equation}
    so $q(k+1) < q(k)$. More generally, the minimality of $q(k)$ with respect to (\ref{q1.2}) is equivalent to
    \begin{equation}\label{q1.5}
        \frac{1}{2}(1 - \mu) < \frac{2^{-q(k)}}{\frac{1}{2} - \sum_{i < k} 2^{-q(i)}} \leq (1 - \mu).
    \end{equation}
    An induction on (\ref{q1.1}) shows that for all $k \geq 0$,
    \begin{equation}\label{q1.3}
        \frac{1}{2} - \sum_{i < k} 2^{-q(i)} \geq \frac{1}{2} \mu^k.
    \end{equation}
    We combine (\ref{q1.5}) and (\ref{q1.3}) to obtain
    \begin{equation}\label{q1.4}
        2^{-q(k)} \geq \frac{1}{4}(1 - \mu) \mu^k,
    \end{equation}
    Finally, we justify that $\sum_i 2^{-q(i)} = \frac{1}{2}$. It is clear that $\sum_i 2^{-q(i)} \leq \frac{1}{2}$ so $\lim\limits_{i \to \infty} 2^{-q(i)} = 0$ and then (\ref{q1.5}) implies that $\frac{1}{2} - \sum_i 2^{-q(i)} = 0$.

    Now, we are ready to build the complexes $(K_k)_k$. For $k \geq 0$, we define $S_k$ as the set of dyadic cells of sidelength $2^{-q(k)}$ subdivising the cube $(\frac{1}{2} + \sum_{i < k} 2^{-q(i)}) [-1,1]^n$ but not lying on its boundary. Then we define $K_k$ as the collection of maximal cells of $\bigcup_{i=0}^k S_i$, as in Definition \ref{defi_system}. We abuse a bit the definition because $S_i$ is not an $n$-chart.
    \begin{figure}[ht]
        \begin{center}
            \begin{minipage}[c]{60px}
                \setlength{\fboxsep}{0pt}
                \fbox{\includegraphics[width=59px]{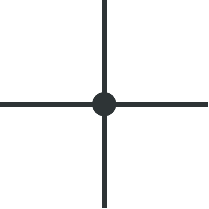}}
            \end{minipage}
            \quad
            \begin{minipage}[c]{90px}
                \setlength{\fboxsep}{0pt}
                \fbox{\includegraphics[width=89px]{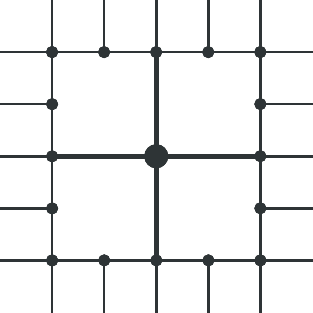}}
            \end{minipage}
            \quad
            \begin{minipage}[c]{105px}
                \setlength{\fboxsep}{0pt}
                \fbox{\includegraphics[width=104px]{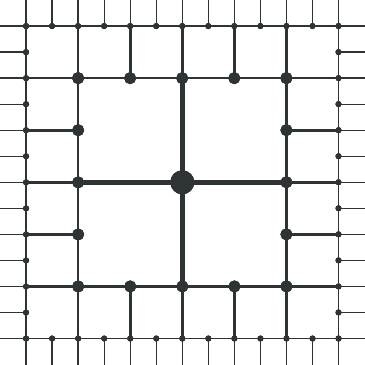}}
            \end{minipage}
            \caption{From left to right, examples of $K_0$, $K_1$, $K_2$. The external edges are not part of the complexes.}
        \end{center}
    \end{figure}

    We need a small number of observations. The set $K_k$ is a finite $n$-complex subordinated to $E_n$ and
    \begin{subequations}\label{K1}
        \begin{align}
            \abs{K_k}   &   = (\tfrac{1}{2} + \sum_{i < k} 2^{-q(i)}) [-1,1]^n,\\
            U(K_k)      &   = (\tfrac{1}{2} + \sum_{i < k} 2^{-q(i)}) \mathopen{]}-1,1\mathclose{[}^n.
        \end{align}
    \end{subequations}
    We abbreviate $U(K_k)$ as $U_k$; in particular $U_0 = \mathopen{]}-\tfrac{1}{2},\tfrac{1}{2}\mathclose{[}^n$. We also define
    \begin{equation}
        U_\infty = \bigcup_k U_k = \mathopen{]-1,1[}^n.
    \end{equation}
    The cells of $K_k$ have bounded overlap: each point of $\abs{K_k}$ belongs to at most $3^n$ cells $A \in K_k$. We show that $K_k$ is a subcomplex of $K_{k+1}$. We check first that $K_k \subset K_{k+1}$. This means that if $A$ is a maximal cell of $\bigcup_{i=0}^k S_i$, then it is also a maximal cell of $\bigcup_{i=0}^{k+1} S_i$. Indeed, we have for $B \in S_{k+1}$,
    \begin{equation}
        \mathrm{int}(A) \cap \mathrm{int}(B) \ne \emptyset \implies \mathrm{int}(B) \subset \mathrm{int}(A)
    \end{equation}
    because $A$ and $B$ are dyadic cells and the sidelength of $B$ is strictly less than the sidelength of $A$. Then, we check $K_k$ satisfies the definition of subcomplexes. This means that if $A \in K_k$ and $B$ is a maximal cell of $\bigcup_{i=0}^{k+1} S_i$ such that $A \subset B$, then $B$ is a maximal cell of $\bigcup_{i=0}^k S_i$. Indeed, $B$ cannot belong to $S_{k+1}$ because of the strictly decreasing sidelength. Finally, we show that
    \begin{align}
        \abs{K_{k+1}} \setminus U_k &= \bigcup \set{A \in K_{k+1} \setminus K_k}\\
                                    &= \bigcup \set{A \in K_{k+1} | A \cap U_k = \emptyset}.
    \end{align}
    The two right-hand sides are equivalent because according the properties of subcomplexes, we have $A \cap U_k \ne \emptyset$ if and only if $A \in K_k$. We justify that every $x \in \abs{K_{k+1}} \setminus U_k$ belongs to the right-hand side. We distinguish two cases. If $x \notin \abs{K_k}$, this is obvious. If $x \in \abs{K_k}$, then $x \in U_{k+1}$ so there exists $A \in K_{k+1}$ such that $x \in \mathrm{int}(A)$ and since $x \notin U_k$, we have $A \in K_{k+1} \setminus K_k$.

    Let $\phi$ be a Federer--Fleming projection of $E \cap U_{k+1}$ in $K_{k+1}$. We apply the quasiminimality condition (\ref{E_quasi}) with respect to the deformation $\phi$ in $B= B(0,\sqrt{n})$,
    \begin{equation}
        \HH^d(E \cap W_\phi) \leq \kappa \HH^d(f(E \cap W_\phi)) + h \HH^d(E \cap B(0,\tfrac{1}{2}))
    \end{equation}
    As $W_\phi \subset U_{k+1}$, we can use the following form which is more convenient
    \begin{equation}
        \HH^d(E \cap U_{k+1}) \leq (1+\kappa) \HH^d(f(E \cap U_{k+1})) + h \HH^d(E \cap B(0,\tfrac{1}{2}))
    \end{equation}
    We also assume $h \leq \tfrac{1}{2}$ so this simplifies to
    \begin{equation}
        \HH^d(E \cap U_{k+1})) \leq C \HH^d(\phi(E \cap U_{k+1})).
    \end{equation}
    We decompose $E \cap U_{k+1}$ in two parts: $E \cap U_k$ and $E \cap U_{k+1} \setminus U_k$. First, we claim $\HH^d(\phi(E \cap U_k)) \leq \HH^d(\abs*{K_{k+1}^d})$. Indeed, by the properties of the Federer--Fleming projections,
    \begin{equation}
        \phi(E \cap U_k) \subset \abs{K_{k+1}} \setminus \bigcup \set{\mathrm{int}(A) | A \in K_{k+1},\ \mathrm{dim} \, A > d}.
    \end{equation}
    As $K_k \subset K_{k+1}$, the function $\phi$ preserves the cells of $K_k$ and thus $\phi(U_k) \subset \abs{K_k}$. We deduce that
    \begin{align}
        \phi(E \cap U_k)    &   \subset \abs{K_k} \setminus \bigcup \set{\mathrm{int}(A) | A \in K_{k+1},\ \mathrm{dim} \, A > d}\\
                            &   \subset U_{k+1} \setminus \bigcup \set{\mathrm{int}(A) | A \in K_{k+1},\ \mathrm{dim} \, A > d}\\
                            &   \subset \bigcup \set{\mathrm{int}(A) | A \in K_{k+1},\ \mathrm{dim} \, A \leq d}.
    \end{align}
    Next, we claim that
    \begin{equation}
        \HH^d(\phi(E \cap U_{k+1} \setminus U_k)) \leq C \HH^d(E \cap U_{k+1} \setminus U_k).
    \end{equation}
    This is deduced from the fact that
    \begin{equation}
        U_{k+1} \setminus U_k = \bigcup \set{U_{k+1} \cap A | A \in K_{k+1},\ A \cap U_k = \emptyset},
    \end{equation}
    that for all $A \in K_{k+1}$, $\HH^d(\phi(E \cap U_{k+1} \cap A)) \leq C\HH^d(E \cap U_{k+1} \cap A)$ and that the cells of $K_{k+1}$ have bounded overlap. In conclusion,
    \begin{equation}\label{FF1}
        \HH^d(E \cap U_{k+1}) \leq C\HH^d(K_{k+1}^d) + C\HH^d(E \cap U_{k+1} \setminus U_k).
    \end{equation}
    We rewrite this inequality as
    \begin{equation}
        \HH^d(E \cap U_k) - \lambda \HH^d(E \cap U_{k+1}) \leq \HH^d(\abs*{K_{k+1}^d}),
    \end{equation}
    where $\lambda = C^{-1}(C - 1)$. We multiply both sides of the inequation by $\lambda^k$:
    \begin{equation}
        \lambda^k \HH^d(E \cap U_k) - \lambda^{k+1} \HH^d(E \cap U_{k+1}) \leq \lambda^k \HH^d(K_{k+1}^d).
    \end{equation}
    We sum this inequality over $k \geq 0$. Since $0 < \lambda < 1$ and $\HH^d(E \cap U_\infty) < \infty$, we have $\lim_{k \to \infty} \lambda^k \HH^d(E \cap U_k) = 0$ so
    \begin{equation}
        \HH^d(E \cap U_0) \leq \sum_k \lambda^k \HH^d(\abs*{K_{k+1}^d}).
    \end{equation}
    It is left to bound $\sum_k \lambda^k \HH^d(\abs*{K_{k+1}^d}) \leq C$. We choose $\mu$ so that $\mu^{n-d} = \lambda$. As the the cells of $K_{k+1}$ have a sidelength $\sim \mu^k$, we will see that this choice implies
    \begin{equation}\label{observation0}
        \lambda^k \left(\HH^d(\abs*{K_{k+1}^d}) - \HH^d(\abs*{K_k^d})\right) \leq C \LL^n(U_{k+1} \setminus U_k)
    \end{equation}
    Summing (\ref{observation0}) over $k \geq 0$ will give then
    \begin{align}
        (1 - \lambda) \sum_k \lambda^k \HH^d(\abs*{K_{k+1}^d}) - \HH^d(\abs*{K_0^d})    &\leq C\sum_k \LL^n(U_{k+1} \setminus U_k)\\
                                                                                        &\leq C.
    \end{align}
    Now, we justify (\ref{observation0}). It is straightforward that $\abs*{K_k^d} \subset \abs*{K_{k+1}^d}$ and that for $x \in \abs*{K_{k+1}^d} \setminus \abs*{K_k^d}$, there exists $A \in K_{k+1}^d \setminus K_k$ such that $x \in A$. We have thus
    \begin{equation}
        \HH^d(\abs*{K_{k+1}^d}) - \HH^d(\abs*{K_k^d}) \leq C \sum_{A \in K_{k+1}^d \setminus K_k} \mathrm{diam}(A)^d.
    \end{equation}
    By the choice $\mu^{n-d} = \lambda$, the fact that $\mu^k \leq C 2^{-q(k+1)}$ (see (\ref{q1.4})) and the fact that the cells of $K_{k+1}$ have sidelength $2^{-q(k+1)}$, we have then
    \begin{equation}
        \lambda^k \left(\HH^d(\abs*{K_{k+1}^d}) - \HH^d(\abs*{K_k^d})\right)  \leq C\sum_{A \in K_{k+1}^d \setminus K_k} \mathrm{diam}(A)^n
    \end{equation}
    The cardinal of $K_{k+1}^d \setminus K_k$ is comparable to the cardinal of $K_{k+1}^n \setminus K_k$ (modulo a multiplicative constant that depends on $n$) so finally
    \begin{equation}
        \lambda^k \left(\HH^d(\abs*{K_{k+1}^d}) - \HH^d(\abs*{K_k^d})\right) \leq C\LL^n(U_{k+1} \setminus U_k).
    \end{equation}

    \emph{Step 2. We prove that}
    \begin{equation}
        \HH^d(E \cap [-1,1]^n) \geq C^{-1}.
    \end{equation}
    We will proceed by contradiction. We fix $0 < \mu < 1$ close enough to $1$ (this will be precised later). We aim to apply the Federer--Fleming projection in a sequence of complexes $(K_k)_{k \in \N}$ whose supports is of the form
    \begin{equation}
        \abs{K_k} = (\tfrac{1}{2} + \mu^k) [-1,1]^n.
    \end{equation}
    We also want the complexes $K_k$ to be composed of dyadic cells so that $K_k \preceq E_n$. The supports of $(K_k)_{k \in \N}$ will be in fact $(1 - \sum_{i < k} 2^{-q(i)}) [-1,1]^n$ where $(q(k))_{k \in \N}$ is as a sequence of non-negative integer such that, in some sense,
    \begin{equation}
        \frac{1}{2} - \sum_{i=0}^{k-1} 2^{-q(i)} \sim \mu^k.
    \end{equation}
    We define $(q(k))_{k \in \N}$ as in step 1.

    For each $k$, let $K_k$ be the set of dyadic cells of sidelength $2^{-q(k)}$ subdivising the cube $(1 - \sum_{i < k} 2^{-q(i)}) [-1,1]^n$ but not lying on its boundary. The set $K_k$ is a finite $n$-complex subordinated to $E_n$ and
    \begin{subequations}\label{K2}
        \begin{align}
            \abs{K_k}   &   =(1 - \sum_{i < k} 2^{-q(i)}) [-1,1]^n,\\
            U(K_k)      &   =(1 - \sum_{i < k} 2^{-q(i)}) \mathopen{]}-1,1\mathclose{[}^n.
        \end{align}
    \end{subequations}
    We abbreviate $U(K_k)$ as $U_k$; in particular $U_0 = \mathopen{]}-1,1\mathclose{[}^n$. We also define
    \begin{equation}
        U_\infty = \bigcap_k U_k = [-\tfrac{1}{2},\tfrac{1}{2}]^n.
    \end{equation}
    The cells of $K_k$ have bounded overlap: each point of $\abs{K_k}$ belongs to at most $3^n$ cells $A \in K_k$. We observe that $K_{k+1}$ is subordinate to $K_k$. Indeed, $K_k$ and $K_{k+1}$ are composed of dyadic cells, $U_{k+1} \subset U_k$ and the cells of $K_{k+1}$ have a sidelength which is strictly less than those of $K_k$. We also observe that there exists a subcomplex $L$ of $K_k$ such that $\abs{L} = \abs{K_{k+1}}$ and $U(L) = U_{k+1}$. We deduce as in step 1 that
    \begin{equation}
        \abs{K_k} \setminus U_{k+1} = \bigcup \set{A \in K_k | A \cap U_{k+1} = \emptyset}.
    \end{equation}

    Let $\phi$ be a Federer--Fleming projection of $E \cap U_k$ in $K_k$. By the properties of the Federer--Fleming projection, there exists $C_0 \geq 1$ (depending on $n$) such that
    \begin{equation}
        \HH^d(\phi(E \cap U_k)) \leq C_0 \HH^d(E \cap U_k).
    \end{equation}
    As a cell $A \in K_k$ has sidelength $2^{-q(k)}$, its area $\HH^d(\tfrac{1}{2}A)$ is of the form $C 2^{-dq(k)}$. We deduce that there exists $C_1 > 0$ (depending only on $n$) such that
    \begin{equation}
        \HH^d(E \cap U_k) \leq C_1^{-1} 2^{-dq(k)} \implies \HH^d(\phi(E \cap U_k)) < \HH^d(\tfrac{1}{2}A).
    \end{equation}
    In this case, the set $\phi(E \cap U_k)$ does not include $\frac{1}{2}A$. It can be postcomposed with a radial projection centered in $\tfrac{1}{2}A$ and sent to $\partial A$. If $\HH^d(E \cap U_k) \leq C_1^{-1} 2^{-dq(k)}$, we can thus assume that for all $A \in K_k^d$, $\HH^d(\phi(E \cap U_k) \cap A) = 0$. That being done, we apply the quasiminimality condition (\ref{E_quasi}) with respect to $\phi$ in $B=B(0,\sqrt{n})$. We also assume $h \leq \tfrac{1}{2}$ so that
    \begin{equation}
        \HH^d(E \cap U_k)) \leq C \HH^d(\phi(E \cap U_k)).
    \end{equation}
    We decompose $E \cap U_k$ in two parts. The points of $E \cap U_{k+1}$ are sent in the $d$-dimensional skeleton of $K_k$ (as in step 1) and then radially projected in the $(d-1)$-skeleton so their image is $\HH^d$-negligible. On the other hand, we have
    \begin{equation}\label{FF2}
        \HH^d(\phi(E \cap U_k \setminus U_{k+1})) \leq C \HH^d(E \cap U_k \setminus U_{k+1}).
    \end{equation}
    This is deduced for the fact that 
    \begin{equation}
        U_k \setminus U_{k+1} = \bigcup \set{U_k \cap A | A \in K_k,\ A \cap U_{k+1} = \emptyset} ,
    \end{equation}
    that for all $A \in K_k$, $\HH^d(\phi(E \cap U_k \cap A)) \leq C\HH^d(E \cap U_k \cap A)$ and that the cells of $K_k$ have bounded overlap. In conclusion,
    \begin{equation}
        \HH^d(E \cap U_k) \leq C \HH^d(E \cap U_k \setminus U_{k+1}).
    \end{equation}
    We rewrite this inequality as
    \begin{equation}
        \HH^d(E \cap U_{k+1}) - \lambda \HH^d(E \cap U_k) \leq 0.
    \end{equation}
    where $\lambda = C^{-1}(C - 1)$. Multiplying this inequality by $\lambda^{-k}$, we obtain a telescopic term:
    \begin{equation}\label{FF3}
        \lambda^{-(k+1)}\HH^d(E \cap U_{k+1}) - \lambda^{-k} \HH^d(E \cap U_k) \leq 0.
    \end{equation}
    We choose $\mu$ such that $\mu^d = \lambda$. Next, we show that there exists a constant $C_3 \geq 1$ (depending on $n$, $\kappa$) such that if $\HH^d(E \cap U_0) \leq C_3^{-1}$, then (\ref{FF3}) holds for all $k \in \N$. The idea is to observe that if (\ref{FF3}) holds for $i=0,\ldots,k-1$, we can sum this telescopic inequality and obtain
    \begin{equation}\label{FF5}
        \lambda^{-k}\HH^d(E \cap U_k) - \HH^d(E \cap U_0) \leq 0.
    \end{equation}
    so
    \begin{equation}
        \HH^d(E \cap U_k) \leq \lambda^k \HH^d(E \cap U_0).
    \end{equation}
    According to the choice $\mu^d = \lambda$ and (\ref{q1.5}), there exists $C_2 \geq 1$ (depending on $n$, $\kappa$) such that for all $k \in \N$, $2^{-dq(k)} \geq C_2^{-1} \lambda^k$. We conclude that if $\HH^d(E \cap U_0) \leq (C_1 C_2)^{-1}$, then
    \begin{equation}
        \HH^d(E \cap U_k) \leq C_1^{-1} 2^{-dq(k)}
    \end{equation}
    and the process can be iterated. However, taking the limit $k \to \infty$ in (\ref{FF5}) yields a contradiction because $0 < \lambda < 1$ and $0 < \HH^d(E \cap U_0) \leq \HH^d(E \cap U_\infty) < \infty$.
\end{proof}

We are going to use our new estimate for the Federer--Fleming projection (\ref{FF_estimate2}) to prove the rectifiability of quasiminimal sets.
\begin{prop}\label{prop_rect}
Fix $\kappa \geq 1$, $h > 0$ and a scale $s \in ]0,\infty]$. Assume that $h$ is small enough (depending on $n$, $\GL$, $\IL$). Let $E$ be a $(\kappa,h,s)$-quasiminimal set with respect to $\II$ in $X$. Then $E$ is $\HH^d$ rectifiable.
\end{prop}
\begin{proof}
    The letter $C$ plays the role of a constant $\geq 1$ that depends on $n$, $\kappa$, $\GL$, $\IL$. Its value can increase from one line to another (but a finite number of times). We assume, without loss of generality, that $E = E^*$.

    We decompose $E$ in two disjoint $\HH^d$ measurable parts: a $d$-rectifiable part $E_r$ and a purely $d$-unrectifiable part $E_u$. We are going to prove that there exists $\lambda \geq 1$ such that for all $x \in E$, for all small $r > 0$,
    \begin{equation}\label{af_rect}
        \HH^d(E \cap B(x,r)) \leq C \HH^d(E_r \cap B(x, \lambda r)).
    \end{equation}
    Let us explain why this this fulfills the lemma. The set $E_r$ is $\HH^d$ measurable and $\HH^d$-locally finite so for $\HH^d$-a.e. $x \in \R^n \setminus E_r$,
    \begin{equation}
        \lim r^{-d} \HH^d(B(x,r) \cap E_r) = 0.
    \end{equation}
    See for example \cite[Theorem 6.2(2)]{Mattila}. We combine this property with (\ref{af_rect}) to deduce that for $\HH^d$-a.e. $x \in E_u$,
    \begin{equation}
        \lim r^{-d} \HH^d(B(x,r) \cap E) = 0.
    \end{equation}
    This contradicts the Ahlfors-regularity of $E$ so $E_u = \emptyset$.

    We fix $x \in E$ and $0 < r \leq \min \set{r_s(x),r_t(x),r_u}$ where $u$ is the scale given by Proposition \ref{prop_density}. By quasiminimality, we know that for all sliding deformations $f$ of $E$ in $B(x,r)$,
    \begin{equation}
        \II(E \cap W_f) \leq \kappa \II(f(E \cap W_f)) + h \II(E \cap B(x,hr)),
    \end{equation}
    By Ahlfors-regularity, we also know that for all $0 < \lambda \leq 1$,
    \begin{equation}
        \HH^d(E \cap B(x,r)) \leq C \lambda^{-d} \HH^d(E \cap B(x,\lambda r)).
    \end{equation}
    We can simplify the setting as we did in Proposition \ref{prop_density}. It suffices to solve the following problem. Let the ball $B(0,\sqrt{n})$ be denoted by $B$. We assume that $E$ is a $\HH^d$ finite closed subset of $B$ and
    \begin{enumerate}
        \item $\HH^d(E \cap B) \leq C \HH^d(E \cap B(0,\tfrac{1}{2}))$;
        \item there exists $S \subset \mathcal{E}_n(1)$ such that $\Gamma \cap B = \abs{S} \cap B$;
        \item for all sliding deformations $f$ of $E$ along $\Gamma$ in $B$,
            \begin{equation}\label{E_rect_quasi}
                \HH^d(E \cap W_f) \leq \kappa \HH^d(f(E \cap W_f)) + h \HH^d(E \cap B(0,\tfrac{1}{2}));
            \end{equation}
    \end{enumerate}
    Then we show that
    \begin{equation}\label{rect_goal}
        \HH^d(E \cap \mathopen{]}-\tfrac{1}{2},\tfrac{1}{2}\mathclose{[}^n) \leq C \int_{G(d,n)} \HH^d(p_V(E \cap \mathopen{]}-1,1\mathclose{[}^n)) \, \mathrm{d}V.
    \end{equation}
    Thanks to the Besicovitch--Federer projection theorem \cite[Theorem 18.1]{Mattila}, the right-hand side cancels the purely $d$-unrectifiable part of $E$.

    We fix $q \in \N^*$. For $0 \leq k < 2^q$, let $K_k$ be the set of dyadic cells of sidelength $2^{-q}$ subdivising the cube $(1 - k2^{-q}) [-1,1]^n$ but which are not lying in its boundary. The set $K_k$ is a finite $n$-complex subordinated to $E_n$ and
    \begin{subequations}
        \begin{align}
            \abs{K_k}   &   = (1 - k2^{-q}) [-1,1]^n,\\
            U(K_k)      &   = (1 - k2^{-q}) \mathopen{]}-1,1\mathclose{[}^n.
        \end{align}
    \end{subequations}
    We abbreviate $U(K_k)$ as $U_k$; in particular $U_0 = \mathopen{]}-1,1\mathclose{[}^n$. The cells of $K_k$ have bounded overlap: each point of $\abs{K_k}$ belongs to at most $3^n$ cells $A \in K_k$. It is clear that $K_{k+1}$ is a subcomplex of $K_k$ (the sidelength $2^{-q}$ does not depends on $k$). We deduce that
    \begin{align}
        \abs{K_k} \setminus U_{k+1} &= \bigcup \set{A \in K_k | A \in K_k \setminus K_{k+1}}\\
                                    &= \bigcup \set{A \in K_k | A \cap U_k \ne \emptyset}.
    \end{align}

    Let $\phi$ be a Federer--Fleming projection of $E \cap U_k$ in $K_k$. We apply the quasiminimality of $E$ with respect to $\phi$ in $B = B(0,\sqrt{n})$. We also assume $h \leq \tfrac{1}{2}$ so that
    \begin{equation}
        \HH^d(E \cap \mathopen{]}-\tfrac{1}{2},\tfrac{1}{2}\mathclose{[}^n)) \leq C \HH^d(\phi(E \cap U_k)).
    \end{equation}
    We decompose $E \cap U_k$ in two parts: $E \cap U_{k+1}$ and $E \cap U_k \setminus U_{k+1}$. First, we have (as in step $1$ of Proposition \ref{prop_density}) ,
    \begin{equation}
        \phi(E \cap U_{k+1}) \subset \bigcup\set{A \in K_k | \mathrm{dim} \, A \leq d}.
    \end{equation}
    By the properties of Federer--Fleming projections, we have for $A \in K_k^d$,
    \begin{equation}
        \HH^d(\phi(E) \cap A) \leq C \int_{G(d,n)} \HH^d(p_V(E \cap U_0)) \, \mathrm{d}V
    \end{equation}
    and since $K_k^d$ contains at most $C 2^q$ cells,
    \begin{equation}
        \HH^d(\phi(E \cap U_{k+1})) \leq C 2^{qd} \int_{G(d,n)} \HH^d(p_V(E \cap U_0)) \, \mathrm{d}V.
    \end{equation}
    Next, we claim that
    \begin{equation}
        \HH^d(\phi(E \cap U_k \setminus U_{k+1})) \leq C \HH^d(E \cap U_{k+1} \setminus U_k).
    \end{equation}
    This is deduced from the fact that
    \begin{equation}
        U_k \setminus U_{k+1} = \bigcup \set{U_k \cap A | A \in K_k,\ A \cap U_{k+1} = \emptyset},
    \end{equation}
    that for all $A \in K_k$, $\HH^d(\phi(E \cap U_k \cap A)) \leq C\HH^d(E \cap U_k \cap A)$ and that the cells of $K_k$ have bounded overlap. In sum
    \begin{multline}\label{FF3a}
        \HH^d(E \cap \mathopen{]}-\tfrac{1}{2},\tfrac{1}{2}\mathclose{[}^n)) \leq C 2^q \int_{G(d,n)} \HH^d(p_V(E \cap U_0)) \, \mathrm{d}V\\ + C \HH^d(E \cap U_k \setminus U_{k+1}).
    \end{multline}
    We are going to use a Chebyshev argument to find an index $0 \leq k < 2^q$ (depending on $n, \kappa$) such that $\frac{1}{2}\mathopen{]}-1,1\mathclose{[}^n \subset U_k$ and
    \begin{equation}
        \HH^d(E \cap U_k \setminus U_{k+1}) \leq \tfrac{1}{2} C^{-1} \HH^d(E \cap \mathopen{]}-\tfrac{1}{2},\tfrac{1}{2}\mathclose{[}^n).
    \end{equation}
    The sets $E \cap U_k \setminus U_{k+1}$ are disjoint so
    \begin{equation}
        \sum_k \HH^d(E \cap U_k \setminus U_{k+1}) \leq \HH^d(E \cap U_0).
    \end{equation}
    As $U_k = (1 - k2^{-q}) \mathopen{]}-1,1\mathclose{[}^n$, there exist at least $2^{q-1}$ index $0 \leq k < 2^q$ such that $\mathopen{]}-\tfrac{1}{2},\tfrac{1}{2}\mathclose{[}^n \subset U_k$ so there exists an index $k$ such that $\mathopen{]}-\tfrac{1}{2},\tfrac{1}{2}\mathclose{[}^n \subset U_k$ and
    \begin{equation}
        2^{q-1} \HH^d(E \cap U_k \setminus U_{k+1}) \leq \HH^d(E \cap U_0).
    \end{equation}
    As $\HH^d(E \cap B) \leq C \HH^d(E \cap B(0,\tfrac{1}{2}))$, we then have
    \begin{equation}
        \HH^d(E \cap U_k \setminus U_{k+1}) \leq C 2^{-q} \HH^d(E \cap \mathopen{]}-\tfrac{1}{2},\tfrac{1}{2}\mathclose{[}^n).
    \end{equation}
    so we can choose $q$ big enough (depending on $n$, $\kappa$) such that
    \begin{equation}
        \HH^d(E \cap U_k \setminus U_{k+1}) \leq \tfrac{1}{2} C^{-1} \HH^d(E \cap \mathopen{]}-\tfrac{1}{2},\tfrac{1}{2}\mathclose{[}^n).
    \end{equation}
    Now, (\ref{FF3a}) implies
    \begin{equation}
        \HH^d(E \cap \mathopen{]}-\tfrac{1}{2},\tfrac{1}{2}\mathclose{[}^n) \leq C \int_{G(d,n)} \HH^d(p_V(E \cap U_0)) \, \mathrm{d}V.
    \end{equation}
    The constant $2^q$ has been absorbed in $C$ because $q$ depends now on $n$, $\kappa$.
\end{proof}

\subsection{Weak Limits of Quasiminimizing Sequences}
We prove that a weak limit of a quasiminimizing sequence is a quasiminimal set. Our working space is an open set $X$ of $\R^n$.
\begin{thm}[Limiting Theorem]\label{thm_limit}
    Fix a Whitney subset $\Gamma$ of $X$. Fix an admissible energy $\II$ which is induced by a continuous integrand. Fix $\kappa \geq 1$, $h > 0$ and a scale $s \in ]0,\infty]$. Assume that $h$ is small enough (depending on $n$, $\Gamma$, $\II$). Let $(E_i)$ be a sequence of closed, $\HH^d$ locally finite subsets of $X$ satisfying the following conditions:
    \begin{enumerate}[label=(\roman*)]
        \item the sequence of Radon measures $(\II \mres E_i)$ has a weak limit $\mu$ in $X$;
        \item for all open balls $B$ of scale $\leq s$ in $X$, there exists a sequence $(\varepsilon_i) \to 0$ such that for all global sliding deformations $f$ in $B$,
            \begin{equation}
                \II(E_i \cap W_f) \leq \kappa \II(f(E_i \cap W_f)) + h \II(E_i \cap hB) + \varepsilon_i
            \end{equation}
    \end{enumerate}
    Then we have
    \begin{equation}
        \II \mres E \leq \mu \leq \ka \II \mres E,
    \end{equation}
    where $E = \spt(\mu)$ and $\ka = \kappa + h$. The set $E$ is $(\kappa,\ka h,s)$-quasiminimal with respect to $\II$ in $X$ and in particular, $\HH^d$ rectifiable.
\end{thm}
\begin{rmk}\label{rmk_limit}
    Let us assume that $(E_i)$ satisfies an additional condition:
    \begin{enumerate}[label=(\roman*)]
    \setcounter{enumi}{2}
        \item there exists $\kb \geq 1$ such that for open balls $B$ of scale $\leq s$ in $X$, there exists a sequence $(\varepsilon_i) \to 0$ such that for all global sliding deformations $f$ in $B$,
    \begin{equation}
        \II(E_i \cap B) \leq \kb \II(f(E_i \cap B)) + \varepsilon_i.
    \end{equation}
    \end{enumerate}
    Then the theorem holds true even if $\II$ is not explicitly induced by a continuous integrand (as in Definition \ref{defi_energy}). More precisely, we have
    \begin{equation}
        \II \mres E \leq \mu \leq \kb \II \mres E,
    \end{equation}
    where $E = \spt(\mu)$. The set $E$ is $(\kappa,\kb h,s)$-quasiminimal with respect to $\II$ in $X$ and in particular, $\HH^d$ rectifiable. This will be justified during step 4 of the proof.
\end{rmk}
Theorem \ref{thm_limit} is proved by constructing relevant sliding deformations. We distinguish three intermediate result:
\begin{enumerate}
    \item \emph{The limit measure $\mu$ is locally Ahlfors-regular and rectifiable of dimension $d$.} We adapt the techniques of Propositions \ref{prop_density} and \ref{prop_rect}. It seems that an error term of the form $h \mathrm{diam}(B)^d + \varepsilon_i$ would pose a problem to the lower Ahlfors-regularity. However, an error term $h \mathrm{diam}(B)^d$ (without $\varepsilon_i$) should be fine. In this case the sets $E_i$ are quasiminimals for a slight variation of our definition \cite[Definition 2.3]{Sliding}. As a consequence, they are locally Ahlfors-regular \cite[Proposition 4.74]{Sliding} and then $h \mathrm{diam}(B)^d \approx \bar{h} \HH^d(E_i \cap \bar{h} B)$, where $h = \bar{h}^{d+1}$. The choice of the error term is not so significant in the rest of the proof.
    \item \emph{For all sliding deformation $f$ of $E$ in small balls $B$, we have}
        \begin{equation}
            \mu(W_f) \leq \kappa \II(f(W_f)) + h \mu(h B).
        \end{equation}
        This part follows from a technical lemma which is inspired by the techniques of David in \cite{Sliding}. The lemma was not conceptualized in \cite{Sliding} and allows significant simplifications. We use this quasiminimality condition to deduce that $\mu \leq (\kappa + h)\II \mres E$.
    \item \emph{We have $\mu \geq \II \mres E$.} We make use of an argument introduced by Fang in \cite{FangReifenberg}. It bypasses the concentration Lemma of Dal Maso, Morel and Solimini \cite{DMS} when the limit is already known to be rectifiable.
\end{enumerate}

\subsubsection{Technical Lemma}\label{technical_subsection}
Our ambient space is an open set $X$ of $\R^n$. We work with an admissible energy $\II$ in $X$ (Definition \ref{defi_energy}). Given a map $f$, the symbol $\norm{f}_L$ denotes the Lipschitz constant of $f$ (possibly $\infty$). The assumptions on the boundary are detailed in Definitions \ref{retract_definition} and \ref{cone_definition}.
\begin{lem}\label{technical_lemma}
    Fix $\Gamma$ a Lipschitz neighborhood retract in $X$ which is locally diffeomorphic to a cone. Let $f$ be a global sliding deformation in an open set $U \subset X$. Let $W \subset U$ be an open set and let $E \subset W$ be a $\HH^d$ measurable, $\HH^d$ finite and $\HH^d$ rectifiable set. For all $\varepsilon > 0$, there exists a global sliding deformation $g$ in $U$ and an open set $V \subset W$ such that $g - f$ has a compact support included in $W$, $\abs{g - f} \leq \varepsilon$, $\norm{g}_L \leq C \norm{f}_L$ (where $C \geq 1$ depends on $n$, $\Gamma$) and
    \begin{subequations}
        \begin{align}
            &   \HH^d(E \setminus V) \leq \varepsilon\\
            &   \II(g(V))          \leq \II(f(E)) + \varepsilon.
        \end{align}
    \end{subequations}
\end{lem}
Even with $\Gamma = \emptyset$, the lemma is not trivial. Roughly speaking, $g$ smashes an almost neighborhood $V$ of $E$ onto an approximation of $f(E)$. One can think of $g$ as a composition $p \circ f$ where $p$ is Lipschitz retraction onto $f(E)$. In practice, such a retraction may not exist. We find disjoint balls $(B_i)_i$ such that $\HH^d(f(E) \setminus \bigcup_i B_i) \leq \varepsilon$ and such that in each ball $B_i$, $f(E)$ is well approximated by a plane $V_i$. Then we define $p$ in each ball $B_i$ to be the orthogonal projection onto $V_i$ and finally $g = p \circ f$ and $V = f^{-1}(\bigcup B_i)$. The lemma requires further work because this construction does not suffice to control $\HH^d(E \setminus V)$.
\begin{proof}
    The letter $C$ plays the role of a constant $\geq 1$ that depends on $n$, $\Gamma$ and $\II$. Its value can increase from one line to another (but a finite number of times.) For $x \in \R^n$, for $V$ a plane passing through $x$ and for $0 < \varepsilon \leq 1$, we introduce the open cone
    \begin{equation}
        C(x,V,\varepsilon) = \Set{z \in \R^n | \mathrm{d}(z, V) < \varepsilon \abs{z - x}}.
    \end{equation}
    We take the convention that the $\HH^d$ measure of a $d$-dimensional disk of radius $r \geq 0$ equals $(2r)^d$.

    Let us get rid of the case $\norm{f}_L = 0$, that is, $f$ constant. In this case, we take $g = f$ and it is left to define $V$. If $E \ne \emptyset$, then $f(W) = f(E)$ so we can take $V = W$. If $E = \emptyset$, we take $V = \emptyset$. From now on, we assume $\norm{f}_L > 0$. It is more pleasant to work with maps defined over $\R^n$ so as not to worry about the domains of definition when composing functions. We apply McShane's extension Lemma (Lemma \ref{lipschitz_extension}) to extend $f$ as a Lipschitz map $\R^n \to \R^n$. Its new Lipschitz constant has been multiplied by $C$ at most so it can be denoted by $\norm{f}_L$ without altering the lemma. We also observe that it suffices to prove the lemma whenever $W \subset \subset U$. Indeed, in the general case, there exists an open set $W' \subset \subset W$ such that $H^d(E \setminus W') \leq \varepsilon$ and it suffices to apply the lemma to $E' = E \cap W'$. Thus, we assume $E \subset W \subset \subset U$. This property will only be used at the end of the proof.

    Next, we justify that the measures
    \begin{equation}
        \mu\colon A \to \HH^d(E \cap f^{-1}(A))
    \end{equation}
    and
    \begin{equation}
        \lambda\colon A \to \HH^d (f(E) \cap A)
    \end{equation}
    are Radon measures in $\R^n$. We recall that a finite Borel regular measure in $\R^n$ is a Radon measure \cite[Corollary 1.11]{Mattila}. The measure is finite because $\HH^d(f(E)) \leq \norm{f}_L^d \HH^d(E) < \infty$. It is also Borel regular because $f(E)$ is $\HH^d$ measurable. This proves that $\lambda$ is a Radon measure. It is clear that $\mu$ is a finite measure. We work a bit more to prove that $\mu$ is Borel regular. We have to show that for all $A \subset \R^n$, there exists a Borel set $C$ containing $A$ such that $\mu(A) = \mu(C)$, that is, $\HH^d(E \cap f^{-1}(A)) = \HH^d(E \cap f^{-1}(C))$. Let $O$ be an open set of $\R^n$ containing $f^{-1}(A)$. Then the set
    \begin{align}
        B   &   = \set{y \in \R^n | f^{-1}(y) \subset O}\\
            &   = \R^n \setminus f(\R^n \setminus O)
    \end{align}
    is a Borel set (because $\R^n \setminus O$ is $\sigma$-compact), it contains $A$ and
    \begin{equation}
        \HH^d(E \cap f^{-1}(B)) \leq \HH^d(E \cap O).
    \end{equation}
    By regularity of the Radon measure $\HH^d \mres E$, there exists a sequence of open sets $(O_k)$ containing $f^{-1}(A)$ such that
    \begin{equation}
        \lim_k \HH^d(E \cap O_k) = \HH^d(E \cap f^{-1}(A)).
    \end{equation}
    To each set $O_k$, we associate a Borel set $B_k$ containing $A$ and such that $\HH^d(E \cap f^{-1}(B_k)) \leq \HH^d(E \cap O_k)$. Then $C = \bigcap_k B_k$ is the solution.

    We are going to differentiate $\mu$ with respect to $\lambda$. For $t > 0$, we introduce $Y_t$: the set of points $y \in \R^n$ such that
    \begin{equation}
        \liminf\limits_{r \to 0} \frac{\HH^d(E \cap f^{-1}(B(y,r)))}{\HH^d(f(E) \cap B(y,r))} \leq t.
    \end{equation}
    The set $Y_t$ is $\HH^d$ measurable (see \cite[Remark 2.10]{Mattila}). According to the Differentiation Lemma \cite[Lemma 2.13]{Mattila}, we have for all $A \subset Y_t$,
    \begin{equation}
        \HH^d(E \cap f^{-1}(A)) \leq t \HH^d(f(E) \cap A).
    \end{equation}
    and for all $A \subset \R^n \setminus Y_t$,
    \begin{equation}
        \HH^d(E \cap f^{-1}(A)) \geq t \HH^d(f(E) \cap A).
    \end{equation}
    We define $E_1 = E \cap f^{-1}(Y_t)$ and $E_2 = E \setminus f^{-1}(Y_t)$. Then for all $A \subset \R^n$,
    \begin{equation}
        \HH^d(E_1 \cap f^{-1}(A)) \leq t\HH^d(f(E) \cap A)\\
    \end{equation}
    and
    \begin{equation}\label{diff2}
        \HH^d(f(E_2)) \leq t^{-1} \HH^d(E).
    \end{equation}
    Now, we fix a constant $\varepsilon_0 > 0$. We will deal independently with $E_1$ and $E_2$ in step 1 and step 2 respectively. The constant $t$ is fixed in step 1 whereas it is chosen in step 2.

    \emph{Step 0.} \emph{We build two open sets $W_1, W_2 \subset W$ such that $\overline{W_1} \cap \overline{W_2} = \emptyset$ and}
    \begin{equation}\label{W_i}
        \HH^d(E_i \setminus W_i) \leq \varepsilon_0.
    \end{equation}
    According to the Approximation Theorem \cite[Theorem 1.10]{Mattila}, there exist two compact sets $K_1, K_2 \subset W$ and two open sets $O_1, O_2 \subset W$ such that $K_i \subset E_i \subset O_i$ and
    \begin{equation}
        \HH^d \mres E (O_i \setminus K_i) \leq \varepsilon_0.
    \end{equation}
    Since $K_1$ and $K_2$ are disjoint compact sets and $K_i \subset O_i$, there exists two open sets $W_1, W_2$ such that $K_i \subset W_i \subset O_i$ and $\overline{W_1} \cap \overline{W_2} = \emptyset$. As $E_i \subset O_i$, we deduce
    \begin{equation}
        \HH^d(E_i \setminus W_i) \leq \HH^d(E \cap O_i \setminus K_i) \leq \varepsilon_0.
    \end{equation}

    \emph{Step 1.} \emph{Whichever is $t$, we build a Lipschitz map $g_1\colon \R^n \to \R^n$ such $g_1(\Gamma) \subset \Gamma$, $g_1 - f$ has a compact support included in $W_1$, $\abs{g_1 - f} \leq \varepsilon_0$, $\norm{g_1}_L \leq C\norm{f}_L$ and such that there is an open set $V_1 \subset W_1$ satisfying}
    \begin{subequations}
        \begin{align}
            &   \HH^d(E_1 \setminus V_1) \leq \varepsilon_0,\\
            &   \II(g_1(V_1)) \leq \II(f(E)) + \varepsilon_0.
        \end{align}
    \end{subequations}
    Our constructions will rely on intermediate variables $0 < \delta, \varepsilon \leq 1$. These variables will be as small as we want but the choice of $\varepsilon$ will be subordinated to the choice of $\delta$.

    We start by introducing a Lipschitz retraction onto the boundary. The set $\Gamma$ is a Lipschitz neighborhood retract so there exists an open set $O_\Gamma \subset X$ containing $\Gamma$ and a $C$-Lipschitz map $p\colon O_\Gamma \to \Gamma$ such that $p = \mathrm{id}$ on $\Gamma$. For $\delta > 0$, we define
    \begin{equation}
        \Gamma(\delta) = \set{x \in X | \mathrm{d}(x,\Gamma) < \delta}.
    \end{equation}
    As $H^d \mres f(E)$ is a finite measure, there exists $\delta > 0$ such that
    \begin{equation}\label{0delta}
        H^d(f(E) \cap \Gamma(2\delta) \setminus \Gamma) \leq t^{-1}\varepsilon_0.
    \end{equation}
    We restrict $O_\Gamma$ so that $O_\Gamma \subset \Gamma(\delta)$ and $\abs{p - \mathrm{id}} \leq \delta$ on $O_\Gamma$. Note that $O_\Gamma$ does not depend on the intermediate variable $\varepsilon$. We extend $p$ in $X \setminus \Gamma(\delta)$ by $p = \mathrm{id}$. Let us check the Lipschitz constant of this extension. For $x \in O_\Gamma$ and $y \in X \setminus \Gamma(\delta)$, we have $\abs{x - y} \geq \delta$ so
    \begin{align}
        \abs{p(x) - p(y)}   &   \leq \abs{p(x) - y}\\
                            &   \leq \abs{p(x) - x} + \abs{x - y}\\
                            &   \leq \delta + \abs{x - y}\\
                            &   \leq 2\abs{x - y}.
    \end{align}
    Thus $p$ is still $C$-Lipschitz. Now, we use McShane's extension Lemma (Lemma \ref{lipschitz_extension}) to extend $p$ as a $C$-Lipschitz map $p\colon \R^n \to \R^n$ such that $\abs{p - \mathrm{id}} \leq \delta$.

    Next, we recall some classical properties of rectifiable sets (see \cite[Theorem 4.8]{DL}). The set $f(E)$ is $\HH^d$ measurable, $\HH^d$ finite and $\HH^d$ rectifiable so for $\HH^d$-ae. $x \in f(E)$, there exists a unique $d$-plane $V_x$ passing through $x$ such that
    \begin{equation}\label{0mu_V}
        (2r)^{-d} \lambda_{x,r} \rightharpoonup \HH^d \mres V_x,
    \end{equation}
    where $\lambda = \HH^d \mres f(E)$, $\lambda_{x,r}\colon A \mapsto \lambda(x + r(A - x))$ and the arrow $\rightharpoonup$ denotes the weak convergence of Radon measures as $r \to 0$. In this case, we say that $V_x$ is an approximate tangent plane of $f(E)$ at $x$. Fix a point $x$ satisfying (\ref{0mu_V}). Then, for all $0 < \varepsilon \leq 1$,
    \begin{equation}
        \lim\limits_{r \to 0} r^{-d} \HH^d(f(E) \cap B(x,r) \setminus C(x,V_x,\varepsilon)) = 0
    \end{equation}
    and for $0 < a < 1$,
    \begin{equation}\label{Y0}
        \limsup\limits_{r \to 0} r^{-d} \HH^d(f(E) \cap B(x,r) \setminus a B(x,r)) \leq C(1 - a).
    \end{equation}
    Moreover, the definition of admissible energies (Definition \ref{defi_energy}) implies that for $\HH^d$-ae. $x \in f(E)$,
    \begin{equation}
        \lim\limits_{r \to 0} \frac{\II(f(E) \cap B(x,r))}{\II(V_x \cap B(x,r))} = 1.
    \end{equation}
    We need additional properties for the points $x \in f(E) \cap \Gamma$. They will help us to preserve the boundary. For $\HH^d$-ae. $x \in f(E) \cap \Gamma$, 
    \begin{equation}\label{0mu_gamma}
        \lim_{r \to 0} r^{-d} \HH^d(f(E) \cap B(x,r) \setminus \Gamma) = 0.
    \end{equation}
    This can be justified by applying \cite[Theorem 6.2(2)]{Mattila} to the set $f(E) \setminus \Gamma$. Let us fix a point $x \in f(E) \cap \Gamma$ which satisfies (\ref{0mu_V}) and (\ref{0mu_gamma}). We show that $V=V_x$ is tangent to $\Gamma$. More precisely, for all $0 < \varepsilon \leq 1$, there exists $r > 0$ such that
    \begin{equation}\label{0V_gamma}
        V \cap B(x,r) \subset \set{y \in X | \mathrm{d}(y,\Gamma) \leq \varepsilon \abs{y - x}}.
    \end{equation}
    If the statement is not true, there exists $\varepsilon > 0$ and a sequence $(y_k) \in V \cap X$ such that $y_k \to x$ and $\mathrm{d}(y_k,\Gamma) > \varepsilon \abs{y_k - x}$. As $x \in \Gamma$ and $\mathrm{d}(y_k, \Gamma) > 0$, we necessarily have $y_k \ne x$. Let $r_k = \abs{y_k - x}$ and $\hat{y_k} = x + r_k^{-1}(y_k - x)$. Observe that $\hat{y_k} \in V \cap S(x,1)$ (where $S(x,1)$ is the unit sphere of centered at $x$) and
    \begin{equation}
        y_k = x + r_k(\hat{y_k} - x).
    \end{equation}
    By compactness, we can assume that $\hat{y_k} \to \hat{z} \in V \cap S(x,1)$. Then we consider the sequence $(z_k)$ defined by
    \begin{equation}\label{z_k_hat}
        z_k = x + r_k(\hat{z} - x).
    \end{equation}
    According to (\ref{z_k_hat}) and the weak convergence (\ref{0mu_V}),
    \begin{equation}\label{csq_V}
        \lim_k (2r_k)^{-d} \lambda(B(z_k,\tfrac{1}{2}\varepsilon r_k)) = \HH^d(V \cap B(\hat{z},\tfrac{1}{2}\varepsilon)) > 0.
    \end{equation}
    On the other hand, we have $\abs{y_k - z_k} \leq r_k \abs{\hat{y_k} - \hat{z}}$ so, whenever $k$ is big enough such that $\abs{\hat{y_k} - \hat{z}} \leq \frac{1}{2}\varepsilon$,
    \begin{align}
        \mathrm{d}(z_k,\Gamma)  &\geq \mathrm{d}(y_k,\Gamma) - \abs{z_k - y_k}\\
                                &\geq \varepsilon r_k - \tfrac{1}{2} \varepsilon r_k\\
                                &\geq \tfrac{1}{2}\varepsilon r_k
    \end{align}
    This says that the open ball $B(z_k,\tfrac{1}{2} \varepsilon r_k)$ is disjoint from $\Gamma$. In particular, $B(z_k, \tfrac{1}{2} \varepsilon r_k) \subset B(x,2r_k) \setminus \Gamma$ so, by (\ref{0mu_gamma}),
    \begin{equation}
        \limsup_k r_k^{-d} \lambda(B(z_k,\tfrac{1}{2}\varepsilon r_k)) \leq \limsup_k r_k^{-d} \mu(B(x,2r_k) \setminus \Gamma) = 0
    \end{equation}
    contradicting (\ref{csq_V})! From (\ref{0V_gamma}), we are going to show that there exists $R > 0$ such that $\overline{B}(x,R) \subset X$ and a $C^1$ map $p\colon V \cap B(x,R) \to \Gamma$ such that $p(x) = x$ and whose differential at $x$ is $\mathrm{id}$. Let us draw a few consequences. We fix $\varepsilon > 0$. For $r > 0$ is small enough, the application $p$ is $2$-Lipschitz in $V \cap B(x,r)$ and for all $y \in V \cap B(x,r)$, $\abs{p(y) - y} \leq \varepsilon \abs{y - x}$. According to the definition of admissible energies, we have
    \begin{equation}
        \lim_{r \to 0} \frac{\II(p(V \cap B(x,r))}{\II(V \cap B(x,r))} = 1.
    \end{equation}
    so for $r > 0$ small enough,
    \begin{equation}
        \II(p(V \cap B(x,r)) \leq (1+\varepsilon) \II(V \cap B(x,r)).
    \end{equation}
    These are the main properties of $p$ that we will use later. We proceed to the construction. As $\Gamma$ is locally diffeomorphic to a cone, there exists $R_0 > 0$ such that $\overline{B}(x,R_0) \subset X$, an open set $O \subset \R^n$ containing $x$, a $C^1$ diffeomorphism $\GT\colon B(x,R_0) \to O$ and a closed cone $S \subset \R^n$ centered at $x$ such that $\GT(x)=x$ and $\GT(\Gamma \cap B(x,R_0)) = S \cap O$. Let $A\colon \R^n \to \R^n$ be the affine differential of $\GT$ at $x$. For all $\varepsilon > 0$, there exists $0 < r \leq R_0$ such that for all $y \in B(x,r)$,
    \begin{equation}\label{T_gamma}
        \abs{\GT(y) - A(y)} \leq \varepsilon \abs{y - x}. 
    \end{equation}
    As $A$ is invertible, there exists $L \geq 1$ such that for all $y, z \in \R^n$,
    \begin{equation}\label{A_lip}
        L^{-1} \abs{y - z} \leq \abs{A(y) - A(z)} \leq L\abs{y - z}
    \end{equation}
    With these notations at hand, we are going to prove that the plane $W: = A(V)$ is included in $S$. First, we check that $W$ is tangent to $S$, that is, for all $0 < \varepsilon \leq 1$, there exists $r > 0$ such that
    \begin{equation}\label{0W_gamma}
        W \cap B(x,r) \subset \set{w \in \R^n | \mathrm{d}(w,S) \leq \varepsilon \abs{w-x}}.
    \end{equation}
    Let $0 < \varepsilon \leq 1$ and let $0 < r \leq R_0$ be such that (\ref{0V_gamma}) and (\ref{T_gamma}) hold true. For $w \in W$, we can apply the bilipschitz property (\ref{A_lip}) to see that $y = A^{-1}(w)$ satisfies
    \begin{equation}
        L^{-1} \abs{y - x} \leq \abs{w - x} \leq L \abs{y - x}
    \end{equation}
    Let $w \in W \cap B(0, \tfrac{1}{3} L^{-1} r)$. We have $y = A^{-1}(w) \in V \cap B(x, \tfrac{1}{3}r)$ and by (\ref{0V_gamma}), there exists $z \in \Gamma$ such that $\abs{y - z} \leq 2 \varepsilon \abs{y - x}$. In particular, $\abs{z - x} \leq 3 \abs{y - x} \leq r$ so $z \in B(x,R_0)$. Using (\ref{T_gamma}), we obtain that
    \begin{align}
        \mathrm{d}(w, S)    &\leq \mathrm{d}(w, \GT(\Gamma \cap B(x,R_0)))\\
                            &\leq \abs{A(y) - \GT(z)}\\
                            &\leq \abs{\GT(y) - A(y)} + \abs{\GT(z) - A(z)} + \abs{A(z) - A(y)}\\
                            &\leq \varepsilon \abs{y - x} + 3 \varepsilon \abs{y - x} + 2 L \varepsilon \abs{y - x}\\
                            &\leq 6L \varepsilon \abs{y - x}\\
                            &\leq 6L^2 \varepsilon \abs{w - x}.
    \end{align}
    This proves that $W$ is tangent to $S$. Observe that the right-hand side in (\ref{0W_gamma}) is invariant by homothety of center $x$. We deduce that for all $0 < \varepsilon \leq 1$, we have
    \begin{equation}
        W \cap B(x,1) \subset \set{w \in \R^n | \mathrm{d}(w,S) \leq \varepsilon \abs{w-x}}.
    \end{equation}
    Since $\varepsilon$ is arbitrary, we have in fact $W \subset S$. Now, let $p_W \colon \R^n \to W$ be the orthogonal projection onto $W$. Since $\GT(x) = x$, $p_W(x) = x$ and $O$ is a neighborhood of $x$, there exists a radius $0 < R \leq R_0$ such that $p_W \GT(B(x,R)) \subset O$. Then,
    \begin{equation}
        p_W \GT(B(x,R)) \subset W \cap O \subset S \cap O = \GT(\Gamma \cap B(x,R_0)).
    \end{equation}
    We define $p\colon V \cap B(x,R) \to \R^n$, $y \mapsto \GT^{-1} p_W \GT (y)$. It is clear that $p$ is $C^1$ on $V \cap B(x,R)$, that $p(V \cap B(x,R)) \subset \Gamma$, that $p(x) = x$ and that its differential at $x$ is $\mathrm{id}$. The preliminary part is finished.

    We are going apply the Vitali covering theorem \cite[Theorem 2.8]{Mattila} to the set
    \begin{equation}
        Y = (f(E) \cap \Gamma) \cup (f(E) \setminus \Gamma(2\delta)).
    \end{equation}
    Let $0 < \varepsilon \leq 1$. There exists a \emph{finite} sequence of open balls $(B_j)_j$ (of center $y_j \in Y$ and radius $0 < r_j \leq 1)$ such that the closures $(\overline{B_j})$ are disjoint and included in $X$,
    \begin{equation}\label{0YB}
        \HH^d(Y \setminus \bigcup_j \overline{B_j}) \leq t^{-1}\varepsilon_0,
    \end{equation}
    there exists an approximate tangent plane $V_j$ of $f(E)$ at $y_j$,
    \begin{subequations}\label{B_j}
        \begin{align}
            &   \HH^d(f(E) \cap B_j \setminus a B_j) \leq C (1 - a) r_j^d\\
            &   \HH^d(f(E) \cap B_j \setminus C_j) \leq \varepsilon r_j^d,\\
            &   \II(f(E) \cap B_j) \geq (1 + \varepsilon)^{-1} \II(V_j \cap B_j)
        \end{align}
    \end{subequations}
    where $C_j = C(y_j,V_j,\varepsilon)$. If $y_j \in f(E) \cap \Gamma$, we require that there exists a $C$-Lipschitz map $p_j\colon V_j \cap B_j \to \Gamma$ such that $\abs{p_j - \mathrm{id}} \leq \varepsilon r_j$ and
    \begin{equation}
        \II(p_j(V_j \cap B_j)) \leq (1 + \varepsilon) \II(V_j \cap B_j).
    \end{equation}
    If $y_j \in f(E) \setminus \Gamma(2\delta)$, we require that $B_j \subset X \setminus \Gamma(\delta)$ and we define $p_j = \mathrm{id}$ in $V_j \cap B_j$. Note that in both cases, $p_j$ is $C$-Lipschitz, $\abs{p_j - \mathrm{id}} \leq \varepsilon r_j$ and
    \begin{equation}\label{p_j}
        \II(p_j(V_j \cap B_j)) \leq (1 + \varepsilon) \II(V_j \cap B_j).
    \end{equation}
    Combining (\ref{0delta}) and (\ref{0YB}), we have
    \begin{equation}
        \HH^d(f(E) \setminus \bigcup_j \overline{B_j}) \leq 2t^{-1}\varepsilon_0
    \end{equation}
    We finally add the condition $\HH^d(f(E) \cap \partial B_j) = 0$ to get the "almost-covering" property with open balls,
    \begin{equation}\label{YB}
        \HH^d(f(E) \setminus \bigcup_j B_j) 2t^{-1}\leq \varepsilon_0.
    \end{equation}

    In each ball $B_j$, we make a projection onto $V_j \cap B_j$ followed by a retraction on the boundary if necessary. Let $\pi_j$ be partially defined by
    \begin{equation}
        \pi_j(y) =
        \begin{cases}
            p_j(y') &   \text{in} \ a B_j \cap C_j\\
            y       &   \text{in} \ \R^n \setminus B_j,
        \end{cases}
    \end{equation}
    where $y'$ is the orthogonal projection of $y$ onto $V_j$ and $C_j = C(y_j,V_j,\varepsilon)$. We estimate $\abs*{\pi_j - \mathrm{id}}$. By definition of $C_j$, we have $\abs{y' - y} \leq \varepsilon r_j$ in $aB_j \cap C_j$. By definition of $p_j$, we have $\abs{p_j - \mathrm{id}} \leq \varepsilon r_j$ in $V_j \cap aB_j$. We deduce that for $y \in aB_j \cap B_j$,
    \begin{equation}
        \abs*{p_j(y') - y} \leq \abs*{p_j(y') - y'} + \abs*{y' - y} \leq C\varepsilon r_j
    \end{equation}
    Next, we estimate $\norm{\pi_j - \mathrm{id}}_L$. It is clear that $y \mapsto p_j(y')$ is $C$-Lipschitz in $aB_j \cap C_j$ so $\pi_j - \mathrm{id}$ is $C$-Lipschitz in $aB_j \cap C_j$. The map $\pi_j - \mathrm{id}$ is also clearly $C$-Lipschitz in $\R^n \setminus B_j$. We have
    \begin{equation}
        \mathrm{d}(a B_j, \R^n \setminus B_j) \geq (1 - a) r_j
    \end{equation}
    so for $x \in a B_j \cap C_j$ and for $y \in \R^n \setminus B_j$,
    \begin{align}
        \abs{(\pi_j - \mathrm{id})(x) - (\pi_j - \mathrm{id})(y)}   &   \leq \abs{(\pi_j - \mathrm{id})(x)}\\
                                                                    &   \leq C\varepsilon r_j\\
                                                                    &   \leq \frac{C\varepsilon}{1 - a} \abs{x - y}.
    \end{align}
    We choose $a = 1 - \varepsilon$ so that
    \begin{equation}
        \abs{(\pi_j - \mathrm{id})(x) - (\pi_j - \mathrm{id})(y)} \leq C \abs{x - y}
    \end{equation}
    We conclude that $\pi_j - \mathrm{id}$ is $C$-Lipschitz on its domain. We apply the McShane's extension Lemma (Lemma \ref{lipschitz_extension}) to $\pi_j - \mathrm{id}$ and extend $\pi_j$ as a Lipschitz map $\pi_j\colon \R^n \to \R^n$ such that $\abs{\pi_j - \mathrm{id}} \leq C\varepsilon r_j$ and $\pi_j - \mathrm{id}$ is $C$-Lipschitz. It is left to paste together the functions $\pi_j$ into a function $\pi$:
    \begin{equation}
        \pi =
        \begin{cases}
            \pi_j   &   \text{in} \ B_j\\
            \mathrm{id} &   \text{in} \ \R^n \setminus \bigcup_j B_j.
        \end{cases}
    \end{equation}
    It is clear that $\abs{\pi - \mathrm{id}} \leq \varepsilon$. The map $\pi - \mathrm{id}$ is also locally $C$-Lipschitz. Indeed, $\pi = \mathrm{id}$ in the open set $\R^n \setminus \bigcup_j \overline{B_j}$ and for each index $j$, $\pi = \pi_j$ in the open set $\R^n \setminus \bigcup_{i \ne j} \overline{B_i}$. These these open sets cover $\R^n$ because the closed balls $(\overline{B_i})_i$ are disjoints. Then $\pi - \mathrm{id}$ is globally $C$-Lipschitz by convexity of $\R^n$. 

    To solve step 1, we would like to define $g_1 \approx p \circ \pi \circ f$ and $V_1 \approx f^{-1}(V'_1)$ where
    \begin{equation}
        V'_1 = \bigcup_j aB_j \cap C_j.
    \end{equation}
    Remember that $p$ is the retraction onto $\Gamma$ defined at the beginning of this step. The application $f$ sends the open set $V_1$ to $V_1' = \bigcup_j aB j \cap C_j$, $\pi$ smashes each $aB_j \cap C_j$ to a tangent disk $V_j \cap B_j$ (possibly composed with $p_j$ to be included in $\Gamma$). The final composition with $p$ makes sure that $\Gamma$ is preserved.

    In fact, the definition of $g_1$ will be a bit more involved because we also need that need $g_1 = f$ in $\R^n \setminus W_1$. We are going to estimate $\HH^d(f(E) \setminus V'_1)$ and $\II(p \circ \pi(V'_1))$ and then make the necessary adjustments to $g_1$. According to (\ref{B_j}), the fact that $C^{-1} \HH^d \leq \II \leq C \HH^d$ and the choice $a = 1 - \varepsilon$, we have
    \begin{align}
        \begin{split}
            \HH^d(f(E) \setminus V'_1)  &   \leq \HH^d(f(E) \setminus \bigcup_j B_j) + \sum_j \HH^d(f(E) \cap (B_j \setminus a B_j))\\
                                        &   \qquad + \sum_j \HH^d(f(E) \cap (B_j \setminus C_j))
        \end{split}\\
                                        &   \leq 2t^{-1}\varepsilon_0 + \sum_j C (1 - a) r_j^d + \varepsilon r_j^d\\
                                        &   \leq 2t^{-1}\varepsilon_0 + C \varepsilon \sum_j r_j^d\\
                                        &   \leq 2t^{-1}\varepsilon_0 + C \varepsilon \sum_j \HH^d(f(E) \cap B_j)\\
                                        &   \leq 2t^{-1}\varepsilon_0 + C \varepsilon \HH^d(f(E))\label{technical1}.
    \end{align}
    To estimate $\II(p \circ \pi(V'_1))$, we observe that
    \begin{equation}
        p \circ \pi (V'_1) \subset \bigcup_j p \circ \pi_j(a B_j \cap C_j).
    \end{equation}
    For each $j$, we distinguish two cases. If $y_j \in f(E) \cap \Gamma$, then $\pi_j(a B_j \cap C_j) \subset \Gamma$ so $p = \mathrm{id}$ on $\pi_j(a B_j \cap C_j)$. If $y_j \in f(E) \setminus \Gamma(2\delta)$ then $\pi_j(a B_j \cap C_j) \subset B_j \subset X \setminus \Gamma(\delta)$ so $p = \mathrm{id}$ on $\pi_j(a B_j \cap C_j)$. It follows that
    \begin{equation}
        p \circ \pi (V'_1) \subset \bigcup_j \pi_j(a B_j \cap C_j).
    \end{equation}
    Using (\ref{B_j}) and (\ref{p_j}), we deduce
    \begin{align}
        \II(p \circ \pi(V'_1))  &\leq \sum_j \II(\pi_j(aB_j \cap C_j))\\
                                &\leq (1+\varepsilon) \sum_j \II(V_j \cap B_j))\\
                                &\leq (1 + \varepsilon)^2 \sum_j \II(f(E) \cap B_j)\\
                                &\leq (1 + \varepsilon)^2 \HH^d(f(E))\label{technical2}
    \end{align}

    Finally, we build $g_1$. By definition of $W_1$, $\HH^d(E_1 \setminus W_1) \leq \varepsilon_0$ so there exists an open set $W''_1 \subset \subset W_1$ such that
    \begin{equation}\label{W_0}
        \HH^d(E_1 \setminus W''_1) \leq 2\varepsilon_0.
    \end{equation}
    We also consider an intermediate open set $W'_1$ such that $W''_1 \subset \subset W'_1 \subset \subset W_1$. Note that $W''_1$ and $W'_1$ do not depend on the intermediate variables $\delta$, $\varepsilon$. We define
    \begin{equation}
        V_1 = W''_1 \cap f^{-1}(V'_1)
    \end{equation}
    and
    \begin{equation}
        f_1 =
        \begin{cases}
            \pi \circ f &   \text{in} \ V_1\\
            f           &   \text{in} \ \R^n \setminus W'_1.
        \end{cases}
    \end{equation}
    We have $\abs{f_1 - f} \leq \varepsilon$ because $\abs{\pi - \mathrm{id}} \leq \varepsilon$. We are going to estimate $\norm{f_1 - f}_L \leq C \norm{f}_L$. In $V_1$, $(f_1 - f) = (\pi - \mathrm{id}) \circ f$ is $C \norm{f}_L$ Lipschitz because $\pi - \mathrm{id}$ is $C$-Lipschitz.vIn addition, we assume $\varepsilon$ small enough so that $\mathrm{d}(W''_1, \R^n \setminus W'_1) \geq \varepsilon \norm{f}_L^{-1}$. Then, for $x \in V_1$ and $y \in \R^n \setminus W'_1$,
    \begin{align}
        \abs{(f_1 - f)(x) - (f_1 - f)(y)}   &\leq \abs{(f_1 - f)(x)}\\
                                            &\leq \varepsilon\\
                                            &\leq \norm{f}_L \abs{x - y}.
    \end{align}
    We apply McShane's extension Lemma (Lemma \ref{lipschitz_extension}) to $f_1 - f$ and obtain a Lipschitz map $f_1\colon \R^n \to \R^n$ such that $\abs{f_1 - f} \leq \varepsilon$ and $\norm{g_1 - f}_L \leq C\norm{f}_L$. As we have to preserve $\Gamma$, we finally define
    \begin{equation}
        g_1 =
        \begin{cases}
            p \circ f_1 &   \text{in} \ V_1 \cup \Gamma\\
            f_1 (= f)   &   \text{in} \ \R^n \setminus W'_1.
        \end{cases}
    \end{equation}
    As $\abs{p - \mathrm{id}} \leq \delta$ and $\abs{f_1 - f} \leq \varepsilon$, we have $\abs{g_1 - f} \leq \varepsilon + \delta$. We can assume $\varepsilon$ and $\delta$ small enough so that $\abs{g_1 - f} \leq \varepsilon_0$. Next, we estimate $\norm{g_1 - f}_L \leq C \norm{f}_L$. It suffices to have $\norm{g_1 - f_1} \leq C \norm{f}$. The first ingredients are the fact that $\abs{p - \mathrm{id}} \leq \delta$, that $V_1 \subset W''_1$ and that for $\delta$ small enough, $\mathrm{d}(W''_1, \R^n \setminus W'_1) \geq \delta \norm{f}_L^{-1}$. As a consequence, for $x \in V_1$ and $y \in \R^n \setminus W'_1$,
    \begin{align}
        \abs{(g_1 - f_1)(x) - (g_1 - f_1)(y)}   &   \leq \abs{(g_1 - f_1)(x)}\\
                                                &   \leq \delta\\
                                                &   \leq \norm{f}_L \abs{x - y}.
    \end{align}
    The second ingredients are the facts that $p$ is $C$-Lipschitz and $p = \mathrm{id}$ on $\Gamma$. It implies that for all $x \in \R^n$,
    \begin{align}
        \abs{p(x) - x}  &\leq \norm{p - \mathrm{id}}_L \mathrm{d}(x,\Gamma)\\
                        &\leq C \mathrm{d}(x,\Gamma).
    \end{align}
    We also recall that $f(\Gamma) \subset \Gamma$. As a consequence, for $x \in \Gamma$ and $y \in \R^n \setminus W'_1$,
    \begin{align}
        \abs{(g_1 - f_1)(x) - (g_1 - f_1)(y)}   &= \abs{p(f_1(x)) - f_1(x)}\\
                                                &\leq C \mathrm{d}(f_1(x), \Gamma)\\
                                                &\leq C \abs{f_1(x) - f(x)}\\
                                                &\leq C \abs{(f_1 - f)(x) - (f_1 - f)(y)}\\
                                                &\leq C \norm{f}_L \abs{x - y}.
    \end{align}
    We apply McShane's extension Lemma (Lemma \ref{lipschitz_extension}) to $g_1 - f$ and obtain a Lipschitz map $g_1\colon \R^n \to \R^n$ such that $\abs{g_1 - f} \leq \varepsilon_0$ and $\norm{g_1 - f}_L \leq C\norm{f}_L$. Next, we check that $g_1(\Gamma) \subset \Gamma$. In view of the definition of $p$ and $O_\Gamma$, it suffices that $f_1(\Gamma) \subset p^{-1}(O_\Gamma)$. As $\Gamma$ is relatively closed in $X$ and $\overline{W'_1}$ is a compact subset of $X$, the intersection $\Gamma \cap \overline{W'_1}$ is compact. The image $f(\Gamma \cap \overline{W'_1})$ is a compact subset of $\Gamma \subset p^{-1}(O_\Gamma)$. We take $\varepsilon$ small enough so that for all $x \in \Gamma \cap \overline{W'_1}$,
    \begin{equation}
        \overline{B}(f(x),\varepsilon) \subset p^{-1}(O_\Gamma).
    \end{equation}
    We deduce that $f_1(\Gamma) \subset p^{-1}(O_\Gamma)$ as $\abs{f_1 - f} \leq \varepsilon$ and $f_1 = f$ in $\R^n \setminus W'_1$. Next, we estimate $\HH^d(E_1 \setminus V_1)$ and $\II(g_1(V_1))$. It was almost done in (\ref{technical1}) and (\ref{technical2}). We recall that by definition of $E_1$, for all $A \subset \R^n$,
    \begin{equation}
        \HH^d(E_1 \cap f^{-1}(A)) \leq t\HH^d(f(E) \cap A).
    \end{equation}
    By (\ref{technical1}), we have
    \begin{align}
        \HH^d(E_1 \setminus V_1)    &\leq \HH^d(E_1 \setminus W''_1) + \HH^d(E_1 \setminus f^{-1}(V'_1))\\
                                    &\leq 2\varepsilon_0 + t \HH^d(f(E) \setminus V'_1)\\
                                    &\leq 4\varepsilon_0 + C t \varepsilon \HH^d(f(E)).
    \end{align}
    By the definition of $g_1$ and (\ref{technical2}), we have
    \begin{equation}
        \II(g_1(V_1)) \leq \II(p \circ \pi(V'_1)) \leq (1+\varepsilon)^2 \II(f(E)).
    \end{equation}
    We take one last time $\varepsilon$ small enough so that
    \begin{align}
        &   \HH^d(E_1 \setminus V_1)    \leq 5\varepsilon_0,\\
        &   \HH^d(g_1(V_1))             \leq \HH^d(f(E)) + \varepsilon_0.
    \end{align}

    \emph{Step 2. For a suitable choice of $t$, we build a Lipschitz map $g_2\colon \R^n \to \R^n$ such that $g_2(\Gamma) \subset \Gamma$, $g_2 - f$ has a compact support included in $W_2$, $\abs{g_2 - f} \leq \varepsilon$, $\norm{g_2}_L \leq \norm{f}_L$ and such that there is an open set $V_2 \subset W_2$ satisfying}
    \begin{subequations}
        \begin{align}
            &   \HH^d(E_2 \setminus V_2) \leq \varepsilon_0,\\
            &   \HH^d(g_2(V_2)) \leq \varepsilon_0.
        \end{align}
    \end{subequations}
    We recall that for $\HH^d$-ae. $x \in E_2$, there exists a (unique) approximate tangent plane $V_x$ of $E_2$ at $x$. For such point $x$, for all $0 < \varepsilon \leq 1$ for all $0 < a < 1$, 
    \begin{align}
        &   \lim_{r \to 0} (2r)^{-d} \HH^d(E \cap B(x,r)) = 1\\
        &   \limsup_{r \to 0} r^{-d} \HH^d(E \cap a\overline{B}(x,r) \setminus B(x,r)) \leq C(1-a)r^d\\
        &   \lim_{r \to 0} r^{-d} \HH^d(E \cap \overline{B}(x,r) \setminus C(x,V_x,\varepsilon) = 0.
    \end{align}
    Moreover, for $\HH^d$-ae. $x \in E_2$, there exists a (unique) affine map $A_x\colon V_x \to \R^n$ such that for all $\varepsilon > 0$, there exists $r > 0$ such that for all $y \in V_x \cap B(x,r)$,
    \begin{equation}
        \abs{f(y) - A_x(y)} \leq \varepsilon \abs{y - x}.
    \end{equation}

    Let $0 < \varepsilon \leq 1$, let $0 < a < 1$ (to be chosen later). There exists an $\HH^d$ measurable subset $K \subset E_2$ such that $\HH^d(E_2 \setminus K) \leq \varepsilon_0$ and for some radius $0 < r_0 \leq 1$, for all $x \in K$,
    \begin{subequations}\label{B_k}
        \begin{align}
            &   \overline{B}(x,r_0) \subset W_2\label{B_k1}\\
            &   \forall\ 0 < r \leq r_0,\ (2r)^{-d} \HH^d(E \cap B(x,r)) \geq \tfrac{1}{2}\label{A2}\\
            &   \forall\ 0 < r \leq r_0,\ \HH^d(E \cap \overline{B}(x,r) \setminus aB(x,r)) \leq C (1 - a) r^d\\
            &   \forall\ 0 < r \leq r_0,\ \HH^d(E \cap \overline{B}(x,r) \setminus C(x,V_x,\varepsilon) \leq \varepsilon r^d\\
            &   \forall y \in V_x \cap B(x,r_0),\ \abs{f(y) - A_x(y)} \leq \varepsilon \abs{y - x}.\label{B_k5}
        \end{align}
    \end{subequations}
    The Approximation Theorem \cite[Theorem 1.10]{Mattila} also allows to assume that $K$ is compact. As $\HH^d(E_2 \setminus W_2) \leq \varepsilon_0$, we finally assume that $K \subset E_2 \cap W_2$ and
    \begin{equation}\label{K_cover}
        \HH^d(E_2 \setminus K) \leq 2 \varepsilon_0.
    \end{equation}
    By the definition of Hausdorff measures, there exists a countable set $\mathcal{S}$ of closed balls $D = B(y,\rho)$ of center $y \in f(K)$ and radius $0 < \rho \leq \varepsilon r_0$ such that
    \begin{equation}
        f(K) \subset \bigcup_{D \in \mathcal{S}} D
    \end{equation}
    and
    \begin{equation}\label{hausdorff_cover}
        \sum_{D \in \mathcal{S}} \mathrm{diam}(D)^d \leq C \HH^d(f(K)) + t^{-1}.
    \end{equation}
    We cover $K$ by a \emph{finite} sequence of open balls $(B_k)=(B(x_k,r_k))$ (of center $x_k \in K$ and radius $r_k > 0$) such that for each $k$, there exists a ball $D_k \in \mathcal{S}$ (of center $y_k$ and radius $\rho_k$) such that $x_k \in f^{-1}(D_k)$ and $r_k = \varepsilon^{-1} \rho_k$. In particular,
    \begin{equation}\label{x_k_def}
        \abs{f(x_k) - y_k} \leq \varepsilon r_k.
    \end{equation}
    In any metric space, a \emph{finite} family of balls $(B_k)_k$ admits a subfamily of disjoint balls $(B_l)_l$ such that $\bigcup_k B_k \subset \bigcup_l 3 B_l$. Thus, we can require that the smaller balls $(\frac{1}{3} B_k)_k$ are disjoint while
    \begin{equation}\label{B_k_cover}
        K \subset \bigcup_k B_k.
    \end{equation}
    Since $\rho_k \leq \varepsilon r_0$, we have $r_k \leq r_0$. In particular, $\overline{B_k} \subset W_2$ by (\ref{B_k1}). Moreover, (\ref{A2}) and the fact that the balls $(\frac{1}{3} B_k)$ are disjoint yields
    \begin{equation}\label{disjoint_balls}
        \sum_k r_k^d \leq C \HH^d(E).
    \end{equation}
    In each ball $B_k$, we are going to replace $f$ by an orthogonal projection onto a plane. For each $k$, we define $p_k$ as the orthogonal projection onto $\mathrm{Im} \, A_{x_k}$. We can write $p_k$ as
    \begin{equation}\label{p_k_def}
        p_k = f(x_k) + \vec{p}_k(\cdot - f(x_k)),
    \end{equation}
    where $\vec{p}_k$ is the (linear) orthogonal projection onto the direction of $\mathrm{Im} \, A_{x_k}$. We consider a maximal set $\mathcal{T}$ of linear orthogonal projection of rank $\leq d$ and of mutual distances $> \varepsilon$. It contains at most $N(\varepsilon)$ elements, where $N(\varepsilon)$ depends on $n$ and $\varepsilon$. For each $k$, let $\vec{\pi_k} \in \mathcal{T}$ be such that
    \begin{equation}\label{vecpi_k_def}
        \norm*{\vec{p}_k - \vec{\pi}_k} \leq \varepsilon.
    \end{equation}
    Finally, we define $\pi_k$ as the orthogonal projection onto the affine plane $y_k + \mathrm{Im} \, \vec{\pi}_k$, that is
    \begin{equation}\label{pi_k_def}
        \pi_k = y_k + \vec{\pi}_k(\cdot - y_k).
    \end{equation}
    Without loss of generality, we assume that the domain of definition of $(B_k)$ is totally ordered and that the radius $(r_k)$ are non-increasing. We define
    \begin{equation}
        f_2 =
        \begin{cases}
            \pi_k \circ f   &   \text{in} \ a B_k \cap C_k \setminus \bigcup_{i < k} B_i\\
            f               &   \text{in} \ \R^n \setminus \bigcup_k B_k,
        \end{cases}
    \end{equation}
    where $C_k = C(x_k, V_{x_k},\varepsilon)$. We estimate $\abs{f_2 - f}$. For all $x \in a B_k \cap C_k \setminus \bigcup_{i < k} B_i$, we have
    \begin{align}
        \abs{f_2(x) - f(x)} &   = \abs{\pi_k f(x) - f(x)}\\
                            &   \leq \abs{\pi_k f(x) - p_k f(x)} + \abs{p_k f(x) - f(x)}.
    \end{align}
    By (\ref{x_k_def}), (\ref{p_k_def}), (\ref{vecpi_k_def}) and (\ref{pi_k_def}),
    \begin{align}
        \begin{split}
            \abs{\pi_k f(x) - p_k f(x)} &   \leq \abs{f(x_k) - y_k} + \abs{\vec{\pi}_k(f(x_k) - y_k)}\\
                                        &   \qquad + \abs{(\vec{\pi}_k - \vec{p}_k)(f(x) - f(x_k))}
        \end{split}\\
        &   \leq 2\abs{f(x_k) - y_k} + \norm{\vec{\pi}_k - \vec{p}_{k}}\abs{f(x) - f(x_k)}\\
        &   \leq 2\varepsilon r_k + \norm{f}_L\varepsilon r_k.
    \end{align}
    By the properties of orthogonal projections, the definition of $C_k$ and (\ref{B_k5}),
    \begin{align}
        \abs{p_k f(x) - f(x)}   &   =\mathrm{d}(f(x), \mathrm{Im} \, A_{x_k})\\
                                &   \leq \abs{f(x) - A_{x_k}(x' - x_k)}\\
                                &   \leq \abs{f(x) - f(x')} + \abs{f(x') - A_{x_k}(x' - x_k)}\\
                                &   \leq \norm{f}_L \varepsilon r_k + \varepsilon r_k,
    \end{align}
    where $x'$ is the orthogonal projection of $x$ onto $V_{x_k}$. In conclusion, $\abs{f_2 - f} \leq C(\norm{f}_L+1)\varepsilon r_k$ in $a B_k \cap C_k \setminus \bigcup_{i < k} B_i$. We replace $\varepsilon$ by a smaller value in the previous construction so as to simply assume $\abs{f_2 - f} \leq \varepsilon r_k$. Next, we estimate $\norm{f_2 - f}_L$. In each set $a B_k \cap C_k \setminus \bigcup_{i < k} B_i$, the map $f_2 - f = (\pi_k - \mathrm{id}) \circ f$ is $C\norm{f}_L$-Lipschitz because $(\pi_k - \mathrm{id})$ is $C$-Lipschitz. As
    \begin{equation}
        \mathrm{d}(a B_k, \R^n \setminus B_k) \geq (1 - a) r_k,
    \end{equation}
    we have for $x \in a B_k \cap C_k \setminus \bigcup_{i < k} B_i$ and $y \in \R^n \setminus \bigcup_k B_k$,
    \begin{align}
        \abs{(f_2-f)(x) - (f_2-f)(y)}   &   \leq \abs{f_2(x) - f(x)}\\
                                        &   \leq \varepsilon r_k\\
                                        &   \leq \frac{\varepsilon}{1 - a} \abs{x - y}.
    \end{align}
    For $x \in a B_k \cap C_k \setminus \bigcup_{i < k} B_i$ and $y \in a B_l \cap C_l \setminus \bigcup_{i < l} B_i$ where $k < l$, we have similarly
    \begin{align}
        \abs{(f_2-f)(x) - (f_2-f)(y)}   &   \leq \abs{f_2(x) - f(x)} + \abs{f_2(y) - f(y)}\\
                                        &   \leq \varepsilon r_k +  \varepsilon r_l\\
                                        &   \leq 2\varepsilon r_k\\
                                        &   \leq \frac{2\varepsilon}{1 - a} \abs{x - y}.
    \end{align}
    Taking $a = 1 - \varepsilon$, we conclude that $f_2 - f$ is $C$-Lipschitz. We apply McShane's extension Lemma (Lemma \ref{lipschitz_extension}) to $f_2 - f$ and extend $f_2$ as a Lipchitz map $\R^n \to \R^n$ such that $\abs{f_2 - f} \leq \varepsilon$ and $\norm{f_2 - f}_L \leq C \norm{f}_L$.

    Now, we define
    \begin{equation}
        V_2 = \bigcup_k B_k \setminus A,
    \end{equation}
    where $A$ is the compact set $A = \bigcup_k \overline{B_k} \setminus (a B_k \cap C_k)$ and we estimate $\HH^d(f_2(V_2))$ and $\HH^d(E_2 \setminus V_2)$. By (\ref{B_k}), (\ref{K_cover}), (\ref{B_k_cover}), (\ref{disjoint_balls}) and the choice $a = 1 - \varepsilon$,
    \begin{align}
        \HH^d(E_2 \setminus V_2)    &   \leq \HH^d(E_2 \setminus K) + \HH^d(K \setminus V_2)\\
                                    &   \leq 2 \varepsilon_0 + \HH^d(K \cap A)\\
                                    \begin{split}
                                    &   \leq 2\varepsilon_0 + \sum_k \HH^d(E \cap \overline{B_k} \setminus a B_k)\\
                                    &   \qquad + \sum_k \HH^d(E \cap \overline{B_k} \setminus C_k)
                                    \end{split}\\
                                    &   \leq 2\varepsilon_0 + C (1 - a) \sum_k r_k^d + \varepsilon \sum_k r_k^d\\
                                    &   \leq 2\varepsilon_0 + C\varepsilon \sum_k r_k^d\\
                                    &   \leq 2\varepsilon_0 + C\varepsilon \HH^d(E)\label{V_2_cover}
    \end{align}
    By definition of $A$ and $V_2$,
    \begin{align}
        f_2(V_2)    &   \subset f_2(\bigcup_k B_k \setminus A)\\
                    &   \subset \bigcup_k f_2(a B_k \cap C_k \setminus \bigcup_{i < k} B_i).
    \end{align}
    By definition of $f_2$,
    \begin{equation}
        f_2(a B_k \cap C_k \setminus \bigcup_{i < k} B_i) \subset y_k + \mathrm{Im} \, \vec{\pi}_k.
    \end{equation}
    We have in addition for $x \in a B_k \cap C_k \setminus \bigcup_{i < k} B_i$,
    \begin{align}
        \abs{f_2(x) - y_k}   &= \abs{\pi_k f(x) - \pi_k(y_k)}\\
                             &\leq \abs{f(x) - y_k}\\
                             &\leq \abs{f(x) - f(x_k)} + \abs{f(x_k) - y_k}\\
                             &\leq \norm{f}_L r_k + \varepsilon r_k
    \end{align}
    We assume $\varepsilon \leq \norm{f}_L$ so that this simplifies to $\abs{f_2(x) - y_k} \leq 2 \norm{f}_L r_k$. By definition, $r_k = \varepsilon^{-1} \rho_k$ so
    \begin{equation}
        f_2(a B_k \cap C_k \setminus \bigcup_{i < k} B_i) \subset \left(y_k + \mathrm{Im} \, \vec{\pi}_k\right) \cap B(y_k, \lambda \rho_k)
    \end{equation}
    where $\lambda = 2 \norm{f}_L \varepsilon^{-1}$. We deduce that
    \begin{equation}
        f_2(V_2) \subset \bigcup_{D \in \mathcal{S},\\ \vec{\pi} \in \mathcal{T}} (y + \mathrm{Im} \, \vec{\pi}) \cap \lambda D.
    \end{equation}
    Using the fact that $\mathcal{T}$ has at most $N(\varepsilon)$ elements and using (\ref{hausdorff_cover}), we get
    \begin{align}
        \HH^d(f_2(V_2)) &   \leq C N(\varepsilon) \lambda^d \sum_{D \in \mathcal{S}} \mathrm{diam}(D)^d\\
                        &   \leq C N(\varepsilon) \lambda^d \left(\HH^d(f(K)) + t^{-1}\right)
    \end{align}
    We recall that by definition of $E_2$, $\HH^d(f(E_2)) \leq t^{-1} \HH^d(E)$ so
    \begin{equation}
        \HH^d(f_2(V_2)) \leq C N(\varepsilon) \lambda^d t^{-1} (\HH^d(E) + 1).
    \end{equation}
    We will choose $t$ so that $\HH^d(f_2(V_2)) \leq \varepsilon_0$. However, an extra step is needed to make a retraction onto the boundary. For $\delta > 0$, there exists a $C$-Lipschitz map $p\colon \R^n \to \R^n$ and an open set $O \subset X$ containing $\Gamma$ such that $\abs{p - \mathrm{id}} \leq \delta$, $p(O) \subset \Gamma$ and $p = \mathrm{id}$ on $\Gamma$. We consider
    \begin{equation}
        g_2 =
        \begin{cases}
            p \circ f_2 &   \text{in} \ V_2 \cup \Gamma\\
            f_2 (=f)    &   \text{in} \ \R^n \setminus \bigcup_k B_k,
        \end{cases}
    \end{equation}
    As the family $(B_k)$ is finite and for each $k$, $\overline{B_k} \subset W_2$, we have $\bigcup_k B_k \subset \subset W_2$. As in step 1, we can assume $\varepsilon$ and $\delta$ small enough so that $\abs{g_2 - f} \leq \varepsilon_0$, $\norm{g_2 - f}_L \leq C\norm{f}_L$ and $g_2(\Gamma) \subset \Gamma$. We extend $g_2$ as a Lipschitz map $g_2\colon \R^n \to \R^n$ such that $\abs{g_2 - f} \leq \varepsilon_0$, $\norm{g_2 - f}_L \leq C \norm{f}_L$. Since $p$ is $C$-Lipschitz, we have
    \begin{align}
        \HH^d(g_2(V_2)) &\leq C \HH^d(f_2(V_2))\\
                        &\leq C N(\varepsilon) \lambda^d t^{-1} (\HH^d(E) + 1).
    \end{align}
    We take $\varepsilon$ small enough so that (\ref{V_2_cover}) yields $\HH^d(E_2 \setminus V) \leq 3\varepsilon_0$. Then we choose $t$ big enough such that $\HH^d(f_2(V_2)) \leq \varepsilon_0$.

    \emph{Step 3.} We define $V = V_1 \cup V_2$ and $g$ the piecewise function
    \begin{equation}
        g(z) =
        \begin{cases}
            g_1(z)  &   \text{in} \ W_1\\
            g_2(z)  &   \text{in} \ W_2\\
            f(z)    &   \text{in} \ \R^n \setminus (W_1 \cup W_2).
        \end{cases}
    \end{equation}
    We have $\abs{g - f} \leq \varepsilon_0$ and $\norm{g - f}_L \leq C\norm{f}_L$ (the function $g-f$ is locally $C\norm{f}_L$-Lipschitz because each $g_i - f$ has a compact support included in $W_i$ and then $g-f$ is globally $\norm{f}_L$-Lipschitz by convexity of $\R^n$). As $W_1, W_2 \subset W \subset \subset U$, $g(\Gamma) \subset \Gamma$, $g = f$ in $X \setminus W$ and $\abs{g - f} \leq \varepsilon_0$, Lemma \ref{sliding_perturbation} says that $g$ is a global sliding deformation in $U$ provided that $\varepsilon_0$ is small enough. Finally, we have
    \begin{align}
        \HH^d(E \setminus V)    &   \leq \HH^d(E_1 \setminus V_1) + \HH^d(E_2 \setminus V_2)\\
                                &   \leq 2\varepsilon_0.
    \end{align}
    and
    \begin{align}
        \II^d(g(V)) &   \leq \II(g_1(V_1)) + \II(g_2(V_2))\\
                    &   \leq \II(g_1(V_1)) + C\HH^d(g_2(V_2))\\
                    &   \leq \II(f(E))) + C\varepsilon_0.
    \end{align}
\end{proof}

\subsubsection{Proof of the Limiting Theorem}
\begin{proof}
    The letter $C$ is an unspecified constant $\geq 1$ that depends on $n$, $\kappa$, $\Gamma$, $\II$. Its value can increase from one line to another (but a finite number of times). The letter $u > 0$ is a constant that depends on $n$, $s$, $\Gamma$. Let $E$ be the support of $\mu$ in $X$. Let $t$, $\GL$ and $\IL$ be the constants given by Definitions \ref{defi_whitney} and \ref{defi_energy}.

    \emph{Step 1. We show that for all $x \in E$, for all $0 < r \leq r_u(x)$ and for $i$ big enough,}
    \begin{equation}\label{step1_claim}
        C^{-1} r^d \leq \HH^d(E_i \cap B(x,r)) \leq C r^d.
    \end{equation}
    It is almost the same proof as in Proposition \ref{prop_density}. We fix $x \in E$ and $0 < r \leq \min \set{r_s(x),r_t(x)}$. The proof is based on two ingredients. The first one is that for all $0 < \rho \leq r$,
    \begin{equation}
        \liminf_i \HH^d(E_i \cap B(x, \rho)) > 0.
    \end{equation}
    The second one is that there there exists a sequence $(\varepsilon_i) \to 0$ such that for all global sliding deformations $f$ in $B(x,r)$,
            \begin{equation}
                \II(E_i \cap W_f) \leq \kappa \II(f(E_i \cap W_f)) + h \II(E_i \cap B(x,hr)) + \varepsilon_i.
            \end{equation}
            We can simplify the situation as we did in Proposition \ref{prop_density}. It suffices to solve the following problem. Let the ball $B(0,\sqrt{n})$ be denoted by $B$. We assume that $(E_i)$ is a sequence of $\HH^d$ finite closed subsets of $B$ and
    \begin{enumerate}
        \item $\liminf_i \HH^d(E_i \cap B(0,\tfrac{1}{2})) > 0$;
        \item there exists $S \subset \mathcal{E}_n(1)$ such that $\Gamma \cap B(0,\sqrt{n}) = \abs{S} \cap B(0,\sqrt{n})$;
        \item there exists a sequence $(\varepsilon_i) \to 0$ such that for all global sliding deformations $f$ along $\Gamma$ in $B$,
            \begin{equation}
                \HH^d(E_i \cap W_f) \leq \kappa \HH^d(f(E_i \cap W_f)) + h \HH^d(E_i \cap hB) + \varepsilon_i.
            \end{equation}
    \end{enumerate}
    Then we show that for $i$ big enough, 
    \begin{align}
        \HH^d(E_i \cap \mathopen{]}-\tfrac{1}{2},\tfrac{1}{2}\mathclose{[}^n)    &\leq C\\
        \HH^d(E_i \cap \mathopen{]}-1,1\mathclose{[}^n)                          &\geq C.
    \end{align}
    We can proceed as in the reduced problem of Proposition \ref{prop_density}. Note that for $i$ big enough, $\varepsilon_i \leq h \HH^d(E_i \cap B(0,\tfrac{1}{2}))$ so the term $\varepsilon_i$ is be absorbed in $h \HH^d(E_i \cap B(0,\tfrac{1}{2}))$.

    Before passing to step 2, let us draw the main consequences of (\ref{step1_claim}). We recall that $\II^d \mres E_i \rightharpoonup \mu$ so for $x \in E$ and for $0 < r \leq r_u(x)$,
    \begin{equation}
        C^{-1} r^d \leq \mu(B(x,r)) \leq C r^d.
    \end{equation}
    According to the density theorems on Radon measures \cite[Theorem 6.9]{Mattila},
    \begin{equation}
        C^{-1} \HH^d \mres E \leq \mu \leq C \HH^d \mres E
    \end{equation}
    whence for all $x \in E$ and for all $0 < r \leq r_u(x)$,
    \begin{equation}\label{E_af}
        C^{-1} r^d \leq \HH^d(E \cap (B(x,r)) \leq C r^d.
    \end{equation}
    We are going to deduce the regularity property mentioned in Remark \ref{rmk_af_semi}. Let $v > 0$ be a scale such that $r_v \leq \tfrac{1}{3}r_u$. We fix $x \in E$ and $B = B(x,r_v(x))$. For $y \in B$, we have $r_u(x) \leq r_u(y) + r_v(x) \leq r_u(y) + \tfrac{1}{3}r_u(x)$ so $\tfrac{2}{3} r_u(x) \leq r_u(y)$ and in turn $2 r_v(x) \leq r_u(y)$. We conclude that for all $y \in B$, for all $0 \leq r \leq 2r_v(x)$,
    \begin{equation}
        C^{-1} r^d \leq \HH^d(E \cap B(y,r)) \leq C r^d.
    \end{equation}
    Remark \ref{rmk_af_semi} says that for all subset $S \subset E \cap B(x,r_v(x))$, for all radius $\rho > 0$, the set $S$ can be covered by at most $C \rho^{-d} \mathrm{diam}(S)^d$ balls of radius $\rho$. We reduce $u$ once again so that this property holds for $r_u$ in place of $r_v$.

    \emph{Step 2. We show that $E$ is $\HH^d$ rectifiable.}  More precisely, we show that there exists $\lambda \geq 1$ such that for all $x \in E$, for small all $r > 0$,
    \begin{equation}\label{step2_claim}
        \HH^d(E \cap B(x,r)) \leq C \HH^d(E_r \cap B(x, \lambda r)).
    \end{equation}
    We fix $x \in E$ and $0 < r \leq \min \set{r_s(x),r_t(x),r_u(x)}$. The proof relies on five arguments. For all $0 < \rho \leq r$,
    \begin{equation}
        \liminf_i \HH^d(E_i \cap B(x,\rho)) > 0.
    \end{equation}
    For all compact set $K \subset B(x,r)$,
    \begin{equation}
        \limsup_i \HH^d(E_i \cap K) \leq C\HH^d(E \cap K).
    \end{equation}
    For all $0 < \rho \leq r$,
    \begin{equation}
        \rho^{-d} \HH^d(E \cap B(x,\rho)) \leq C r^{-d} \HH^d(E \cap B(x,r)).
    \end{equation}
    For all subset $S \subset E \cap B(x,r)$, for all radius $\rho > 0$, the set $S$ can be covered by at most $C \rho^{-d} \mathrm{diam}(S)^d$ balls of radius $r$. And finally, there exists a sequence $(\varepsilon_i) \to 0$ such that for all global sliding deformations $f$ in $B(x,r)$,
            \begin{equation}
                \II(E_i \cap W_f) \leq \kappa \II(f(E_i \cap W_f)) + h \II(E_i \cap B(x,hr)) + \varepsilon_i.
            \end{equation}
            We can simplify the situation as usual. It suffices to solve the following problem. Let the ball $B(0,\sqrt{n})$ be denoted by $B$. We assume that $E$ and $(E_i)$ are $\HH^d$ finite closed subsets of $B$ and
    \begin{enumerate}
        \item $\liminf_i \HH^d(E_i \cap B(0,\tfrac{1}{2})) > 0$;
        \item for all compact set $K \subset B$, $\limsup_i \HH^d(E_i \cap K) \leq C \HH^d(E \cap K)$;
        \item $\HH^d(E \cap B) \leq C \HH^d(E \cap B(0,\tfrac{1}{2}))$;
        \item for all subset $S \subset E \cap B$, for all radius $r > 0$, the set $S$ can be covered by at most $C r^{-d} \mathrm{diam}(S)^d$ balls of radius $r$;
        \item there exists a set $S \subset \mathcal{E}_n(1)$ such that $\Gamma \cap B = \abs{S} \cap B$;
        \item there exists a sequence $(\varepsilon_i) \to 0$ such that for all global sliding deformations $f$ in $B$,
            \begin{equation}\label{step2_quasi}
                \HH^d(E_i \cap W_f) \leq \kappa \HH^d(f(E_i \cap W_f)) + h \HH^d(E_i \cap B(0,\tfrac{1}{2})) + \varepsilon_i.
            \end{equation}
    \end{enumerate}
    Then we show that
    \begin{equation}
        \HH^d(E \cap \mathopen{]}-\tfrac{1}{2},\tfrac{1}{2}\mathclose{[}^n) \leq C \int_{G(d,n)} \HH^d(p_V(E \cap \mathopen{]}-1,1\mathclose{[}^n)) \, \mathrm{d}V.
    \end{equation}
    We want to proceed as in the proof of the reduced problem of Corollary \ref{prop_density} but there is a difference. We are not going to make Federer--Fleming projection of the sets $E_i$ but of the set $E$ instead.

    We fix $q \in \N^*$. For $0 \leq k < 2^q$, let $K_k$ be the set of dyadic cells of sidelength $2^{-q}$ subdivising the cube $(1 - k2^{-q}) [-1,1]^n$ but which are not lying in its boundary. The set $K_k$ is a finite $n$-complex subordinated to $E_n$ and
    \begin{subequations}
        \begin{align}
            \abs{K_k}   &   = (1 - k2^{-q}) [-1,1]^n,\\
            U(K_k)      &   = (1 - k2^{-q}) \mathopen{]}-1,1\mathclose{[}^n.
        \end{align}
    \end{subequations}
    We abbreviate $U(K_k)$ as $U_k$; in particular $U_0 = \mathopen{]}-1,1\mathclose{[}^n$. The cells of $K_k$ have bounded overlap: each point of $\abs{K_k}$ belongs to at most $3^n$ cells $A \in K_k$. It is clear that $K_{k+1}$ is a subcomplex of $K_k$ (the sidelength $2^{-q}$ does not depend on $k$). We deduce that
    \begin{align}
        \abs{K_k} \setminus U_{k+1} &= \bigcup \set{A \in K_k | A \in K_k \setminus K_{k+1}}\\
                                    &= \bigcup \set{A \in K_k | A \cap U_k \ne \emptyset}.
    \end{align}

    Let $\phi$ be a Federer--Fleming projection of $E \cap U_k$ in $K_k$. We can assume that $\phi$ is $C$-Lipschitz thanks to Lemma \ref{lem_semi} and Remark \ref{rmk_af_semi}. We apply the quasiminimality of $E$ with respect to $\phi$ in $B = B(0,\sqrt{n})$. For $i$ big enough, $\varepsilon_i \leq h \HH^d(E_i \cap B(0,\tfrac{1}{2}))$ so the term $\varepsilon_i$ is be absorbed in $h \HH^d(E_i \cap B(0,\tfrac{1}{2}))$. We also assume $h \leq \tfrac{1}{4}$ so that
    \begin{equation}
        \HH^d(E_i \cap \mathopen{]}-\tfrac{1}{2},\tfrac{1}{2}\mathclose{[}^n)) \leq C \HH^d(\phi(E_i \cap U_k)).
    \end{equation}
    Now we prove that there exists an open set $O \subset U_{k+1}$ which contains $E \cap U_{k+1}$ and such that
    \begin{equation}\label{assumption_O}
        \HH^d(\phi(O)) \leq C 2^q \int_{G(d,n)} \HH^d(p_V(E \cap U_0)) \, \mathrm{d}V + \varepsilon
    \end{equation}
    where $\varepsilon > 0$ is an error term which is small compared to $\HH^d(E_i \cap U_k)$. We will use this to decompose $E_i \cap U_k$ in two parts: $E_i \cap O$ and $E_i \cap U_k \setminus O$. By the properties of the Federer--Fleming projection, there exists an open set $O' \subset U_k$ containing $E \cap U_k$ and such that
    \begin{equation}
        \phi(O') \subset \abs{K_k} \setminus \bigcup \set{\mathrm{int}(A) | A \in K_k,\ \mathrm{dim} \, A > d}.
    \end{equation}
    As $K_{k+1} \subset K_k$, the function $\phi$ preserves the cells of $K_{k+1}$ and thus $\phi(U_{k+1}) \subset \abs{K_{k+1}}$. We deduce that
    \begin{align}
        \phi(O' \cap U_{k+1})   &   \subset \abs{K_{k+1}} \setminus \bigcup \set{\mathrm{int}(A) | A \in K_k,\ \mathrm{dim} \, A > d}\\
                                &   \subset U_k \setminus \bigcup \set{\mathrm{int}(A) | A \in K_k,\ \mathrm{dim} \, A > d}\\
                                &   \subset \bigcup \set{\mathrm{int}(A) | A \in K_k,\ \mathrm{dim} \, A \leq d}.
    \end{align}
    We denote the set $\bigcup \set{\mathrm{int}(A) | A \in K_k,\ \mathrm{dim} \, A \leq d}$ by $S_k$. Since $O'$ contains $E \cap U_k$, we also have $\phi(E \cap U_{k+1}) \subset S_k$. We recall that for $A \in K_k^d$,
    \begin{equation}
        \HH^d(\phi(E) \cap A) \leq C \int_{G(d,n)} \HH^d(p_V(E \cap U_0)) \, \mathrm{d}V.
    \end{equation}
    As $\phi(E \cap U_{k+1}) \subset S_k$ and $K_k^d$ contains at most $C 2^q$ cells, we deduce that
    \begin{equation}
        \HH^d(\phi(E \cap U_{k+1})) \leq C 2^q \int_{G(d,n)} \HH^d(p_V(E \cap U_0)) \, \mathrm{d}V.
    \end{equation}
    The set $O$ that we intend to build is a small neighborhood of $E \cap U_{k+1}$ for which this estimate still holds with an error term $\varepsilon$. First, we observe that $\HH^d \mres S_k$ is a Radon measure because $K$ is finite. According to the regularity properties of Radon measure, there exists an open set $O''$ containing $\phi(E \cap U_{k+1})$ such that
    \begin{equation}
        \HH^d(S_k \cap O'') \leq \HH^d(\phi(E \cap U_{k+1})) + \varepsilon.
    \end{equation}
    Since $\phi(O' \cap U_{k+1}) \subset S_k$, the set $O = O' \cap U_{k+1} \cap \phi^{-1}(O'')$ is a solution: it is an open subset of $U_{k+1}$ which contains $E \cap U_{k+1}$ and which satisfies
    \begin{align}
        \HH^d(\phi(O))  &\leq \HH^d(\phi(S_k \cap O''))\\
                        &\leq \HH^d(\phi(E \cap U_{k+1})) + \varepsilon\\
                        &\leq C 2^q \int_{G(d,n)} \HH^d(p_V(E \cap U_0)) \, \mathrm{d}V + \varepsilon.
    \end{align}
    In conclusion,
    \begin{align}
        \HH^d(E_i \cap \mathopen{]}-\tfrac{1}{2},\tfrac{1}{2}\mathclose{[}^n))  &\leq C \HH^d(\phi(E_i \cap U_k))\\
                                                                                &\leq C \HH^d(\phi(O)) + C\HH^d(E_i \cap U_k \setminus O))\\
                                                                                \begin{split}
                                                                                &\leq C 2^q \int_{G(d,n)} \HH^d(p_V(E \cap U_0)) \, \mathrm{d}V + C\varepsilon\\
                                                                                &\qquad + C \HH^d(\phi(E_i \cap U_k \setminus O)).
                                                                                \end{split}
    \end{align}
    As $\liminf_i \HH^d(E_i \cap B(0,\tfrac{1}{2})) > 0$, we can take $\varepsilon$ small enough (independently from $i$) so that $C\varepsilon < \tfrac{1}{2} \liminf_i \HH^d(E_i \cap B(0,\tfrac{1}{2}))$ and then assume $i$ big enough so that
    \begin{equation}
        C\varepsilon \leq \tfrac{1}{2} \HH^d(E_i \cap B(0,\tfrac{1}{2})).
    \end{equation}
    Thus,
    \begin{multline}
        \HH^d(E_i \cap \mathopen{]}-\tfrac{1}{2},\tfrac{1}{2}\mathclose{[}^n) \leq C 2^q \int_{G(d,n)} \HH^d(p_V(E \cap U_0)) \, \mathrm{d}V +\\C \HH^d(\phi(E_i \cap U_k \setminus O))
    \end{multline}
    and since $\phi$ is $C$-Lipschitz,
    \begin{multline}
        \HH^d(E_i \cap \mathopen{]}-\tfrac{1}{2},\tfrac{1}{2}\mathclose{[}^n) \leq C 2^q \int_{G(d,n)} \HH^d(p_V(E \cap U_0)) \, \mathrm{d}V\\ + C \HH^d(E_i \cap U_k \setminus O).
    \end{multline}
    We pass to the limit $i \to \infty$,
    \begin{multline}
        \HH^d(E \cap \mathopen{]}-\tfrac{1}{2},\tfrac{1}{2}\mathclose{[}^n) \leq C 2^q \int_{G(d,n)} \HH^d(p_V(E \cap U_0)) \, \mathrm{d}V\\ + C \HH^d(E \cap \abs{K_k}\setminus O).
    \end{multline}
    The open set $O$ contains $E \cap U_{k+1}$ so
    \begin{multline}
        \HH^d(E \cap U_k) \leq C 2^q \int_{G(d,n)} \HH^d(p_V(E \cap U_0)) \, \mathrm{d}V\\ + C \HH^d(E \cap \abs{K_k} \setminus U_{k+1})).
    \end{multline}
    This is (almost) the same inequality as (\ref{FF3a}) in the proof of Proposition \ref{prop_rect}. We can conclude in the same way by a Chebyshev argument.

    \emph{Step 3. We show that for all open balls $B$ of scale $\leq s$ in $X$, for all sliding deformations $f$ of $E$ in $B$,}
    \begin{equation}
        \mu(W_f) \leq \kappa \II(f(W_f)) + h \mu(h B),
    \end{equation}
    where $W_f = \set{x \in E | f(x) \ne x}$. This step relies mainly on the technical Lemma \ref{technical_lemma}. Let us fix an open ball $B$ of scale $\leq s$ in $X$. Let $f$ be a sliding deformation of $E$ in $B$. Let $\varepsilon > 0$. Let $K$ be a compact subset of $E \cap W_f$ such that $\HH^d(E \cap W_f \setminus K) \leq \varepsilon$. There exists $\delta > 0$ such that $\abs{f - \mathrm{id}} > \delta$ on $K$. According to Lemma \ref{sliding_alternative}, there exists a global sliding deformation $f_1$ in $B$ such that $\abs{f_1 - f} \leq \frac{\delta}{2}$, $E \cap W_{f_1} \subset \subset W_f$ and
    \begin{equation}\label{fB}
        \HH^d(f_1(W_f) \setminus f(W_f)) \leq \varepsilon.
    \end{equation}
    In particular, $K \subset W_{f_1}$ because $\abs{f - f_1} \leq \frac{\delta}{2} < \abs{f - \mathrm{id}}$ on $K$. Since $\HH^d(E \cap W_f \setminus K) \leq \varepsilon$, this implies
    \begin{equation}\label{fA}
        \HH^d(E \cap W_f \setminus W_{f_1}) \leq \varepsilon.
    \end{equation}
    We are going to apply Lemma \ref{technical_lemma}. There exists a global sliding deformation $g$ in $B$ (whose Lipschitz constant does not depend on $\varepsilon$) and an open set $V \subset W_{f_1}$ such that $W_g \subset W_{f_1}$, $\abs{g - f_1} \leq \frac{\delta}{4}$ and
    \begin{subequations}
        \begin{align}
            &   \HH^d(E \cap W_{f_1} \setminus V)   \leq \varepsilon\label{gA}\\
            &   \II(g(V))                           \leq \II(f_1(E \cap W_{f_1})) + \varepsilon.\label{gB}
        \end{align}
    \end{subequations}
    Let us draw a few consequences. It is straightforward that $E \cap \overline{W_g} \subset W_f$. Moreover, $K \subset W_g$ because $\abs{g - f} \leq \tfrac{3\delta}{4} < \abs{f - \mathrm{id}}$ on $K$. The conditions (\ref{fA}), (\ref{gA}) imply
    \begin{align}
        \mu(W_f \setminus V)    &   \leq \mu(W_f \setminus W_{f_1}) + \mu(W_{f_1} \setminus V)\\
                                &   \leq C\varepsilon.
    \end{align}
    and the conditions (\ref{fB}), (\ref{gB}) imply
    \begin{align}
        \II(g(V))   &   \leq \II(f_1(E \cap W_{f_1})) + \varepsilon\\
                    &   \leq \II(f_1(W_f)) + \varepsilon\\
                    &   \leq \II(f(W_f)) + 2\varepsilon.
    \end{align}
    Now, we apply the quasiminimality of $E_i$ with respect to $g$ in $B$,
    \begin{equation}
        \II(E_i \cap W_g) \leq \kappa \II(g(E_i \cap W_g)) + h \II(E_i \cap hB) + \varepsilon_i.
    \end{equation}
    By construction,
    \begin{align}
        \II(g(E_i \cap W_g))    &   \leq \II(g(V)) + \II(g(E_i \cap W_g \setminus V))\\
                                &   \leq \II(f(W_f)) + C\norm{g}_L \II(E_i \cap W_g \setminus V) + 2\varepsilon,
    \end{align}
    where $\norm{g}_L$ is the Lipschitz constant of $g$. Passing to the limit in $i$, we have
    \begin{align}
        \limsup_i \II(g(E_i \cap W_g))  &   \leq \II(f(W_f)) + C\norm{g}_L \mu(\overline{W_g} \setminus V) + 2\varepsilon\\
                                        &   \leq \II(f(W_f)) + C\norm{g}_L \mu(W_f \setminus V) + 2\varepsilon\\
                                        &   \leq \II(f(W_f)) + C\norm{g}_L \varepsilon + 2\varepsilon.
    \end{align}
    On the other hand, we have
    \begin{equation}
        \liminf_i \mu(E_i \cap W_g) \geq \mu(W_g) \geq \mu(K).
    \end{equation}
    We conclude that
    \begin{align}
        \mu(W_f)    &   \leq \mu(K) + \varepsilon\\
                    &   \leq \kappa \II(f(W_f)) + h \mu(h\overline{B}) + C \varepsilon (\norm{g}_L + 1).
    \end{align}
    The Lipschitz constant of $g$ does not depend on $\varepsilon$ so we can make $\varepsilon \to 0$ to obtain
    \begin{equation}
        \mu(W_f) \leq \kappa \II(f(W_f)) + h \mu(h \overline{B}).
    \end{equation}
    To replace $\overline{B}$ by $B$ in the lower-order term, we can apply the previous reasoning in a slightly smaller ball $B'$ where $f$ is still a sliding deformation.

    \emph{Step 4. We show that $\mu \leq (\kappa + h) \II \mres E$.}
    Here it is convenient to assume that there exists a continuous function $i\colon X \times G(d,n) \to \mathopen{]}0,\infty\mathclose{[}$ such that for all $\HH^d$ measurable, $\HH^d$ finite, $\HH^d$ rectifiable set $F \subset X$,
    \begin{equation}
        \II(F) = \int_E \! i(x,T_xF) \, \mathrm{d}\HH^d
    \end{equation}
    where $T_xF$ is the linear tangent plane of $F$ at $x$.

    As $E$ is Borel and $\HH^d$ locally finite, the measure $\II \mres E$ is a Radon measure. Then according to the Differentiation Lemma \cite[Lemma 2.13]{Mattila}, it suffices to prove that for $\HH^d$-a.e. $x \in E$,
    \begin{equation}
        \liminf_{r \to 0} \frac{\mu(B(x,r))}{\II(E \cap B(x,r))} \leq \kappa + h
    \end{equation}
    We know that for $\HH^d$-a.e. $x \in E$, there exists an embedded submanifold $M \subset \R^n$ such that
    \begin{equation}
        \lim_{r \to 0} r^{-d} \HH^d(B(x,r) \cap (E \Delta M)) = 0.
    \end{equation}
    We fix such point $x \in E$ and to simplify the notation we assume $x = 0$. We denote the tangent plane of $M$ at $0$ by $V_0$ and the value $i(0,V_0)$ by $i_0$. Let $0 < \varepsilon_0 < 1$. By continuity of $i$, there exists $0 < r_0 \leq 1$ such that for all $x \in B(0,r_0)$ and for all $W \in B(V,r_0)$,
    \begin{equation}
        (1 + \varepsilon_0)^{-1} i_0 \leq i(x,W) \leq (1 + \varepsilon_0) i_0.
    \end{equation}
    Let $0 < r_1 \leq r_0$ be a radius such that for all $0 < r \leq r_1$,
    \begin{equation}\label{MD}
        \HH^d(B(0,r) \cap (E \Delta M)) \leq \varepsilon_0 r^d
    \end{equation}
    and for all $x \in M \cap B(0,r_1)$,
    \begin{equation}\label{MV}
        \mathrm{d}(T_xM,V) \leq \tfrac{1}{2}r_0.
    \end{equation}
    We have in particular for all $0 < r \leq r_1$,
    \begin{equation}\label{M}
        i_0 (1 + \varepsilon_0)^{-1} \leq \frac{\II(M \cap B(0,r))}{\HH^d(M \cap B(0,r))} \leq i_0(1 + \varepsilon_0)
    \end{equation}
    Let $0 < r \leq \min \set{r_0,r_1}$ be a small radius. For all compact set $K \subset B(0,r)$ and for all $0 < \varepsilon \leq \varepsilon_0$, we are going to build a global sliding deformation $f$ in $B(0,r)$ such that $K\setminus \set{0} \subset W_f$ and $\norm{Df - \mathrm{id}} \leq \varepsilon$. We explain why this solves step 4. According to step 3, we have
    \begin{equation}\label{step4}
        \mu(K) \leq \kappa \II^d(f(E \cap W_f)) + h \II^d(E \cap B(0,hr)).
    \end{equation}
    Next, we estimate $\II^d(f(E \cap W_f))$. For all $x \in B(0,r)$, the differential $Df(x)$ is inversible because $\norm{Df(x) - \mathrm{id}} < 1$ and by Lemma \ref{lem_grass_iso}
    \begin{equation}
        \mathrm{d}(Df(x)(T_xM),T_xM) \leq \frac{\varepsilon}{1 - \varepsilon}.
    \end{equation}
    We assume $\varepsilon$ small enough so that $\frac{\varepsilon}{1 - \varepsilon} \leq \tfrac{r_0}{2}$ and then by (\ref{MV}), for all $x \in B(0,r)$,
    \begin{equation}
        \mathrm{d}(Df(x)(T_xM),V)   \leq r_0.
    \end{equation}
    According to the definition of $r_0$, we have
    \begin{equation}
        \II(f(M \cap B(0,r)) \leq i_0 (1 + \varepsilon_0) \HH^d(f(M \cap B(0,r))
    \end{equation}
    and by fact that $f$ is $(1+\varepsilon_0)$-Lipschitz and (\ref{M}),
    \begin{align}
        \II(f(M \cap B(0,r)) &\leq i_0 (1 + \varepsilon_0)^{d+1} \HH^d(M \cap B(0,r))\\
                             &\leq (1 + \varepsilon_0)^{d+2} \II(M \cap B(0,r)).
    \end{align}
    Plugging (\ref{MD}), we finally estimate
    \begin{equation}
        \II(f(E \cap B(0,r)) \leq (1 + \varepsilon_0)^{d+2} \II(E \cap B(0,r)) + C \varepsilon_0 r^d.
    \end{equation}
    Now, (\ref{step4}) yields
    \begin{equation}
        \mu(K) \leq \kappa (1 + \varepsilon_0)^{d+2} \II(E \cap B(0,r)) + h \II(E \cap B(0,hr)) + C \varepsilon_0 r^d.
    \end{equation}
    We make $K \to B(0,r)$, $r \to 0$ and $\varepsilon_0 \to 0$ to deduce
    \begin{equation}
        \limsup_{r \to 0} \frac{\mu(B(0,r))}{\II(E \cap B(0,r))} \leq \kappa + h.
    \end{equation}
    We detail the construction of $f$. If $0 \in E \setminus \Gamma$, the construction is easier because we can work with a small radius $r > 0$ such that $B(0,r) \cap \Gamma = \emptyset$ and there is no need to check the sliding constraint. From now on, we assume $0 \in E \cap \Gamma$. The boundary $\Gamma$ is locally diffeomorphic to a cone so for $r > 0$ small enough, there exists an open set $O$ containing $0$, a $C^1$ diffeomorphism $\GT\colon B(0,r) \to O$ and a closed cone $S \subset \R^n$ centered at $0$ such that $\GT(0) = 0$ and $\GT(\Gamma \cap B(0,r)) = S \cap O$.

    Let $W$ be open set such that $K \subset W \subset \subset B(0,r)$. We define $K'=\GT(K)$ and $W' = \GT(W)$. Note that there exists $r' \geq 1$ such that $W' \subset B(0,r')$. Let $\chi \in C^1_c(W')$ be such that $0 \leq \chi \leq 1$ and $\chi = 1$ in $K'$. For a certain $\delta > 0$ (it will be taken small enough later), we introduce
    \begin{equation}
        g = \mathrm{id} + \delta \chi \mathrm{id}.
    \end{equation}
    It is clear that $g$ is $C^1$ on $\R^n$, $g(S) \subset S$ and $\abs{g - \mathrm{id}} \leq \delta r'$. We assume $\delta$ small so that for all $y \in W'$, $\overline{B}(y,\delta r') \subset O$. As a consequence, $g(W') \subset O$ but we also have $g(O) \subset O$ since $g = \mathrm{id}$ in $O \setminus W'$. The function $g$ is global sliding deformation along $S$ in $O$ associated to the homotopy
    \begin{equation}
        G_t = \mathrm{id} + t \delta \chi \mathrm{id}.
    \end{equation}
    Finally, the function
    \begin{equation}
        f = \GT^{-1} \circ g \circ \GT
    \end{equation}
    is a global sliding deformation along $\Gamma$ in $B(0,r)$. As $g \ne \mathrm{id}$ in $K' \setminus \set{0}$, we have $K \setminus \set{0} \subset f$. Next, we estimate $\norm{Df - \mathrm{id}}$. We take the convention that when a point $x \in B(0,r)$ is given, the letter $y$ denotes $\GT(x)$. For $x \in B(0,r)$, we have
    \begin{equation}
        Df(x)   = D \GT^{-1}(g(y)) \circ Dg(y) \circ D \GT(x)
    \end{equation}
    It is clear that if $x \notin W$, $Df(x) = \mathrm{id}$ so we only focus about the points $x \in W$. Note that there exists $L \geq 1$ such that for all $x \in W$, $\norm{D\GT(x)} \leq L$ and for all $y \in W'$, $\norm{D\GT^{-1}(y)} \leq L$. We see that for $y \in W'$,
    \begin{equation}
        \norm{Dg(y) - \mathrm{id}} \leq \delta \left(\abs{\chi}_\infty + \abs{\nabla \chi}_\infty r'\right)
    \end{equation}
    so we can assume $\delta$ small so that for all $y \in W'$,
    \begin{equation}
        \norm{Dg(y) - \mathrm{id}} \leq \varepsilon.
    \end{equation}
    In particular, for $x \in W$,
    \begin{equation}
        \norm{Df(x) - D \GT^{-1}(g(y)) \circ D \GT(x)} \leq L^2 \varepsilon
    \end{equation}
    The application $y \mapsto D\GT^{-1}(y)$ is uniformly continuous on $W' \cup g(W')$ so we can assume $\delta$ small so that for all $y, z \in W \cup g(W)$, the condition $\abs{y - z} \leq \delta r'$ implies
    \begin{equation}
        \norm{D\GT^{-1}(z) - D\GT^{-1}(y)} \leq \varepsilon
    \end{equation}
    In particular, for all $x \in W$, $\norm{D\GT^{-1}(g(y)) - D\GT^{-1}(y)} \leq \varepsilon$ so
    \begin{equation}
        \norm{D \GT^{-1}(g(y)) \circ D \GT(x) - \mathrm{id}} \leq L \varepsilon.
    \end{equation}
    We conclude that for all $x \in W$, $\norm{Df(x) - \mathrm{id}} \leq 2 L^2\varepsilon$. We might want to replace $\varepsilon$ by a smaller constant in the previous reasoning so that $\norm{Df - \mathrm{id}} \leq \varepsilon$. This concludes the proof of step 4.

    We want to give another proof of this step under different assumptions. We forget about the fact that $\II$ is induced by a continuous integrand but we assume that there exists $\kb \geq 1$ such that the following condition holds true: for ball open ball $B$ of scale $\leq s$ in $X$, there exists a sequence $(\varepsilon_i) \to 0$ such that for all global sliding deformations $f$ in $B$,
    \begin{equation}\label{Bquasi}
        \II(E_i \cap B) \leq \kb \II(f(E_i \cap B)) + \varepsilon_i.
    \end{equation}
    Then we can proceed as in step 3 to show that $\mu(B) \leq \kb \II(B)$. Here are the details. We consider $f = \mathrm{id}$. Let $\varepsilon > 0$. According to Lemma \ref{technical_lemma}, there exists a global sliding deformation $g$ in $B$ (whose Lipschitz constant does not depend on $\varepsilon$) and an open set $V \subset B$ such that
    \begin{subequations}
        \begin{align}
            &   \HH^d(E \cap B \setminus V) \leq \varepsilon\\
            &   \II(g(V))                   \leq \II(E \cap B) + \varepsilon.
        \end{align}
    \end{subequations}
    Now, we apply (\ref{Bquasi}) with respect to $g$ in $B$,
    \begin{equation}
        \II(E_i \cap B) \leq \kb \II(g(E_i \cap B)) + \varepsilon_i.
    \end{equation}
    By construction,
    \begin{align}
        \II(g(E_i \cap B))  &   \leq \II(g(V)) + \II(g(E_i \cap B \setminus V))\\
                            &   \leq \II(E \cap B) + C\norm{g}_L \II(E_i \cap B \setminus V) + \varepsilon,
    \end{align}
    where $\norm{g}_L$ is the Lipschitz constant of $g$. Passing to the limit in $i$, we have
    \begin{align}
        \limsup_i \II(g(E_i \cap B))    &   \leq \II(E \cap B) + C\norm{g}_L \mu(\overline{B} \setminus V) + \varepsilon\\
                                        &   \leq \II(E \cap B) + C\norm{g}_L \mu(\partial B) + C \norm{g}_L \varepsilon + \varepsilon.
    \end{align}
    On the other hand, we have
    \begin{equation}
        \liminf_i \mu(E_i \cap B) \geq \mu(B).
    \end{equation}
    We conclude that
    \begin{equation}
        \mu(B) \leq \kb \II(E \cap B) + C\kb(\norm{g} \mu(\partial B) + \norm{g}\varepsilon + \varepsilon).
    \end{equation}
    The Lipschitz constant of $g$ does not depend on $\varepsilon$ so we can make $\varepsilon \to 0$ to obtain
    \begin{equation}
        \mu(B) \leq \kb \II(E \cap B) + C\kb \norm{g} \mu(\partial B).
    \end{equation}
    To get rid of $\mu(\partial B)$, we can apply the previous reasoning in slightly smaller balls $B'$ where $f$ is still a sliding deformation and such that $\mu(\partial B') = 0$.

    \emph{Step 5. We show that $\II \mres E \leq \mu$.}
    According to the Differentiation Lemma \cite[Lemma 2.13]{Mattila}, it suffices to prove that for $\HH^d$-ae. $x \in E$,
    \begin{equation}\label{mu_goal}
        \limsup_r \frac{\mu(B(x,r))}{\II(E \cap B(x,r))} \geq 1.
    \end{equation}
    We are going to work with points $x \in E$ satisfying nice properties. The set $E$ is $\HH^d$ rectifiable so for $\HH$-ae. $x \in E$, there exists a (unique) d-plane $V_x$ passing through $x$ and a constant $\theta > 0$ such that 
    \begin{equation}\label{mu_V}
        r^{-d} \mu_{x,r} \rightharpoonup \theta \HH^d \mres V_x,
    \end{equation}
    where $\mu_{x,r}\colon A \mapsto \mu(x + r(A - x))$ and the arrow $\rightharpoonup$ denotes the weak convergence of Radon measures as $r \to 0$ (see \cite[Theorem 4.8]{DL}). By the properties of admissible energies, for $\HH^d$-ae. $x \in E$,
    \begin{equation}\label{II_V}
        \lim_r \frac{\II(E \cap B(x,r))}{\II(V_x \cap B(x,r))} = 1.
    \end{equation}
    Moreover, for $\mu$-ae. $x \in \Gamma$,
    \begin{equation}\label{mu_gamma}
        \lim_{r \to 0} r^{-d} \mu(B(x,r) \setminus \Gamma) = 0.
    \end{equation}
    This can be justified by applying \cite[Theorem 6.2(2)]{Mattila} to the set $E \setminus \Gamma$. For $x \in \Gamma$ satisfying (\ref{mu_V}) and (\ref{mu_gamma}), one can show (this was done in Lemma \ref{technical_lemma}) that for all $0 < \varepsilon \leq 1$, there exists $r > 0$ such that
    \begin{equation}\label{V_gamma}
        V_x \cap B(x,r) \subset \set{y \in X | \mathrm{d}(y,\Gamma) \leq \varepsilon \abs{y - x}}.
    \end{equation}
    Now, we fix $x$ which satisfies (\ref{mu_V}), (\ref{II_V}) and such that either $x \notin \Gamma$, either $x \in \Gamma$ and (\ref{mu_gamma}). We aim to prove (\ref{mu_goal}) at $x$. According to (\ref{II_V}), the statement (\ref{mu_goal}) is equivalent to
    \begin{equation}\label{mu_goal0}
        \limsup_r \frac{\mu(B(x,r))}{\II(V_x \cap B(x,r))} \geq 1.
    \end{equation}
    According to the definition of admissible energies, there exists $r_0 > 0$ and $\varepsilon_0 \colon \mathopen{]}0,r_0\mathclose{]} \to \R^+$ such that $\lim_{r \to 0} \varepsilon(r) = 0$ and
    \begin{equation}
        \II(V_x \cap B(x,r)) \leq \II(S \cap B(x,r)) + \varepsilon(r)r^d
    \end{equation}
    when $0 < r \leq r_0$ and $S \subset \overline{B}(x,r)$ is a compact $\HH^d$ finite set which cannot be mapped into $V \cap \partial B(x,r)$ by a Lipschitz mapping $\psi\colon \overline{B}(x,r) \to \overline{B}(x,r)$ with $\psi = \mathrm{id}$ on $V \cap \partial B(x,r)$. We plan to take advantage of this principle and to show that there exists $r_1 > 0$ such that for $0 < r \leq r_1$ and for $i$ big enough (depending on $r$),
    \begin{equation}
        \II(V_x \cap B(x,r)) \leq \II(E_i \cap B(x,r)) + \varepsilon(r)r^d.
    \end{equation}
    This would imply (\ref{mu_goal0}). We could assume that $E_i \cap \overline{B}(x,r)$ is mapped into $V \cap \partial B(x,r)$ by a Lipschitz mapping $\psi\colon \overline{B}(x,r) \to \overline{B}(x,r)$ with $\psi = \mathrm{id}$ on $V \cap \partial B(x,r)$, turn this map into a sliding deformation and get a contradiction. However, we would have to make a transition between this map and $\mathrm{id}$ away from $B(x,r)$. As we do not know the Lispchitz constant of $\psi$, we cannot measure the image of $E_i$ in the transition area. Therefore, we have to clean the set $E_i$ first in the transition area.

    To simplify the notations, we assume that $x = 0$, we denote $V_x$ by $V$ and assume that
    \begin{equation}
        V  = \set{x \in \R^n | x_i = 0 \ \text{for} \ i=d+1,\ldots,n}.
    \end{equation}
    In particular, we identify $V$ to $\R^d$. For $r > 0$, let $B_r = B(0,r)$ and $V_r =  V \times (B_r \cap V^\perp)$. For $t > 0$, the symbol $tB_r$ denotes $B_{tr}$ and $tV_r$ denotes $V_{tr}$. We fix $0 < a \leq \frac{1}{3}$ (close to $0$). We define layers surrounding the disk $B_r \cap V$,
    \begin{align}
        R^0_r   &   = B_r \cap aV_r\\
        R^1_r   &   = (1+a)B_r \cap 2aV_r\\
        R^2_r   &   = (1+2a)B_r \cap 3aV_r\\
        R^3_r   &   = (1+3a)B_r \cap 4aV_r.
    \end{align}
    and we define an area of $2B_r$ away from $V$,
    \begin{equation}
        S_r = 2B_r \setminus \tfrac{a}{2}V_r.
    \end{equation}
    We are going to the clean set $E_i$ in two steps. First, to we want to get $\HH^d(E_i \cap S_r) = 0$ and then, that
    \begin{equation}
        E_i \cap R^3_r \subset B_r \cup V \cup \R^n \setminus (1+a)B_r.
    \end{equation}

    We clean $E_i$ in $S_r$ by projecting it to $(d-1)$-skeleton with a Federer--Fleming projection. We introduce an open neighborhood of $S_r$,
    \begin{equation}
        T_r = 3B_r \setminus \tfrac{a}{3}\overline{V_r}.
    \end{equation}
    The set $E_i$ has a very small $\HH^d$ measure in $T_r$ when $i$ goes to $\infty$. More precisely, by the weak convergence $\II \mres E_i \rightharpoonup \mu$ in $X$, we have
    \begin{equation}
        \limsup_i \HH^d(E_i \cap T_r) \leq C\mu(\overline{T_r})
    \end{equation}
    and by the weak convergence $r^{-d} \mu_{0,r} \rightharpoonup \theta \HH^d \mres V$,
    \begin{align}
        \lim_r r^{-d} \mu(\overline{T_r})   &   =\theta \HH^d \mres V(3\overline{B_1} \setminus \tfrac{a}{3}V_1)\\
                                            &   = 0.
    \end{align}
    Now we fix $\varepsilon > 0$ (to be chosen later). For $r > 0$ small enough and for $i$ big enough (depending on $r$),
    \begin{equation}
        \HH^d(E_i \cap \overline{T_r}) < \varepsilon r^d.
    \end{equation}
    If $\varepsilon$ is small enough, we will build a global sliding deformation $\phi$ in $T_r$ such that
    \begin{subequations}
        \begin{align}
        &   \HH^d(\phi(E_i \cap T_r)) \leq C \HH^d(E \cap T_r)\\
        &   \HH^d(\phi(E_i \cap T_r) \cap S_r) = 0.
        \end{align}
    \end{subequations}
    The map $\phi$ will be a Federer--Fleming projection of the set $E_i \cap S_r$. However, our statement of Federer--Fleming projections can only preserve rigid boundaries. This is the reason why we need to work with the rigid version of our boundary . According to Definition \ref{defi_whitney}(iii*), there exists $\rho > 0$, an open set $O \subset \R^n$, a $C$-bilipschitz bijective map $\GT\colon B_\rho \to O$, an integer $k \in \N$ and a subset $S \subset \mathcal{E}_n(k)$ such that $O \subset B(0,2^{-k-1})$ and $\GT(\Gamma \cap B_\rho) = \abs{S} \cap O$.

    The radius $\rho$ does not depend on $r$ and $i$ and it will be fixed during the rest of the proof. We assume $r \leq \tfrac{1}{3}\rho$ so that $T_r \subset B_\rho$ and we introduce
    \begin{subequations}
        \begin{align}
            \tilde{E_i} &= \GT(E_i \cap B_\rho),\\
            \tilde{S_r} &= \GT(S_r),\\
            \tilde{T_r} &= \GT(T_r).
        \end{align}
    \end{subequations}
    Since $T(B_\rho) \subset B(0,2^{-k-1})$, we have $r \leq \rho \leq C2^{-k}$. Thus, there exists $p \geq k$ be such that $2^{-p} \leq Cr$. In addition, we consider $q \in \N$ (to be chosen later). We consider the complex of $\R^n$ composed of all the cells of the form
    \begin{equation}
        B = \prod_{i=1}^n [p_i,p_i + (2^{-p-q})\alpha_i]
    \end{equation}
    where $p \in (2^{-p-q})\Z^n$ and $\alpha \in \set{0,1}^n$. This describes a uniform grid of sidelength $2^{-p-q}$ as in Example \ref{grid_example} of Section \ref{rigid_boundaries}. Then, we consider the subcomplex $K$ composed of such cells $B$ such that
    \begin{equation}
        B \cap \tilde{S_r} \ne \emptyset.
    \end{equation}
    We are going to justify that if $q$ is big enough (depending on $n$, $\kappa$, $\Gamma$, $a$ but not $r$, $i$)
    \begin{equation}\label{Tinclusion}
        \tilde{S_r} \subset U(K) \subset \abs{K} \subset \tilde{T_r}.
    \end{equation}
where $U(K) = \bigcup \set{\mathrm{int}(B) | B \in K}$. For $x \in \tilde{S_r}$, there exists a cell $B$ in the uniform grid such that $x \in \mathrm{int}(B)$. Then by definition of $K$, we have $B \in K$. This proves that $\tilde{S_r} \subset U(K)$. For $B \in K$, we have $B \cap \tilde{S_r} \ne \emptyset$ so $\GT^{-1}(B) \cap S_r \ne \emptyset$. As $T_r$ contains a $(C^{-1}ar)$-neighborhood of $S_r$ and $\mathrm{diam}(B) \leq C2^{-p-q} \leq C 2^{-q}r$, we deduce that $\GT^{-1}(B) \subset T_r$ if $q$ is big enough (depending on $n$, $C$, $a$ but not $r$, $i$). The inclusions (\ref{Tinclusion}) gives in particular $\abs{K} \subset O$. As $O \subset B(0,\tfrac{1}{2})$ and $p + q \geq k$, we deduce that $K \preceq E_n(k)$ where
\begin{equation}
    E_n(k) = \Set{\prod_{i = 1}^n [0,2^{-k}\alpha_i] | \alpha \in \set{-1,0,1}^n}.
\end{equation}
This property ensures that a Federer--Fleming projection in $K$ preserves the cells of $E_n(k)$.

    Let $\tilde{\phi}$ be a Federer--Fleming projection of $\tilde{E_i} \cap U(K)$ in $K$. Thus, $\tilde{\phi}\colon \abs{K} \to \abs{K}$ is a Lipschitz map such that
    \begin{enumerate}
        \item for all $B \in K$, $\tilde{\phi}(B) \subset B$;
        \item $\tilde{\phi} = \mathrm{id}$ in $\abs{K} \setminus U(K)$;
        \item for all $B \in K$,
            \begin{equation}\label{phi_B}
                \HH^d(\tilde{\phi}(\tilde{E_i} \cap B)) \leq C \HH^d(\tilde{E_i} \cap B).
            \end{equation}
    \end{enumerate}
    Using the fact that the cells $B \in K$ have bounded overlap, that $\abs{K} \subset \tilde{T_r}$ and that $\GT$ is $C$-bilipschitz, the property (\ref{phi_B}) implies
    \begin{align}
        \HH^d(\tilde{\phi}(\tilde{E_i} \cap U(K)))  &\leq C \HH^d(\tilde{E_i} \cap U(K))\\
                                                    &\leq C \HH^d(\tilde{E_i} \cap \tilde{T_r})\\
                                                    &\leq C \HH^d(E_i \cap T_r)\label{tilde_phi}.
    \end{align}
    We recall that we have chosen $r$ small enough and $i$ big enough (depending on $r$) such that $\HH^d(E_i \cap T_r) \leq C \varepsilon r^d$. The cells of $K$ have a sidelength of $2^{-p-q} \leq C2^{-q}r$ so if $\varepsilon > 0$ is chosen small enough (depending on $n$, $C$, $a$, but not $r$, $i$), we can make additional projections in the $d$-dimensional cells $B \in K$ to get
    \begin{equation}\label{phi_build1}
        \HH^d(\tilde{\phi}(\tilde{E_i} \cap U(K)) \cap U(K)) = 0.
    \end{equation}
    Note that inequality (\ref{tilde_phi}) still holds: the points of $\tilde{\phi}(\tilde{E_i} \cap U(K))$ belonging to $U(K)$ have been sent to an $\HH^d$ negligible set, whereas the points belonging to $\partial \abs{K}$ have been left untouched.

    Now, we define
    \begin{equation}
        \phi = \GT^{-1} \circ \tilde{\phi} \circ \GT
    \end{equation}
    on $\GT^{-1}(\abs{K})$. Let us recall that $\GT^{-1}(\abs{K}) \subset T_r$ by (\ref{Tinclusion}). We extend $\phi$ over $\R^n$ by $\phi$ = $\mathrm{id}$ in $\R^n \setminus \GT^{-1}(\abs{K})$. Let us check that this extension is still Lipschitz. First, observe that $\GT^{-1}(\abs{K})$ is compact since $\GT$ is an homeomorphism between $B_\rho$ and $O$. Moreover, $\GT^{-1}(U(K))$ is an open set of $\R^n$. We deduce that for $x \in \GT^{-1}(\abs{K})$ and $y \in \R^n \setminus \GT^{-1}(\abs{K})$, there exists $z \in \GT^{-1}(\partial \abs{K})$ such that $z \in [x,y]$. Then,
    \begin{align}
        \abs{\phi(x) - \phi(y)} &\leq \abs{\phi(x) - \phi(z)} + \abs{\phi(z) - \phi(y)}\\
                                &\leq \abs{x - z} + L \abs{z - y},
    \end{align}
    where $L$ is the Lipschitz constant of $\phi$ over $\GT^{-1}(\abs{K})$. As $z \in [x,y]$,
    \begin{equation}
        \abs{x - z}, \abs{z - y} \leq \abs{x - y} 
    \end{equation}
    and our claim follows. As $\tilde{S_r} \subset U(K)$ (this is (\ref{Tinclusion})), it is clear by (\ref{phi_build1}) that
    \begin{equation}
        \HH^d(\phi(E_i) \cap S_r) = 0
    \end{equation}
    and we estimate
    \begin{align}
        \begin{split}
            \HH^d(\phi(E_i \cap T_r))   &\leq \HH^d(\phi(E_i \cap \GT^{-1}(U(K))))\\
                                        &\qquad + \HH^d(\phi(E_i \cap T_r \setminus \GT^{-1}(U(K)))
        \end{split}\\
                                        &\leq C\HH^d(\tilde{\phi}(\tilde{E_i} \cap U(K)) + \HH^d(E_i \cap T_r)\\
                                        &\leq C \HH^d(E_i \cap T_r).\label{phi_build0}
    \end{align}
    As $K \preceq E_n(k)$, $\tilde{\phi}$ preserves each cell of $E_n(k)$ and, by convexity, the homotopy $t \mapsto t \tilde{\phi} + (1-t) \mathrm{id}$ also preserves the cells of $E_n(k)$. Remember that there exists $S \subset E_n(k)$ such that $\GT(\Gamma \cap B_\rho) = \abs{S} \cap O$. Moreover, $\abs{K} \subset O$ so
    \begin{equation}
        \GT(\Gamma \cap \GT^{-1}(\abs{K)}) = \abs{S} \cap \abs{K}.
    \end{equation}
    We deduce that the previous homotopy preserves $\GT(\Gamma \cap \GT^{-1}(\abs{K}))$. We conclude that $\phi$ is a global sliding deformation along $\Gamma$ in $T_r$. Moreover, we have shown that
    \begin{align}
        &   \HH^d(\phi(E_i) \cap S_r) = 0\label{phi_info0},\\
        &   \HH^d(\phi(E_i \cap T_r)) \leq C \HH^d(E_i \cap T_r)\label{phi_info}.
    \end{align}

    Now, we proceed to the second step of the cleaning process. We introduce the partially defined function
    \begin{equation}
        \pi(x) =
        \begin{cases}
            x'         &   \text{in} \ ((1+\tfrac{3a}{2})B_r \setminus (1-\tfrac{a}{2})B_r) \cap \tfrac{a}{2}V_r\\
            \mathrm{id} &   \text{in} \ ((1-a)B_r) \cup (\R^n \setminus R^2_r)
        \end{cases}
    \end{equation}
    where $x'$ is the orthogonal projection of $x$ onto $V$. By definition of $V_r$, we have $\abs{\pi - \mathrm{id}} \leq \tfrac{a}{2}r$. Moreover, $\pi$ is $C$-Lipschitz because its two partial domains are at distance $\geq C^{-1}ar$. We also see that $\pi$ preserves $R^2_r$ on its (partial) domain. We can extend $\pi$ as a $C$-Lipschitz function $\pi\colon \R^n \to \R^n$ such that $\abs{\pi - \mathrm{id}} \leq \tfrac{a}{2}r$ and $\pi(R^2_r) \subset R^2_r$.

    Let us take some time to describe the image of a point $x \in \tfrac{a}{2}V_r$. If $x \in (1-\tfrac{a}{2})B_r$, then $\pi(x) \in B_r \cap aV_r$ because $\abs{\pi - \mathrm{id}} \leq \tfrac{a}{2}r$. If $x \in (1+\tfrac{3a}{2})B_r \setminus (1-\tfrac{a}{2})B_r$, then $\pi(x) \in V$. If $x \notin (1+\tfrac{3a}{2})B_r$, then $\pi(x) \notin (1+a)B_r$ because $\abs{\pi - \mathrm{id}} \leq \tfrac{a}{2}r$. In conclusion,
    \begin{equation}
        \pi(\tfrac{a}{2}V_r) \subset R^0_r \cup V \cup (\R^n \setminus (1+a)B_r).
    \end{equation}

    In the next part of the proof, we work with the clean set $E'_i = \pi(\phi(E_i))$. We want to justify beforehand that the measure $\II(E'_i \cap B_r)$ is close to $\II(E_i \cap B_r)$. As $\phi = \mathrm{id}$ outside $T_r$, we have
    \begin{align}
        \II(E'_i \cap B_r)  &\leq \II(\pi(\phi(E_i \setminus T_r)) \cap B_r) + \II(\pi(\phi(E_i \cap T_r)))\\
                            &\leq \II(\pi(E_i) \cap B_r)) + C\HH^d(E_i \cap T_r).
    \end{align}
    We use the fact that $\pi^{-1}(B_r) \subset (1 + \tfrac{a}{2})B_r$ (because $\abs{\pi - \mathrm{id}} \leq \tfrac{a}{2}r$), that $\pi$ is $C$-Lipschitz and that $\pi = \mathrm{id}$ in $(1-a)B_r$ to see that
    \begin{multline}
        \II(\pi(E_i) \cap B_r) \leq C\HH^d(E_i \cap (1+\tfrac{a}{2})B_r \setminus (1-a)B_r) + \II(E_i \cap B_r).
    \end{multline}
    In sum,
    \begin{multline}\label{Eprime_i}
        \II(E_i' \cap B_r) \leq \II(E_i \cap B_r) + C\HH^d(E_i \cap (1+\tfrac{a}{2})B_r \setminus (1-a)B_r) +\\C\HH^d(E_i \cap T_r).
    \end{multline}
    Here, the terms $\HH^d(E_i \cap (1+\tfrac{a}{2})B_r \setminus (1-a)B_r)$ and $\HH^d(E_i \cap T_r)$ will disappear when $i \to \infty$, $r \to 0$ and $a \to 0$.

    Now, the goal of the proof is to show that there exists $a_1 > 0$, $r_1 > 0$ such that for all $0 < a \leq a_0$, for all $0 < r \leq r_1$ there exists $i_1 \in \mathbf{N}$ such that for $i \geq i_1$,
    \begin{equation}
        \II(V \cap B_r) \leq \II(E'_i \cap B_r) + \varepsilon(r)r^d.
    \end{equation}
    Let us explain why this suffices to finish the proof.  Passing $i$ to the limit in (\ref{Eprime_i}) yields for $0 < a \leq a_0$ and for $0 \leq r \leq r_1$,
    \begin{equation}
        \II(V \cap B_r) \leq \mu(\overline{B_r}) + C\mu((1+\tfrac{a}{2})\overline{B_r} \setminus (1-a)B_r)) + C\mu(\overline{T_r}) + \varepsilon(r)r^d.
    \end{equation}
    Taking $a \to 0$, it then implies that
    \begin{equation}
        \II(V \cap B_r) \leq \mu(\overline{B_r}) + C\mu(\partial B_r) + C\mu(\overline{T_r}) + \varepsilon(r)r^d
    \end{equation}
    for $0 < r \leq r_1$. As $\mu(\partial B_r) = 0$ for arbitrary small $r$ and $\lim_{r \to 0} r^{-d}\mu(\overline{T_r}) = 0$, it follows that
    \begin{equation}
        \limsup_r \frac{\mu(B_r)}{\II(V \cap B_r)} \geq 1
    \end{equation}
    as wanted.

    We proceed by contradiction and assume that for arbitrary small $a$ and $r$ and for $i$ big enough (depending on $a$, $r$), $E'_i \cap B_r$ is mapped into $V \cap \partial B_r$ by a Lipschitz mapping $\psi_0\colon \overline{B_r} \to \overline{B_r}$ with $\psi_0 = \mathrm{id}$ on $V \cap \partial B_r$. We consider
    \begin{equation}
        \psi = 
        \begin{cases}
            \psi_0      &   \text{in} \ E'_i \cap R^0_r\\
            \mathrm{id} &   \text{in} \ (V \setminus B_r) \cup (\R^n \setminus R^1_r)
        \end{cases}
    \end{equation}
    Let us check that this map is still Lipschitz (even with a bad a Lipschitz constant). For $x \in E'_i \cap R^0_r$ and $y \in V \setminus B_r$, let $z \in V \cap \partial B_r$ be the orthogonal projection of $y$ onto $B_r$. Then, $\abs{y - z} \leq \abs{y - x}$ and by the triangular inequality, $\abs{x - z} \leq \abs{y - x} + \abs{y - z} \leq 2\abs{y - x}$. Denoting $L$ the Lipschitz constant of $\psi_0$ on $\overline{B_r}$, we have
    \begin{align}
        \abs{\psi(x) - \psi(y)} &\leq \abs{\psi(x) - \psi(z)} + \abs{\psi(z) - \psi(y)}\\
                                &\leq L\abs{\psi(x) - \psi(z)} + \abs{z - y}\\
                                &\leq L\abs{x - z} + \abs{z - y}\\
                                &\leq (1 + 2L)\abs{z - y}.
    \end{align}
    As $\psi_0(E'_i \cap R^0_r) \subset V \cap \partial B_r$, we also see that $\psi$ preserves $R^2_r$ where it is defined. We can extend $\psi$ as a Lipschitz map $\R^n \to \R^n$ such that $\psi(R^2_r) \subset R^2_r$.

    We introduce $f = \psi \circ \pi$. By construction, $f = \mathrm{id}$ in $\R^n \setminus R^2_r$ and $f(R^2_r) \subset R^2_r$. We are going to postcompose $f$ with a retraction onto the boundary to obtain a sliding deformation. If $0 \notin \Gamma$, then for $r > 0$ small enough, $B_r$ is disjoint from $\Gamma$ and there is nothing to do. We assume that $0 \in \Gamma$ and we define our retraction. There exists an open set $O \subset X$ containing $\Gamma$ and a $C$-Lipschitz map $p\colon O \to \Gamma$ such that $p = \mathrm{id}$ on $\Gamma$. We are going to see that if $r$ is small enough, then $\abs{p - \mathrm{id}} \leq Car$ in $R^2_r$. By (\ref{V_gamma}), we can assume that $r > 0$ is small enough so that
    \begin{equation}\label{V_gamma2}
        2B_r \cap V \subset \set{y \in X | \mathrm{d}(y,\Gamma) \leq ar}.
    \end{equation}
    Let $y \in R^2_r$ and let $y'$ be the orthogonal projection of $y$ onto $V$. Thus, $y' \in 2B_r \cap V$ and $\abs{y - y'} \leq ar$. By (\ref{V_gamma2}), $\mathrm{d}(y',\Gamma) \leq ar$ so $\mathrm{d}(y,\Gamma) \leq 2ar$ and there exists $z \in \Gamma$ such that $\abs{y - z} \leq 3ar$. Then,
    \begin{align}
        \abs{(p - \mathrm{id})(y)}  &   =\abs{(p - \mathrm{id})(y) - (p - \mathrm{id})(z)}\\
                                    &   \leq C \abs{y - z}\\
                                    &   \leq C ar.
    \end{align}
    We restrict $p$ to $\Gamma \cup R^2_r$, next extend this restriction as a $C$-Lipschitz map $p\colon \R^n \to \R^n$ such that $\abs{p - \mathrm{id}} \leq C ar$. We still have $p(R^2_r) \subset \Gamma$ and $p = \mathrm{id}$ on $\Gamma$. We extend $p$ by $p = \mathrm{id}$ in $\R^n \setminus R^3_r$. This extension is still $C$-Lipschitz because $\abs{p - \mathrm{id}} \leq Car$ and
    \begin{equation}
        \mathrm{d}(R^2_r, \R^n \setminus R^3_r) \geq C^{-1}ar.
    \end{equation}
    As $\abs{p - \mathrm{id}} \leq Car$, we can assume $a$ small enough so that $p(R^3_r) \subset 3B_r$.

    Now, we consider $g = p \circ f$. We have $g = \mathrm{id}$ in $\R^n \setminus R^3_r$ and $g(R^3_r) \subset 3B_r$. We have also $g(\Gamma) \subset \Gamma$ because $f(R^2_r) \subset R^2_r$, $p(R^2_r) \subset \Gamma$ and $f = \mathrm{id}$ outside $R^2_r$. Finally, we justify that the map $g$ is a global sliding deformation in $3B_r$. We introduce
    \begin{equation}
        G =
        \begin{cases}
            \mathrm{id}                         &   \text{in} \ (\set{0} \times \R^n) \cup (I \times (\R^n \setminus R^3_r))\\
            p \circ ((1-t) \mathrm{id} + tf)    &   \text{in} \ I \times \Gamma\\
            g (=p \circ f)                      &   \text{in} \ \set{1} \times \R^n.
        \end{cases}
    \end{equation}
    Using the fact that $R^2_r$ is convex, we see as previously that $G_t(\Gamma) \subset \Gamma$. The map $G$ is continuous as a pasting of continuous maps in closed domains of $X$. According to the Tietze extension theorem, it can be extend as a continuous map $G\colon I \times \R^n \to \R^n$ such that $G_t(B_r) \subset 3B_r$. Thus, $g$ is a sliding deformation in $3B_r$.

    In summary, $g$ and $\phi$ are sliding deformations in $3B_r$ and $g = \phi = \mathrm{id}$ outside $R^3_r \cup T_r$. By definition of $R^3_r$ and $T_r$, $B_r \subset R^3_r \cup T_r$. We apply the quasiminimality of $E_i$ with respect to $g \circ \phi$ in $3B_r$. We assume $h$ small enough (depending only on $n$) such that
    \begin{equation}
        h \HH^d(E_i \cap h 3B_r) \leq \tfrac{1}{2} \HH^d(E_i \cap B_r).
    \end{equation}
    As $g \phi = \mathrm{id}$ outside $R^3_r \cup T_r$, the quasiminimality in $3B_r$ gives
    \begin{equation}
        \II(E_i \cap B_r) \leq C\II(g\phi[E_i \cap (R^3_r \cup T_r)]) + \varepsilon_i.
    \end{equation}
    Since $g = p \circ f$ and $p$ is $C$-Lipschitz, we have
    \begin{equation}
        \HH^d(g\phi[E_i \cap (R^3_r \cup T_r)]) \leq C\HH^d(f\phi[E_i \cap (R^3_r \cup T_r)]).
    \end{equation}
    By the fact that $f = \mathrm{id}$ in $\R^n \setminus R^2_r$, that $\phi = \mathrm{id}$ outside $T_r$ and by (\ref{phi_info}), we estimate first
    \begin{align}
        \begin{split}
        &   \HH^d(f\phi[E_i \cap (R^3_r \cup T_r)])\\
        &   \leq \HH^d(f(\phi(E_i) \cap R^3_r)) + \HH^d(f(\phi[E_i \cap (R^3_r \cup T_r)] \setminus R^3_r))\\
        &   \leq \HH^d(f(\phi(E_i) \cap R^3_r)) + \HH^d(\phi[E_i \cap (R^3_r \cup T_r)] \setminus R^3_r)
        \end{split}\\
        &   \leq \HH^d(f(\phi(E_i) \cap R^3_r)) + \HH^d(\phi(E_i \cap T_r))\\
        &   \leq \HH^d(f(\phi(E_i) \cap R^3_r)) + C\HH^d(E_i \cap T_r).
    \end{align}
    We decompose the set $\phi(E_i) \cap R^3_r$ in three parts by taking its intersection with $\pi^{-1}(R^0_r)$, $S_r$ and $R^3_r \setminus (\pi^{-1}(R^0_r) \cup S_r)$. First, $f(\phi(E_i) \cap \pi^{-1}(R^0_r)) \subset V \cap \partial B_r$ so
    \begin{equation}
        \HH^d(f(\phi(E_i) \cap B_r)) = 0.
    \end{equation}
    Next, by (\ref{phi_info0}),
    \begin{equation}
        \HH^d(f(\phi(E_i) \cap S_r)) = 0.
    \end{equation}
    Finally, remember that by definition of $R^3_r$ and $S_r$ and $R^3_r \setminus S_r \subset \frac{a}{2}V_r$ and
    \begin{equation}
        \pi(\tfrac{a}{2}V_r) \subset R^0_r \cup V \cup (\R^n \setminus (1+a)B_r).
    \end{equation}
    Thus,
    \begin{equation}
        \pi(R^3_r \setminus (p^{-1}(R^0_r) \cup S_r)) \subset (V \setminus B_r) \cup (\R^n \setminus (1+a)B_r).
    \end{equation}
    As $\psi = \mathrm{id}$ on $(V \setminus B_r) \cup (\R^n \setminus (1+a)B_r)$, we deduce that $f$ is $C$-Lipschitz on $R^3_r \setminus (p^{-1}(B_r) \cup S_r)$. Hence,
    \begin{align}
        \II(f(\phi(E_i) \cap R^3_r \setminus p^{-1}(B_r) \cup S_r)) &   \leq C\HH^d(\phi(E_i) \cap R^3_r \setminus p^{-1}(B_r))\\
        \begin{split}
            &\leq C\HH^d(\phi(E_i \cap T_r))\\
            &\qquad + C\HH^d(E_i \cap R^3_r \setminus p^{-1}(B_r))
        \end{split}\\
        \begin{split}
            &   \leq C\HH^d(E_i \cap T_r)\\
            &\qquad + C\HH^d(E_i \cap R^3_r \setminus p^{-1}(B_r)).
        \end{split}
    \end{align}
    As $R^3_r \subset (1+2a)B_r$ and $(1-\tfrac{a}{2})B_r \subset p^{-1}(B_r)$, we have
    \begin{equation}
        R^3_r \setminus p^{-1}(B_r) \subset (1+2a)B_r \setminus (1-\tfrac{a}{2})B_r.
    \end{equation}
    We conclude that
    \begin{align}
        \II(E_i \cap B_r) \leq C\HH^d(E_i \cap T_r) + C\HH^d(E_i \cap (1+2a)B_r \setminus (1-\tfrac{a}{2})B_r).
    \end{align}
    We pass to the limit $i \to +\infty$ and obtain
    \begin{equation}
        \mu(B_r) \leq C \mu(\overline{T_r}) + \mu((1+2a)B_r \setminus (1-\tfrac{a}{2}B_r)).
    \end{equation}
    This is true for arbitrary small $r$ so we can multiply both sides by $r^{-d}$ and we pass to the limit $r \to 0$. As $\limsup_r r^{-d} \mu(\overline{T_r}) = 0$, we have
    \begin{equation}
        \HH^d(B_1 \cap V) \leq Ca.
    \end{equation}
    This is true for arbitrary small $a$ so we can make $a \to 0$ and obtain
    \begin{equation}
        \HH^d(B_r \cap V) = 0.
    \end{equation}
    Contradiction!
\end{proof}

\section{Application}
\subsection{Direct Method}
We use Theorem \ref{thm_limit} to derive a scheme of direct method. This is the same strategy as \cite{I1}, \cite{I2} but we can minimize the competitors on the boundary. Our working space is an open set $X$ of $\R^n$.
\begin{cor}[Direct Method]\label{cor_direct}
    Fix a Whitney subset $\Gamma$ of $X$. Fix an admissible energy $\II$ in $X$. Fix $\mathcal{C}$ a class of closed subsets of $X$ such that
    \begin{equation}
        m = \inf \set{\II(E) | E \in \mathcal{C}} < \infty
    \end{equation}
    and assume that for all $E \in \mathcal{C}$, for all sliding deformations $f$ of $E$ in $X$,
    \begin{equation}
        m \leq \II(f(E)).
    \end{equation}
    Let $(E_k)$ be a minimizing sequence for $\II$ in $\mathcal{C}$. Up to a subsequence, there exists a coral\footnote{A set $E \subset X$ is coral in $X$ if $E$ is the support of $\HH^d \mres E$ in $X$. Equivalently, $E$ is closed in $X$ and for all $x \in E$ and for all $r > 0$, $\HH^d(E \cap B(x,r)) > 0$.} minimal set $E$ with respect to $\II$ in $X$ such that
    \begin{equation}
        \II \mres E_k \rightharpoonup \II \mres E.
    \end{equation}
    where the arrow $\rightharpoonup$ denotes the weak convergence of Radon measures in $X$. In particular, $\II(E) \leq m$.
\end{cor}
\begin{rmk}
    The limit $E$ may not belong to $\mathcal{C}$.
\end{rmk}
\begin{proof}
    Let $(E_k)$ be a minimizing sequence. The sequences ($\varepsilon_k$) defined by $\varepsilon_k = E_k - m$ goes to $0$. The sequence $(\II \mres E_k)$ is a bounded sequence of Radon measure in $X$ so, up to a subsequence, there exists a Radon measure $\mu$ in $X$ such that
    \begin{equation}
        \II \mres E_k \rightharpoonup \mu.
    \end{equation}
    For all $k$, for all global sliding deformations $f$ in $X$,
    \begin{align}
        \II(E_k)    &   \leq m + \varepsilon_k\\
                    &   \leq \II(f(E_k)) + \varepsilon_k
    \end{align}
    whence for all open set $W \subset X$ such that $f = \mathrm{id}$ in $X \setminus W$,
    \begin{equation}
        \II(E_k \cap W) \leq \II(f(E_k \cap W)) + \varepsilon_k.
    \end{equation}
    According to Theorem \ref{thm_limit} and Remark \ref{rmk_limit} we have $\mu = \II \mres E$, where $E$ is the support of $\mu$. Moreover, for all sliding deformations $f$ in $X$,
    \begin{equation}
        \II(E \cap W_f) \leq \II(f(E \cap W_f)).
    \end{equation}
\end{proof}

\textbf{Acknowledgement:} I would like to thank Guy David for his warm and helpful discussions. I thank the anonymous reviewers who have improved this paper.

\begin{appendices}
    \section{Continuous and Lipschitz Extensions}\label{lipschitz_appendix}
    \subsection{Continuous Extensions}
    \begin{lem}[Tietze extension]\label{continuous_extension}
        Let $X$ be a metric space and $A$ be a closed subset of $X$. Any continuous function $f\colon A \to \R^n$ has a continuous extension $g\colon X \to \R^n$.
    \end{lem}
    \begin{rmk}
        Note that we can postcompose $g$ with the orthogonal projection onto the closed convex hull of $f(A)$. Thus, we obtain another continuous extension whose image is included in the convex hull of $f(A)$. For example, if $\abs{f} \leq M$, we can assume $\abs{g} \leq M$ as well.
    \end{rmk}

    \subsection{Lipschitz Extensions}
    The McShane--Whitney formula is a simple technique to build Lipschitz extensions.
    \begin{lem}[McShane--Whitney extension]\label{lipschitz_extension}
        Let $X$ be a metric space and $A \subset X$. Any Lipschitz function $f\colon A \to \R^n$ has a Lipschitz extension $g\colon X \to \R^n$.
    \end{lem}
    \begin{proof}
        We cover the case $n = 1$ because it suffices to extend each coordinate functions independently. Let $L$ be the Lipschitz constant of $f$. Then the McShane--Whitney extension of $f$ is given by the formula
        \begin{equation}
            g(x) = \inf_{y \in A} \set{f(y) + L \abs{y - x}}.
        \end{equation}
        One can check that $g$ is real-valued, coincides with $f$ in $A$ and is $L$-Lipschitz.
    \end{proof}
    \begin{rmk}
        Let $\norm{f}$ be the Lipschitz constant of $f$. The McShane--Whitney extension is $C\norm{f}$-Lipschitz, where $C$ is a positive constant that depends only on $n$. In the case $X = \R^m$, the Kirzbraun theorem gives an extension $g$ with has the same Lipschitz constant as $f$. We can postcompose $g$ with the orthogonal projection onto the closed convex hull of $f(A)$. We obtain another Lipschitz extension whose image is included in the closed convex hull of $f(A)$. This operation does not increase the Lipschitz constant. For example, given $\abs{f} \leq M$, we can assume $\abs{g} \leq M$ as well without increasing the Lipschitz constant of $g$.
    \end{rmk}

    We also want to approximate continuous functions by Lipschitz functions.
    \begin{lem}\label{lipschitz_approximation}
        Let $X$ be a metric space. Let $f\colon X \to \R^n$ be a bounded and uniformly continuous function. Then for all $\varepsilon > 0$, there exists a Lipschitz function $g\colon X \to \R^n$ such that $\abs{g - f} \leq \varepsilon$.
    \end{lem}
    \begin{proof}
        We cover the case $n=1$ because it suffices to approximate each coordinate functions independently. Let us fix $\varepsilon > 0$; there exists $\delta > 0$ such that for all $x, y \in X$ with $\abs{x - y} \leq \delta$, $\abs{f(x) - f(y)} \leq \varepsilon$. We define $M = \sup \abs{f}$ and
        \begin{equation}
            g(x) = \inf_{y \in X} \set{f(y) + 2M \delta^{-1} \abs{x - y}}.
        \end{equation}
        One can check that $g$ is real-valued, $g \leq f$ and $g$ is $2M \delta^{-1}$-Lipschitz. Next, we check that $f \leq g + \varepsilon$. For $x, y \in X$, either $\abs{x - y} \geq \delta$ and then
        \begin{align}
            f(x)    &   \leq f(y) + 2M\\
                    &   \leq f(y) + 2M \delta^{-1} \abs{x - y}
        \end{align}
        or $\abs{x - y} \leq \delta$ and then
        \begin{align}
            f(x)    &   \leq f(y) + \varepsilon\\
                    &   \leq f(y) + 2M \delta^{-1} \abs{x - y} + \varepsilon.
        \end{align}
        In both cases, $f(x) \leq f(y) + 2M\delta^{-1}\abs{x - y} + \varepsilon$ and since $y \in X$ is arbitrary, $f(x) \leq g(x) + \varepsilon$.
    \end{proof}

    \begin{cor}\label{lipschitz_approximation_extension}
        Let $X$ be a metric space. Let $f\colon X \to \R^n$ be a bounded uniformly continuous function which is Lipschitz on some subset $A \subset X$. Then, for all $\varepsilon > 0$, there exists a Lipschitz function $g\colon X \to \R^n$ such that $\abs{g - f} < \varepsilon$ in $X$ and $g = f$ in $A$.
    \end{cor}
    \begin{proof}
        According to Lemma \ref{lipschitz_approximation}, there exists a Lipschitz function $g\colon X \to \R^n$ such that $\abs{g - f} \leq \frac{\varepsilon}{2}$. The function $u = f - g$ is Lipschitz on $A$ and satisfies $\abs{u} \leq \frac{\varepsilon}{2}$, so Lemma \ref{lipschitz_extension} say that it admits a Lipschitz extension $v\colon X \to \R^n$ with $\abs{v} \leq \frac{\varepsilon}{2}$. We conclude that $g + v$ is a solution to our problem.
    \end{proof}

    \section{Toolbox for Sliding Deformations}\label{appendix_toolbox}
    \subsection{Local Retractions}
    We will often need to localize a retraction of the boundary in a given open set.
    \begin{lem}\label{retract_lemma}
        Let $\Gamma$ be a Lipschitz neighborhood retract in $X$. For all open sets $U \subset X$ and for all $\varepsilon > 0$, there exists a Lipschitz map $p\colon \R^n \to \R^n$ and an open subset $O \subset X$ such that $\Gamma \cap U \subset O \subset U$ and
        \begin{subequations}
            \begin{align}
                &   \abs{p - \mathrm{id}} \leq \varepsilon\\
                &   p(O) \subset \Gamma\\
                &   p = \mathrm{id} \ \text{in} \ \Gamma \cup (\R^n \setminus U).
            \end{align}
        \end{subequations}
        Moreover, we can build $p$ such that its Lipschitz constant depends only on $n$ and $\Gamma$ (but not $U$ and $\varepsilon$).
    \end{lem}
    \begin{proof}
        Let $r$ be a Lipschitz retraction from an open set $O_0$ containing $\Gamma$ onto $\Gamma$. Let $\varepsilon > 0$, we define the open set
        \begin{equation}
            O = \Set{x \in U \cap O_0 | \abs{r(x) - x} < \varepsilon \min\set{\mathrm{d}(x,U^c),1}}.
        \end{equation}
        In particular, $\Gamma \cap U \subset O \subset U$ and $r(O) \subset \Gamma$. Consider the partially defined map
        \begin{equation}
            p =
            \begin{cases}
                r           &   \text{in} \ O\\
                \mathrm{id} &   \text{in} \ X \setminus U.
            \end{cases}
        \end{equation}
        Notice that $\Gamma \subset O \cup (X \setminus U)$ so $p = \mathrm{id}$ on $\Gamma$. It is straightforward that $\abs{p - \mathrm{id}} \leq \varepsilon$ because $\abs{r - \mathrm{id}} \leq \varepsilon$ in $O$. Next, we estimate the Lipschitz constant of $p - \mathrm{id}$. Let $\norm{r}$ be the Lipschitz constant of $r$. For $x, y \in O$,
        \begin{align}
            \abs{(p-\mathrm{id})(x) - (p-\mathrm{id})(y)}   &   \leq \abs{r(x) - r(y)} + \abs{x - y}\\
                                                            &   \leq (1 + \norm{r}) \abs{x - y}.
        \end{align}
        For $x \in O$ and for $y \in X \setminus U$,
        \begin{align}
            \abs{(p-\mathrm{id})(x) - (p-\mathrm{id})(y)}   &   \leq \abs{r(x) - x}\\
                                                            &   \leq \varepsilon \mathrm{d}(x, X \setminus U)\\
                                                            &   \leq \varepsilon \abs{x - y}.
        \end{align}
        We assume $\varepsilon \leq 1$ so that $p-\mathrm{id}$ is $(1 + \norm{r})$-Lipschitz on its domain. Finally, we apply Lemma \ref{lipschitz_extension} in Appendix \ref{lipschitz_appendix} to $p-\mathrm{id}$. Thus, $p$ extends as a Lipschitz map $p\colon \R^n \to \R^n$ such that $(p-\mathrm{id})$ is $C(1 + \norm{r})$-Lipschitz (where $C$ depends only on $n$) and $\abs{p-\mathrm{id}} \leq \varepsilon$.
    \end{proof}

    \subsection{Stability of Sliding Deformations}
    Our working space in an open set $X$ of $\R^n$. The next lemma says that a slight modification of a sliding deformation is still a sliding deformation.
    \begin{lem}\label{sliding_perturbation}
        Let $\Gamma$ be a Lipschitz neighborhood retract in $X$. Let $E$ be a closed subset of $X$. Let $f$ be a sliding deformation of $E$ in an open subset $U \subset X$. Let $W$ be an open set such that $W \subset \subset E \cap U$. Then there exists $\delta > 0$ such that all Lipschitz maps $g\colon E \to \R^n$ satisfying
        \begin{subequations}
            \begin{align}
                &   \abs{g - f} \leq \delta,\\
                &   g(E \cap \Gamma) \subset \Gamma,\\
                &   g = \mathrm{id} \ \text{in} \ E \setminus W,
            \end{align}
        \end{subequations}
        are sliding deformations of $E$ in $U$.
    \end{lem}
    \begin{proof}
        Let $F$ be a sliding homotopy associated to $f$. We define
        \begin{equation}
            W_0 = W \cup \Set{x \in E | \exists t \in I,\ F_t(x) \ne x}
        \end{equation}
        and we underline that $W_0 \subset \subset E \cap U$. Thus, there exists an open set $U_0$ such that
        \begin{equation}
            F(I \times \overline{W_0}) \subset U_0 \subset \subset U.
        \end{equation}
        We fix $\delta_0 > 0$ such that $\mathrm{d}(U_0,X \setminus U) \geq \delta_0$. We apply Lemma \ref{retract_lemma} to obtain a Lipschitz function $p\colon X \to \R^n$ and an open set $O \subset X$ containing $\Gamma$ such that
        \begin{subequations}
            \begin{align}
                &   \abs{p - \mathrm{id}} \leq \delta_0\\
                &   p(O) \subset \Gamma\\
                &   p = \mathrm{id} \ \text{in} \ \Gamma.
            \end{align}
        \end{subequations}
        The map $p$ will not be used in step $1$ but the open set $O$ will be needed.

        \emph{Step 1.} Let $\delta > 0$ (to be precised momentarily) and let $g$ be as the statement. We introduce the homotopy $G_t = F_t + t(g - f)$. It is clear that $G_0 = \mathrm{id}$, $G_1 = g$ and for all $t \in I$, $G_t = \mathrm{id}$ in $E \setminus W_0$. The condition $\abs{g - f} \leq \delta$ also implies that for all $t \in I$, $\abs{G_t - F_t} \leq \delta$. We are going to see that if $\delta$ is sufficiently small, then for all $t \in I$,
        \begin{align}
            &   G_t(E \cap \Gamma) \subset O,\\
            &   G_t(E \cap U_0) \subset U_0.
        \end{align}
        Since $I \times \overline{W_0}$ is compact and $F(I \times \overline{W_0}) \subset U_0$, we can take $\delta > 0$ small enough so that for all $t \in I$, for all $x \in \overline{W_0}$,
        \begin{equation}
            \mathrm{d}(F_t(x), \R^n \setminus U_0) > \delta.
        \end{equation}
        Thus, for all $t \in I$, $x \in W_0$,
        \begin{equation}
            G_t(x) \in \overline{B}(F_t(x),\delta) \subset U_0.
        \end{equation}
        We deduce that for all $t \in I$, $G_t(E \cap U_0) \subset U_0$ as $G_t = \mathrm{id}$ in $E \cap U_0 \setminus W_0$. Similarly, $I \times (\Gamma \cap \overline{W_0})$ is compact and $F(I \times (\Gamma \cap \overline{W_0})) \subset O$ because $F(I \times (\Gamma \cap E)) \subset \Gamma \subset O$. We take $\delta > 0$ small enough so for all $t \in I$, for all $x \in \Gamma \cap \overline{W_0}$,
        \begin{equation}
            \mathrm{d}(F_t(x), \R^n \setminus O) > \delta
        \end{equation}
        and we deduce that for all $t \in I$, $G_t(E \cap \Gamma) \subset O$.

        \emph{Step 2.} We would like to retract $G_t(\Gamma)$ onto $\Gamma$ so we define
        \begin{equation}
            H_t =
            \begin{cases}
                \mathrm{id} &   \text{in} \ (0 \times E) \cup (I \times (E \setminus W_0))\\
                p \circ G_t &   \text{in} \ I \times (E \cap \Gamma)\\
                g           &   \text{in} \ 1 \times E.
            \end{cases}
        \end{equation}
        This map is continuous as a pasting of continuous maps in closed domains (relative to $I \times E$). As $G_t(E \cap \Gamma) \subset O$, we have $H_t(E \cap \Gamma) \subset \Gamma$. Since $\abs{p - \mathrm{id}} \leq \delta_0$, $H$ also satisfies the inequality $\abs{H_t - G_t} \leq \delta_0$. We apply the Tietze Extension Theorem (Lemma \ref{continuous_extension}, Appendix \ref{lipschitz_appendix}) in the working space $I \times E$ to $H - G$. Thus, we extend $H$ into a continuous function $H\colon E \to \R^n$ such that $\abs{H_t - G_t} \leq \delta_0$ in $E$. Combining $\abs{H_t - G_t} \leq \delta_0$ and $G_t(E \cap U_0) \subset U_0$, we deduce that $H_t(E \cap U_0) \subset U$ by definition of $\delta_0$. Moreover, $H_t = \mathrm{id}$ on $E \setminus U_0$ so we have in fact $H_t(U) \subset U$. We conclude that $H$ is a sliding homotopy.
    \end{proof}

    \subsection{Global Sliding Deformations}\label{global_deformations}
    Our working space is an open set $X$ of $\R^n$. We recall that a global sliding deformation in an open set $U \subset X$ is a sliding deformation of $X$ in $U$ (the set $E$ is replaced by $X$ in Definition \ref{defi_sliding}). Our goal is to show that global sliding deformations induce the same quasiminimal sets. First, we present a necessary and sufficient condition for a sliding deformations on $E$ to extend as a global deformation.
    \begin{lem}[Sliding Deformation Extension]\label{sliding_extension}
        Let $\Gamma$ be a Lipschitz neighborhood retract in $X$. Let $E$ be a closed subset of $X$. Let $f$ be a sliding deformation of $E$ in an open set $U \subset X$. Then $f$ extends as a global sliding deformation in $U$ if and only if there exists a constant $C \geq 1$ such that for all $x \in E$,
        \begin{equation}\label{extension_assumption}
            \mathrm{d}(f(x), \Gamma) \leq C \mathrm{d}(x, \Gamma).
        \end{equation}
    \end{lem}
    \begin{proof}
        Let us justify that the condition is necessary. Assume that there exists a global sliding deformation $g$ in $U$ which coincides with $f$ on $E$. For all $x \in E$, for all $y \in \Gamma$, we have $g(y) \in \Gamma$ so
        \begin{align}
            \mathrm{d}(f(x), \Gamma)    &   \leq \abs{f(x) - g(y)}\\
                                        &   \leq \abs{g(x) - g(y)}\\
                                        &   \leq \norm{g} \abs{x - y},
        \end{align}
        where $\norm{g}$ is the Lipschitz constant of $g$. Since $y \in \Gamma$ is arbitrary, it follows that $\mathrm{d}(f(x), \Gamma) \leq \norm{g} \mathrm{d}(x,\Gamma)$. From now on, we assume (\ref{extension_assumption}) and we build an extension of $f$.

        This paragraph is devoted to introducing a few objects and notation. Let $F$ be a sliding homotopy associated to $f$. Our extension of $F$ risks overstepping $U$ so we are going to work in a smaller open set $U_0$ which is relatively compact in $U$. Let $K \subset E \cap U$ be a compact set such that for all $t$, $F_t = \mathrm{id}$ in $E \setminus K$. As $F(I \times K)$ is a compact subset of $U$, there exists an open set $U_0 \subset \subset U$ such that $F(I \times K) \subset U_0$. In particular $K \subset U_0$ because $F_0 = \mathrm{id}$. Let $W$ be an open set such that $K \subset W \subset \subset U_0$. The point of such set $W$ is that there exists a constant $M > 0$ such that for all $x \in E \cap K$,
        \begin{equation}\label{extension_assumption2}
            \abs{f(x) - x} \leq M \mathrm{d}(x, X \setminus W).
        \end{equation}
        Of course, the inequality still holds for all $x \in E$ since $f = \mathrm{id}$ outside $K$. This inequality will allow to extend $f$ in a Lipschitz way by $f = \mathrm{id}$ in $X \setminus W$. Finally, we want a Lipschitz retraction onto $\Gamma$. We apply Lemma \ref{retract_lemma} to obtain a Lipschitz map $p\colon X \to \R^n$ and an open subset $O \subset X$ containing $\Gamma$ such that
        \begin{subequations}
            \begin{align}
                &   \abs{p - \mathrm{id}} \leq \varepsilon\\
                &   p(O) \subset \Gamma\\
                &   p = \mathrm{id} \ \text{in} \ \Gamma,
            \end{align}
        \end{subequations}
        where $\varepsilon$ is a small positive constant that we will specify later.

        \emph{Step 1.} The first part of the proof consists in building a continuous function $G\colon I \times X \to \R^n$ which is an extension of $F$ and such that
        \begin{subequations}\label{G}
            \begin{align}
                &   G_0 = \mathrm{id}\\
                &   G_1 \ \text{is Lipschitz}\\
                &   \forall t \in [0,1],\ G_t(\Gamma) \subset O\\
                &   \forall t \in [0,1],\ G_t(U_0) \subset U_0\\
                &   \forall t \in [0,1],\ G_t = \mathrm{id}\ \text{in} \ X \setminus W.
            \end{align}
        \end{subequations}
        The partially defined function
        \begin{equation}
            \begin{cases}
                F_t         &   \text{in} \ I \times E\\
                \mathrm{id} &   \text{in} \ (0 \times X) \cup (I \times (X \setminus W))
            \end{cases}
        \end{equation}
        is continuous because it is obtained by pasting continuous functions in closed domains. We apply the Tietze Theorem to obtain a continuous extension $G\colon I \times X \to \R^n$. In order to obtain the conditions (\ref{G}), we will re-parametrize $G_t$. By compactness, the inclusion $I \times K \subset G^{-1}(U_0)$ implies the existence of an open set $V \subset X$ such that $K \subset V$ and
        \begin{equation}
            I \times V \subset G^{-1}(U_0).
        \end{equation}
        We apply the same argument in $I \times \Gamma$ where the inclusion $I \times (\Gamma \cap K) \subset G^{-1}(\Gamma) \subset G^{-1}(O)$ implies the existence of a relative open set $V_\Gamma \subset \Gamma$ such that $\Gamma \cap K \subset V_\Gamma$ and
        \begin{equation}
            I \times V_\Gamma \subset G^{-1}(O).
        \end{equation}
        Let $\varphi\colon X \to [0,1]$ be a continuous function such that $\varphi = 1$ in $K$ and $\varphi = 0$ in $X \setminus V$ and $\Gamma \setminus V_\Gamma$. We define
        \begin{equation}
            G'_t(x) = G_{t \varphi(x)}(x).
        \end{equation}
        Hence $G'$ is a continuous function which satisfies
        \begin{subequations}\label{G2}
            \begin{align}
                &   G'_0 = \mathrm{id}\\
                &   \forall t \in [0,1],\ G'_t(V_\Gamma) \subset O\\
                &   \forall t \in [0,1],\ G'_t(V) \subset U_0\\
                &   \forall t \in [0,1],\ G'_t = \mathrm{id}\ \text{in}\ (X \setminus W) \cup (X \setminus V) \cup (\Gamma \setminus V_\Gamma).
            \end{align}
        \end{subequations}
        In addition, $G'$ coincides with $F$ on $I \times E$. Combining $G'_t(V) \subset U_0$ and $G'_t = \mathrm{id}$ in $X \setminus V$, one deduces that $G'_t(U_0) \subset U_0$. Similarly, $G'_t(\Gamma) \subset O$. Nex, we replace $G'_1$ with a Lipschitz approximation. As $W \subset \subset U_0$, $G(I \times \overline{W})$ is a compact subset of $U_0$ and there exists $\delta > 0$ such that for all $t \in I$, for all $x \in \overline{W}$,
        \begin{equation}\label{ineq_W1}
            \overline{B}(G_t(x), \delta) \subset U_0.
        \end{equation}
        The set $G(I \times (\Gamma \cap \overline{W}))$ is also a compact subset of $O$ so we can assume that for all $t \in I$, for all $x \in \Gamma \cap \overline{W}$,
        \begin{equation}\label{ineq_W2}
            \overline{B}(G_t(x), \delta) \subset O.
        \end{equation}
        We are going to replace $G'_1$ with a Lipschitz function $g\colon X \to \R^n$ such that $g = G'_1$ in $E \cup (X \setminus W)$ and $\abs{g - G'_1} \leq \delta$ in $X$. We start by checking that $G'_1$ is Lipschitz in $E \cup (X \setminus W)$. Indeed $G_1 = f$ in $E$, $G_1 = \mathrm{id}$ in $X \setminus W$ and for $x \in E$ and $y \in X \setminus W$, (\ref{extension_assumption2}) yields
        \begin{align}
            \abs{G_1(x) - G_1(y)}   &   = \abs{f(x) - y}\\
                                    &   \leq \abs{f(x) - x} + \abs{x - y}\\
                                    &   \leq M \mathrm{d}(x, X \setminus W) + \abs{x - y}\\
                                    &   \leq (M + 1) \abs{x - y}.
        \end{align}
        We can apply Lemma \ref{lipschitz_approximation_extension} to $G'_1 - \mathrm{id}$ as it is continuous with compact support. Hence we obtain a Lipschitz function $v\colon X \to \R^n$ such that $v = G'_1 - \mathrm{id}$ in $E \cup (X \setminus W)$ and $\abs{G'_1 - \mathrm{id} - v} < \delta$. Then we define $g = v + \mathrm{id}$ and we replace $G'$ with
        \begin{equation}
            G''_t = G'_t + t(g - G'_1).
        \end{equation}
        Combining (\ref{ineq_W1}), (\ref{ineq_W2}) and the facts that $\abs{g - G'_1} < \delta$ and $g = \mathrm{id}$ in $X \setminus W$, one can see that $G''_t(\Gamma) \subset O$ and $G''_t(U_0) \subset U_0$. We conclude that $G''$ solves step $1$. It will be denoted $G$ in the next step.

        \emph{Step 2.}
        We would like to retract $G_t(\Gamma)$ onto $\Gamma$ so we define
        \begin{equation}
            H_t =
            \begin{cases}
                \mathrm{id} &   \text{in} \ 0 \times X\\
                p \circ G_t &   \text{in} \ I \times \Gamma\\
                G_t         &   \text{in} \ I \times (E \cup (X \setminus W)).
            \end{cases}
        \end{equation}
        This function is continuous as a pasting of continuous functions in closed domains. As $G_t(\Gamma) \subset O$, we have $H_t(\Gamma) \subset \Gamma$. It satisfies the inequality $\abs{H_t - G_t} \leq \varepsilon$ because $\abs{p - \mathrm{id}} \leq \varepsilon$. Let us check that $H_1$ is Lipschitz on its domain. The Lipschitz constants of $p$, $p - \mathrm{id}$ and $g$ are denoted by $\norm{p}$, $\norm{p - \mathrm{id}}$ and $\norm{g}$ respectively. Note that for all $x \in \R^n$, for all $y \in \Gamma$,
        \begin{align}
            \abs{p(x) - x}  &   = \abs{(p - \mathrm{id})(x) - (p - \mathrm{id})(y)}\\
                            &   \leq \norm{p - \mathrm{id}} \abs{x - y}
        \end{align}
        whence for all $x \in X$,
        \begin{equation}
            \abs{x - p(x)} \leq \norm{p - \mathrm{id}} \mathrm{d}(x, \Gamma).
        \end{equation}
        Using the Lemma assumption (\ref{extension_assumption}), we deduce that for $x \in E$ and $y \in \Gamma$, we have
        \begin{align}
            \abs{H_1(x) - H_1(y)}   &   = \abs{f(x) - p g(y)}\\
                                    &   \leq \abs{f(x) - p f(x)} + \abs{p g(x) - p g(y)}\\
                                    &   \leq \norm{p - \mathrm{id}} \mathrm{d}(f(x), \Gamma) + \norm{p}\norm{g}\abs{x - y}\\
                                    &   \leq \norm{p - \mathrm{id}} C \mathrm{d}(x, \Gamma) + \norm{p}\norm{g}\abs{x - y}\\
                                    &   \leq \norm{p - \mathrm{id}} C \abs{x - y} + \norm{p} \norm{g} \abs{x - y}.
        \end{align}
        We also have for $x \in X \setminus W$ and $y \in \Gamma$,
        \begin{align}
            \abs{H_1(x) - H_1(y)}   &   = \abs{x - p g(y)}\\
                                    &   \leq \abs{x - p(x)} + \abs{p g(x) - p g(y)}\\
                                    &   \leq \norm{p - \mathrm{id}}\mathrm{d}(x, \Gamma) + \norm{p} \norm{g} \abs{x - y}\\
                                    &   \leq \norm{p - \mathrm{id}}\abs{x - y} + \norm{p} \norm{g} \abs{x - y}.
        \end{align}
        We apply the extension Lemma \ref{lipschitz_extension} to extend $H_1$ as a Lipschitz function $H_1\colon X \to \R^n$ such that $\abs{H_1 - g} \leq \varepsilon$. Then, we use the Tietze Extension Lemma (Lemma \ref{continuous_extension}, Appendix \ref{lipschitz_appendix}) to extend $H$ as a continuous function $H\colon I \times X \to \R^n$ such that $\abs{H_t - G_t} \leq \varepsilon$. We assume that $\varepsilon$ is small enough so that for all $x \in U_0$, $\overline{B}(x, \varepsilon) \subset U$. Thus, the conditions $G_t(U_0) \subset U_0$ and $\abs{H_t - G_t} \leq \varepsilon$ ensure that $H_t(U_0) \subset U$. Moreover $H_t = \mathrm{id}$ in $X \setminus U_0$, so we have in fact $H_t(U) \subset U$. We conclude that $H$ solves the lemma.
    \end{proof}

    Combining the previous lemmas, we prove that every sliding deformation can be replaced by an equivalent global sliding deformation.
    \begin{lem}[Sliding Deformation Alternative]\label{sliding_alternative}
        Let $\Gamma$ be a Lipschitz neighborhood retract in $X$. Let $E$ be a closed subset of $X$ which is $\HH^d$ locally finite in $X$. Let $f$ be a sliding deformation of $E$ in an open subset $U \subset X$. Then for all $\varepsilon > 0$, there exists a global sliding deformation $g$ in $U$ such that $\abs{g - f} \leq \varepsilon$, $E \cap W_g \subset \subset W_f$ and
        \begin{equation}
            \HH^d(g(W_f) \setminus f(W_f)) \leq \varepsilon
        \end{equation}
        where
        \begin{align}
            W_g &   = \Set{x \in X | g(x) \ne x},\\
            W_f &   = \Set{x \in E | f(x) \ne x}.
        \end{align}
    \end{lem}
    \begin{proof}
        Given Lemmas \ref{sliding_extension} and \ref{sliding_perturbation}, it suffices to build a Lipschitz function $g\colon E \to \R^n$ which satisfies the following conditions: $\abs{g - f} \leq \varepsilon$, $g(E \cap \Gamma) \subset \Gamma$, $W_g \subset \subset W_f$, there exists $C \geq 1$ such that for all $x \in E$,
        \begin{equation}
            \mathrm{d}(g(x), \Gamma) \leq C \mathrm{d}(x, \Gamma)
        \end{equation}
        and finally,
        \begin{equation}
            \HH^d(g(W_f) \setminus f(W_f)) \leq \varepsilon_0.
        \end{equation}

        We fix $\varepsilon_0 > 0$. The construction will brings into play an intermediate variable $\varepsilon > 0$. First, we want to build a Lipschitz function $p\colon X \to \R^n$ (whose Lipschitz constant depends only on $\Gamma$) such that $\abs{p - \mathrm{id}} \leq \varepsilon_0$, $p = \mathrm{id}$ on $\Gamma$ and such that there exists an open set $O$ with $\Gamma \subset O \subset X$ and $p(O) \subset \Gamma$. Moreover, we want that
        \begin{equation}\label{g-estim1}
            \HH^d(W_f \cap W_p) \leq \varepsilon_0
        \end{equation}
        and
        \begin{equation}\label{g-estim2}
            \HH^d(W_f \cap f^{-1}(W_p)) \leq \varepsilon_0,
        \end{equation}
        where $W_p = \set{x \in X | p(x) \ne x}$. Let us proceed to build $p$. Since $\overline{W_f}$ is a compact subset of $E$, we have $\HH^d(W_f) < \infty$ so we can find an open set $V$ such that $\Gamma \subset V \subset X$ and
        \begin{equation}
            \HH^d(W_f \cap V \setminus \Gamma) \leq \varepsilon_0
        \end{equation}
        and
        \begin{equation}
            \HH^d(W_f \cap f^{-1}(V \setminus \Gamma)) \leq \varepsilon_0.
        \end{equation}
        Then we apply Lemma \ref{retract_lemma} in the open set $V$: there exists a Lipschitz function $p\colon X \to \R^n$ (whose Lipschitz constant depends only on $\Gamma$) and an open set $O$ such that $\Gamma \cap V \subset O \subset V$ and
        \begin{subequations}
            \begin{align}
                &   \abs{p - \mathrm{id}} \leq \varepsilon_0\\
                &   p(O) \subset \Gamma\\
                &   p = \mathrm{id} \ \text{in} \ \Gamma \cup (X \setminus V).
            \end{align}
        \end{subequations}
        As $W_p \subset V \setminus \Gamma$, we deduce
        \begin{equation}
            \HH^d(W_f \cap W_p) \leq \varepsilon_0
        \end{equation}
        and
        \begin{equation}
            \HH^d(W_f \cap f^{-1}(W_p)) \leq \varepsilon_0.
        \end{equation}

        Next, we truncate $f$ in view of obtaining the property $W_g \subset \subset W_f$. We introduce the set
        \begin{equation}
            W_\varepsilon = \set{x \in E | \mathrm{d}(x, E \setminus W_f) \geq \varepsilon}
        \end{equation}
        and since $W_f$ is a relative an open subset of $E$, we can assume $\varepsilon$ small enough so that
        \begin{equation}\label{g-estim0}
            \HH^d(W_f \setminus W_{2\varepsilon}) \leq \varepsilon_0.
        \end{equation}
        Let $f'$ be partially defined by
        \begin{equation}
            f' =
            \begin{cases}
                f           &   \text{in} \ W_{2\varepsilon}\\
                \mathrm{id} &   \text{in} \ E \setminus W_\varepsilon.
            \end{cases}
        \end{equation}
        We are going to estimate $\abs{f - f'}$ and the Lipschitz constant of $f - f'$. The Lipschitz constant of $f - \mathrm{id}$ plays a special role in these estimates and is denoted by $L$. We start by proving that for $x \in E \setminus W_\varepsilon$,
        \begin{equation}
            \abs{f(x) - x} \leq L \varepsilon
        \end{equation}
        Indeed, for $x \in E \setminus W_\varepsilon$, there exists $y \in E \setminus W_f$ such that $\abs{x - y} \leq \varepsilon$ whence
        \begin{align}
            \abs{f(x) - x}  &   = \abs{(f - \mathrm{id})(x) - (f - \mathrm{id})(y)}\\
                            &   \leq L \abs{x - y}\\
                            &   \leq L \varepsilon.
        \end{align}
        Now, it is straightforward that $\abs{f' - f} \leq L \varepsilon$ on the domain of $f'$. We are going to see that $f' - f$ is $L$-Lipschitz on the domain of $f'$. It is clear that $f - f'$ is $L$-Lipschitz on $W_{2\varepsilon}$ and $E \setminus W_\varepsilon$, respectively. For $x \in W_{2\varepsilon}$ and $y \in X \setminus W_\varepsilon$, we have $\abs{x - y} \geq \varepsilon$ so
        \begin{align}
            \abs{(f' - f)(x) - (f' - f)(y)} &   = \abs{f(y) - y}\\
                                            &   \leq L \varepsilon\\
                                            &   \leq L \abs{x - y}.
        \end{align}
        We apply Lemma \ref{lipschitz_extension} to $f' - f$ so as to extend $f'$ as a Lipschitz map $f'\colon E \to \R^n$ such that $\abs{f' - f} \leq L \varepsilon$ with a Lipschitz depending only on $n$ and $f$. Before moving to the next paragraph, we require $f'(E) \subset X$. As $f(E) \subset X$ and $\overline{W_f}$ is a compact subset of $E$, we can assume $\varepsilon$ small enough so that for all $x \in \overline{W_f}$,
        \begin{equation}
            \mathrm{d}(f(x), \R^n \setminus X) > L \varepsilon.
        \end{equation}
        This implies $f'(E) \subset X$ as $f' = f$ in $E \setminus W_f$ and $\abs{f' - f} \leq L \varepsilon$.

        We finally define the map $g$ on $E$ by
        \begin{equation}
            g = p \circ f' + \mathrm{id} - p.
        \end{equation}
        We have $\abs{p - \mathrm{id}} \leq \varepsilon_0$ and $\abs{f' - f} \leq L\varepsilon$ so $\abs{g - f} \leq 2\varepsilon_0 + L \varepsilon$. We assume $\varepsilon$ small enough so that $\abs{g - f} \leq 3\varepsilon_0$. Observe that $g = \mathrm{id}$ in $E \setminus W_\varepsilon$ so $W_g \subset \subset W_f$. Next, we prove that there exists a constant $C \geq 1$ such that for all $x \in E$,
        \begin{equation}\label{goal_g}
            \mathrm{d}(g(x), \Gamma) \leq C \mathrm{d}(x, \Gamma).
        \end{equation}
        This inequality is clearly true for $x \in E \setminus W_f$ so we focus on $W_f$. As $\Gamma$ is relatively closed in $X$ and $\overline{W_f}$ is compact subset of $X$, $\Gamma \cap \overline{W_f}$ is compact. Its image $f(\Gamma \cap \overline{W_f})$ is a compact subset of $\Gamma \subset O$ so there exists $\delta > 0$ such that for all $y \in f(\Gamma \cap \overline{W_f})$,
        \begin{equation}
            \mathrm{d}(y, X \setminus O) \geq \delta.
        \end{equation}
        We introduce
        \begin{equation}
            O_\delta = \set{y \in X | \mathrm{d}(y, X \setminus O) \geq \delta}
        \end{equation}
        so the set $f^{-1}(O_\delta)$ is a relative open set of $E$ containing $\Gamma \cap \overline{W_f}$. By compactness, we can assume $\varepsilon$ small enough so that
        \begin{equation}
            \set{x \in \overline{W_f} | \mathrm{d}(x,\Gamma) \leq \varepsilon} \subset f^{-1}(O_\delta).
        \end{equation}
        Then for $x \in \overline{W_f}$ such that $\mathrm{d}(x,\Gamma) \leq \varepsilon$, we have $f(x) \in O_\delta$ and thus $f'(x) \in O$ assuming $\varepsilon$ small enough so that $\abs{f' - f} \leq \delta$. We are also going to need the fact that for $x \in E$,
        \begin{equation}
            \abs{p(x) - x} \leq \norm{p - \mathrm{id}}\mathrm{d}(x, \Gamma),
        \end{equation}
        where $\norm{p - \mathrm{id}}$ is the Lipschitz constant of $p - \mathrm{id}$. Indeed for all $x \in E$ and all $y \in \Gamma$,
        \begin{align}
            \abs{p(x) - x}  &   = \abs{(p - \mathrm{id})(x) - (p - \mathrm{id})(y)}\\
                            &   \leq \norm{p - \mathrm{id}} \abs{x - y}
        \end{align}
        and since $y$ is arbitrary in $\Gamma$, $\abs{p(x) - x} \leq \norm{p - \mathrm{id}}\mathrm{d}(x,\Gamma)$. We are ready to prove (\ref{goal_g}). For $x \in W_f$, we have either $\mathrm{d}(x, \Gamma) \leq \varepsilon$, either $\mathrm{d}(x,\Gamma) \geq \varepsilon$. In the first case, $f'(x) \in O$ so $p \circ f(x) \in \Gamma$ and then
        \begin{align}
            \mathrm{d}(g(x),\Gamma) &   \leq \abs{g(x) - p \circ f(x)}\\
                                    &   \leq \abs{p(x) - x}\\
                                    &   \leq \norm{p - \mathrm{id}}\mathrm{d}(x,\Gamma).
        \end{align}
        In the second case,
        \begin{align}
            \mathrm{d}(g(x),\Gamma) &   \leq \sup\set{\mathrm{d}(g(u), \Gamma) | u \in \overline{W_f}}\\
                                    &   \leq \sup\set{\mathrm{d}(g(u), \Gamma) | u \in \overline{W_f}} \varepsilon^{-1} \mathrm{d}(x,\Gamma).
        \end{align}
        In both cases, we have $\mathrm{d}(g(x),\Gamma) \leq C \mathrm{d}(x,\Gamma)$, where $C \geq 1$ is a constant that does not depends on $x$. To finish the proof, we show that
        \begin{equation}
            \HH^d(g(W_f) \setminus f(W_f)) \leq 3\norm{g}\varepsilon_0
        \end{equation}
        where $\norm{g}$ is the Lipschitz constant of $g$ (it depends only on $n$, $f$ and $\Gamma$). Observe that $g = f$ on $W_{2\varepsilon} \setminus (W_p \cup f^{-1}(W_p))$. Moreover by (\ref{g-estim0}),
        \begin{equation}
            \HH^d(W_f \setminus W_{2\varepsilon)} \leq \varepsilon_0
        \end{equation}
        and by (\ref{g-estim1}), (\ref{g-estim2}),
        \begin{equation}
            \HH^d(W_f \cap (W_p \cup f^{-1}(W_p))) \leq 2\varepsilon_0.
        \end{equation}
        The result follows.
    \end{proof}

    \section{Grassmannian Space}\label{appendix_grass}
    Let $E$ be an Euclidean vector space of dimension $n$ and $0 \leq d \leq n$ be an integer. The Grassmannian $G(d,E)$ is the set of all $d$-linear planes of $E$. Given a linear map $u\colon E \to E$, the symbol $\norm{u}$ denotes the operator norm
    \begin{equation}
        \norm{u} = \sup \set{\abs{u(x)} | x \in E,\ \abs{x} \leq 1}.
    \end{equation}
    Other authors use the Hilbert--Schmidt norm $\norm{u}_{HS} = \sqrt{\mathrm{trace}(u^* u)}$. Each linear plane $V \in G(d,E)$ is uniquely identified by the orthogonal projection $p_V$ onto $V$. This correspondance induces a distance on $G(d,E)$:
    \begin{equation}
        \mathrm{d}(V,W) = \norm{p_V - p_W}.
    \end{equation}
    The action of the orthogonal group $O(E)$ on $G(d,E)$, $(g,V) \mapsto gV$, is distance-preserving because $p_{gV} = g \circ p_V \circ g^{-1}$. The application $G(d,E) \to G(n-d,E)$, $V \to V^\perp$ is an isometry. In the case $E = \R^n$, the Grassmannian $G(d,E)$ is simply denoted by $G(d,n)$. 

    \subsection{Metric Structure}
    Given a subspace $V$ of $\R^n$ and a linear map $u \colon \R^n \to \R^n$, we define
    \begin{equation}
        \norm{u}_V = \sup \set{\abs{u(x)} | x \in V,\ \abs{x} \leq 1}.
    \end{equation}
    We are going to see that it suffices to compute the norm $\norm{p_V - p_W}$ on $V$, $V^\perp$ or $V$, $W$. This is helpful because the expression $p_V - p_W$ can be simplified on these subspaces.
    \begin{lem}\label{grassEquivalence}
        For $V, W \in G(d,n)$,
        \begin{align}
            \mathrm{d}(V,W) &   = \max\set{\norm{p_V - p_W}_V, \norm{p_V - p_W}_{V^\perp}}\\
                            &   = \max\set{\norm{p_V - p_W}_V, \norm{p_V - p_W}_W}.
        \end{align}
    \end{lem}
    \begin{proof}
        For $x \in \R^n$, we denote $x_V = p_V(x)$ and $x_{V^\perp} = p_{V^\perp}(x)$. Observe that $(p_V - p_W)(x_V) \in W^\perp$ and $(p_V - p_W)(x_{V^\perp}) \in W$ so
        \begin{equation}
            \abs{p_V(x) - p_W(x)}^2 = \abs{p_V(x_V) - p_W(x_V)}^2 + \abs{p_V(x_{V^\perp}) - p_W(x_{V^\perp})}^2.
        \end{equation}
        We deduce that
        \begin{equation}
            \norm{p_V - p_W} \leq \max\set{\norm{p_V - p_W}_V, \norm{p_V - p_W}_{V^\perp}}.
        \end{equation}
        Since the maps $(p_V - p_W)\colon V^\perp \to W$ and $(p_V - p_W)\colon W \to V^\perp$ are adjoints of one another, we have $\norm{p_V - p_W}_{V^\perp} = \norm{p_V - p_W}_W$. Finally, it is clear that
        \begin{equation}
            \max(\norm{p_V - p_W}_V, \norm{p_V - p_W}_W) \leq \norm{p_V - p_W}.
        \end{equation}
    \end{proof}

    The next lemma describes the local structure of $G(d,n)$.
    \begin{lem}
        \begin{enumerate}[label=(\roman*)]
            \item For all $V, W \in G(d,n)$, $\mathrm{d}(V,W) \leq 1$ and there is equality if and only if $V \cap W^\perp \ne 0$ or $V^\perp \cap W \ne 0$.
            \item Let $V, W \in G(d,n)$ be such that $\mathrm{d}(V,W) < 1$, then for all $x \in W$,
                \begin{equation}\label{grassGraph}
                    \abs{x} \leq \frac{1}{\sqrt{1 - \mathrm{d}(V,W)^2}} \abs{p_V(x)}
                \end{equation}
                Thus $p_V$ induces a linear isomorphism $p_* \colon W \to V$. In particular, $W = \Set{x + \varphi(x) | x \in V}$, where $\varphi\colon V \to V^\perp$, $x \to p_*^{-1}(x) - x$.
            \item Let $V$ and $W \in G(d,n)$ for which there exists a linear application $\varphi\colon V \to V^\perp$ such that $W = \Set{x + \varphi(x) | x \in V}$, then
                \begin{equation}\label{grassNorm}
                    \mathrm{d}(V,W) = \frac{\norm{\varphi}}{\sqrt{1 + \norm{\varphi}^2}}.
                \end{equation}
        \end{enumerate}
    \end{lem}
    \begin{proof}
        \textbf{1)} Let $V, W \in G(d,n)$. For $x \in \R^n$, the orthogonal projection $p_V(x)$ satisfies the equation $p_V(x) \cdot (x - p_V(x)) = 0$ because $x - p_V(x) \in V^\perp$. Observe that this is equivalent to $\abs{p_V(x) - \frac{x}{2}} = \frac{\abs{x}}{2}$ (we have a similar property for $p_W(x)$). According to the triangular inequality,
        \begin{align}
            \abs{p_V(x) - p_W(x)}   &   \leq \abs{p_V(x) - \frac{x}{2}} + \abs{p_W(x) - \frac{x}{2}}\\
                                    &   = \frac{\abs{x}}{2} + \frac{\abs{x}}{2}\\
                                    &   = \abs{x}.
        \end{align}
        so $\norm{p_V - p_W} \leq 1$. Assume that $\abs{p_V(x) - p_W(x)} = \abs{x}$ for some $x \in \R^n \setminus 0$. The previous inequalities become equalities so
        \begin{equation}
            \abs{p_V(x) - p_W(x)} = \abs{p_V(x) - \frac{x}{2}} + \abs{p_W(x) - \frac{x}{2}}.
        \end{equation}
        We deduce that $p_V(x)$ and $p_W(x)$ are antipodals on the sphere $S(\frac{x}{2}, \frac{\abs{x}}{2})$, whence
        \begin{equation}
            p_V(x) - \frac{x}{2} = -\left(p_W(x) - \frac{x}{2}\right)
        \end{equation}
        or equivalently
        \begin{equation}
            p_V(x) + p_W(x) = x.
        \end{equation}
        As $x \ne 0$, we have either $p_V(x) \ne 0$ or $p_{V^\perp}(x) \ne 0$. In the first case,
        \begin{equation}
            p_V(x) = p_{W^\perp}(x) \in V \cap W^\perp \setminus \set{0}
        \end{equation}
        and in the second case
        \begin{equation}
            p_{V^\perp}(x) = p_W(x) \in V^\perp \cap W \setminus \set{0}.
        \end{equation}
        The converse is straightforward.

        \textbf{2)} For all $x \in W$,
        \begin{align}
            \abs{x}^2   &   = \abs{p_V(x)}^2 + \abs{x - p_V(x)}^2\\
                        &   = \abs{p_V(x)}^2 + \abs{p_W(x) - p_V(x)}^2\\
                        &   \leq \abs{p_V(x)}^2 + \mathrm{d}(V,W)^2 \abs{x}^2
        \end{align}
        so
        \begin{equation}
            \abs{x} \leq \frac{1}{\sqrt{1 - \mathrm{d}(V,W)^2}} \abs{p_V(x)}.
        \end{equation}

        \textbf{3)} We show first that
        \begin{equation}
            W^\perp = \Set{x - \varphi^*(x) | x \in V^\perp},
        \end{equation}
        where $\varphi^*\colon V^\perp \to V$ is the adjoint of $\varphi$. For $x \in V$ and $y \in V^\perp$, we have
        \begin{align}
            (x + \varphi(x)) \cdot (y - \varphi^*(y))   &   = \varphi(x) \cdot y - x \cdot \varphi^*(y)\\
                                                        &   =0.
        \end{align}
        As $W = \set{x + \varphi(x) | x \in V}$, we have shown that
        \begin{equation}
            \set{x - \varphi^*(x) | x \in V^\perp} \subset W^\perp.
        \end{equation}
        The map $x \mapsto x - \varphi^*(x)$ is injective on $V^\perp$ because $x$ and $\varphi^*(x)$ are orthogonals. We conclude that this inclusion is an equality by a dimension argument.

        Next we prove that $\mathrm{d}(V,W) \leq \frac{\norm{\varphi}}{\sqrt{1 + \norm{\varphi}^2}}$. According to Lemma \ref{grassEquivalence},
        \begin{equation}
            \mathrm{d}(V,W) = \max\set{\norm{\mathrm{id} - p_W}_V, \norm{\mathrm{id} - p_{W^\perp}}_{V^\perp}}.
        \end{equation}
        If we show that
        \begin{equation}\label{grassDualityStep}
            \norm{\mathrm{id} - p_W}_V \leq \frac{\norm{\varphi}}{\sqrt{1 + \norm{\varphi}^2}},
        \end{equation}
        the same proof will yield
        \begin{equation}
            \norm{\mathrm{id} - p_{W^\perp}}_{V^\perp} \leq \frac{\norm{-\varphi^*}}{\sqrt{1 + \norm{-\varphi^*}^2}} = \frac{\norm{\varphi}}{\sqrt{1 + \norm{\varphi}^2}}.
        \end{equation}
        Therefore, we only prove (\ref{grassDualityStep}). Let us fix $x \in V \setminus 0$. As $\abs{x - p_W(x)} = \mathrm{d}(x,W)$, we see that for all $t \in \R$, $\abs{x - p_W(x)} \leq \abs{x - t(x + \varphi(x))}$. The right-hand side attains its minimum for
        \begin{equation}
            t = \frac{x \cdot (x + \varphi(x))}{\abs{x + \varphi(x)}^2} = \frac{\abs{x}^2}{\abs{x + \varphi(x)}^2}
        \end{equation}
        and one can compute
        \begin{align}
            \abs{x - t(x + \varphi(x)}^2    &   = (1 - t)^2 \abs{x}^2 + t^2 \abs{\varphi(x)}^2\\
                                            &   =\frac{\abs{x}^2 \abs{\varphi(x)}^2}{\abs{x}^2 + \abs{\varphi(x)}^2}\\
                                            &   \leq \frac{\norm{\varphi}^2}{1 + \norm{\varphi}^2} \abs{x}^2.
        \end{align}
        We conclude that $\mathrm{d}(V,W) \leq \frac{\norm{\varphi}}{\sqrt{1 + \norm{\varphi}^2}}$. It is left to prove the reverse inequality. For all $x \in V$,
        \begin{align}
            \abs{\varphi(x)}    &   = \abs{(x + \varphi(x)) - x}\\
                                &   = \abs{p_W(x + \varphi(x)) - p_V(x + \varphi(x))}\\
                                &   \leq \mathrm{d}(V,W) \abs{x + \varphi(x)}.
        \end{align}
        By $(\ref{grassGraph})$, we have
        \begin{align}
            \abs{x + \varphi(x)}    &   \leq \frac{1}{\sqrt{1 - \mathrm{d}(V,W)^2}} \abs{p_V(x + \varphi(x))}\\
                                    &   \leq \frac{1}{\sqrt{1 - \mathrm{d}(V,W)^2}} \abs{x}.
        \end{align}
        We deduce that $\norm{\varphi} \leq \frac{\mathrm{d}(V,W)}{\sqrt{1 - \mathrm{d}(V,W)^2}}$ or equivalently,
        \begin{equation}
            \frac{\norm{\varphi}}{\sqrt{1 + \norm{\varphi}^2}} \leq \mathrm{d}(V,W).
        \end{equation}
    \end{proof}

    Finally, we estimate the Lipschitz constant of a linear isomorphism acting on $G(d,n)$.
    \begin{lem}\label{lem_grass_iso}
        Let $u\colon \R^n \to \R^n$ be a linear isomorphism. Then for all $V, W \in G(d,n)$,
        \begin{equation}
            \mathrm{d}(u(V),u(W)) \leq \norm*{u} \norm*{u^{-1}} \mathrm{d}(V,W).
        \end{equation}
        If $\norm{u - \mathrm{id}} < 1$, then for all $V \in G(d,n)$,
        \begin{equation}
            \mathrm{d}(u(V),V) \leq \frac{\norm*{u-\mathrm{id}}}{1 - \norm*{u - \mathrm{id}}}.
        \end{equation}
    \end{lem}
    \begin{proof}
        According to Lemma \ref{grassEquivalence},
        \begin{equation}
            \mathrm{d}(u(V),u(W)) = \max\set{\norm{\mathrm{id} - p_{u(W)}}_{u(V)}, \norm{\mathrm{id} - p_{u(V)}}_{u(W)}}.
        \end{equation}
        By symmetry, we only need to show that
        \begin{equation}
            \norm{\mathrm{id} - p_{u(W)}}_{u(V)} \leq \norm*{u}\norm*{u^{-1}} \mathrm{d}(V,W)
        \end{equation}
        and this amounts to say that for all $y \in u(V)$,
        \begin{equation}
            \mathrm{d}(y,u(W)) \leq \norm*{u}\norm*{u^{-1}} \mathrm{d}(V,W)\abs{y}.
        \end{equation}
        Indeed, there exists $x \in V$ such that $y = u(x)$ so
        \begin{align}
            \mathrm{d}(y, u(W))                 &\leq \abs{u(x) - u(p_W(x))}\\
                                                &\leq \norm*{u} \abs{x - p_W(x)}\\
                                                &\leq \norm*{u} \mathrm{d}(V,W) \abs{x}\\
                                                &\leq \norm*{u}\norm*{u^{-1}} \mathrm{d}(V,W) \abs{y}.
        \end{align}
        This proves the first assertion. We are going to prove the second assertion similarly. For $x \in V$,
        \begin{align}
            \mathrm{d}(x,u(V))  &\leq \abs{u(x) - x}\\
                                &\leq \norm*{u - \mathrm{id}} \abs{x}
        \end{align}
        and for $y \in u(V)$, there exists $x \in V$ such that $y = u(x)$ so
        \begin{align}
            \mathrm{d}(y,V) &\leq \abs{u(x) - x}\\
                            &\leq \norm*{u - \mathrm{id}} \abs{x}\\
                            &\leq \norm*{u - \mathrm{id}} \norm*{u^{-1}} \abs{y}.
        \end{align}
        Note that if $\norm*{u - \mathrm{id}} < 1$, then
        \begin{align}
            \norm*{u^{-1}}  &\leq \sum_{k=0}^\infty \norm*{u - \mathrm{id}}^k\\
                            &\leq \frac{1}{\norm*{u - \mathrm{id}}^n}.
        \end{align}
    \end{proof}

    \subsection{Haar Measure}
    Let $E$ be an Euclidean vector space of dimension $n$ and $0 \leq d \leq n$ be an integer. The Haar measure $\gamma_{d,E}$ is the unique Radon measure on $G(d,E)$ whose total mass is $1$ and which is invariant under the action of $O(E)$ (see \cite[Chapter 3]{Mattila} for existence and unicity). In the case $E = \R^n$, the Haar measure is simply denoted by $\gamma_{d,n}$.

    It is usually not defined that way but $\gamma_{d,n}$ coincides with the Hausdorff measure of dimension $m=d(n-d)$, up to a multiplicative constant. Indeed, $G(d,n)$ is non-empty compact manifold of dimension $m$ so $\sigma = \HH^m(G(d,n))$ is finite and positive. Then $\sigma^{-1} \HH^m$ is a Radon measure on $G(d,n)$ whose total mass is $1$ and which is invariant under the isometries of $G(d,n)$. In particular, it is invariant under the action of $O(n)$. One can justify similarly that for all $A \subset G(1,n+1)$
    \begin{equation}\label{haar_sphere}
        \gamma_{1,n+1}(A) = \sigma_n^{-1} \HH^n(\set{x \in \mathbf{S}^n | \exists L \in A,\ x \in L}),
    \end{equation}
    where $\sigma_n = \HH^n(\mathbf{S}^n)$.

    \begin{lem}[Disintegration Formula]\label{grassSum}
        Let $p,q,n$ be non-negative integers with $p + q \leq n$. For all Borel set $A \subset G(n,p + q)$,
        \begin{equation}
            \gamma_{p+q,n}(A) = \int_{G(p,n)} \gamma_{q,V^\perp}(\set{W | V + W \in A}) \, \mathrm{d}\gamma_{p,n}(V).
        \end{equation}
        We omitted to write $W \in G(q,V^\perp)$ to ease the notation.
    \end{lem}
    \begin{proof}
        We introduce some notation so as to interpret the right-hand side as a pushforward measure. Let us define the space
        \begin{equation}
            X = \Set{(V,W) \in G(p,n) \times G(q,n) | V \perp W}.
        \end{equation}
        This space is closed in $G(p,n) \times G(q,n)$ because it can be written
        \begin{equation}
            \set{(V,W) \in G(p,n) \times G(q,n) | p_V p_W = 0}.
        \end{equation}
        In particular, $X$ is a compact space. We equip $X$ with the Radon measure $\gamma_{p \perp q,n}$ defined for all Borel set $A \subset X$ by
        \begin{equation}
            \gamma_{p \perp q,n}(A) = \int_{G(p,n)} \gamma_{V^\perp,q}(\set{W | (V,W) \in A}) \, \mathrm{d}\gamma_{p,n}(V).
        \end{equation}
        Finally we define the map $f\colon X \to G(p+q,n)$, $(V,W) \to V + W$. The map $f$ is Lipschitz because $V \perp W$ implies $p_{V + W} = p_V + p_W$.

        Now, the lemma reduces to showing that
        \begin{equation}
            \gamma_{p+q,n} = f_{\#} \gamma_{p \perp q,n}.
        \end{equation}
        According to \cite[Theorem 1.18]{Mattila}, $f_{\#} \gamma_{n,p \perp q}$ is a Radon measure on $G(p+q,n)$. It is easy to see that it is invariant under the action of $O(n)$ and that its total mass is $1$. By uniqueness of uniformly distributed measures on $G(p+q,n)$, it coincides with $\gamma_{p+q,n}$.
    \end{proof}

    \begin{lem}\label{grassParam}
        Let $H$ be an affine hyperplane in $\R^{n+1}$ which does not pass through $0$, let $r_0 = \mathrm{d}(0,H)$. Then for all bounded subset $A \subset H$,
        \begin{equation}
            \HH^n(A) \leq \frac{\sigma_n}{2} \left(\frac{r^2}{r_0}\right)^n \gamma_{1,n+1}(\set{L | \text{$L$ meets $A$}}),
        \end{equation}
        where $r = \sup\limits_{x \in A}\abs{x}$ and $\sigma_n = \HH^n(\mathbf{S}^n)$. Moreover for all $L \in G(1,n+1)$ such that $L$ meets $A$, we have
        \begin{equation}
            \mathrm{d}(L,L_0) \leq \sqrt{1 - \left(\frac{r}{r_0}\right)^2},
        \end{equation}
        where $L_0$ is the vector line orthogonal to $H$. 
    \end{lem}
    \begin{proof}
        Let $x_0$ be the orthogonal projection of $0$ onto $H$ (in particular $L_0$ is the line generated by $x_0$). The application
        \begin{equation}
            f\colon x \mapsto \frac{x}{\abs{x}}
        \end{equation}
        is bijective from $H$ onto $\mathbf{S}^n_+ =\set{y \in \mathbf{S}^n | y \cdot x_0 > 0}$. We want to compute the local Lipschitz constant of $f^{-1}$. One idea is to compute the differential of $f^{-1}$ but we propose a geometric approach. The idea is to show that for all $x, y \in \R^n \setminus 0$ such that $x \ne y$,
        \begin{equation}\label{O_distance}
            \mathrm{d}(0,(xy)) = \frac{1}{2} \frac{\abs{x}\abs{y}}{\abs{x - y}} \times \abs{\frac{x}{\abs{x}} - \frac{y}{\abs{y}}} \times \abs{\frac{x}{\abs{x}} + \frac{y}{\abs{y}}}
        \end{equation}
        where $(xy)$ is the line passing through $x$ and $y$. Let $z$ be the orthogonal projection of $0$ onto the line $(xy)$. We write $x = z + u$ and $y = z + v$ where $u, v$ are orthogonal to $z$ and collinear to see that 
        \begin{equation}
            \abs{x}^2 \abs{y}^2 - (x \cdot y)^2 = \abs{z}^2 \abs{x - y}^2.
        \end{equation}
        Then
        \begin{align}
            \abs{x}^2 \abs{y}^2 - (x \cdot y)^2 &= \abs{x}^2\abs{y}^2 \left(1 - \left(\frac{x \cdot y}{\abs{x}\abs{y}}\right)^2\right)\\
                                                &= \abs{x}^2 \abs{y}^2 \left(\frac{1}{2} \abs{\frac{x}{\abs{x}} - \frac{y}{\abs{y}}} \times \abs{\frac{x}{\abs{x}} + \frac{y}{\abs{y}}}\right)^2.
        \end{align}
        This proves (\ref{O_distance}). It is clear that for all $x, y \in H$, we have $\mathrm{d}(0,(xy)) \geq r_0$ so (\ref{O_distance}) gives
        \begin{equation}
            \abs{x - y} \leq \frac{\abs{x}\abs{y}}{r_0} \times \abs{\frac{x}{\abs{x}} - \frac{y}{\abs{y}}}
        \end{equation}
        We deduce that the Lipschitz constant of $f^{-1}$ on $f(A)$ is $\leq \frac{r^2}{r_0}$. By the action of Lipschitz functions on Hausdorff measures,
        \begin{align}
            \HH^n(A)    &\leq \left(\frac{r^2}{r_0}\right)^n \HH^n(\set{x \in \mathbf{S}^n_+ | \text{the line $(0x)$ meets $A$}})\\
                        &\leq \frac{1}{2} \left(\frac{r^2}{r_0}\right)^n \HH^n(\set{x \in \mathbf{S}^n | \text{the line $(0x)$ meets $A$}})\\
                        &\leq \frac{\sigma_n}{2} \left(\frac{r^2}{r_0}\right)^n \gamma_{1,n+1}(\set{L | \text{$L$ meets $A$}}).
        \end{align}

        Next, we show that for all $L \in G(1,n+1)$ which meets $A$, we have
        \begin{equation}
            \mathrm{d}(L,L_0) \leq \sqrt{1 - \left(\frac{r}{r_0}\right)^2}
        \end{equation}
        First of all, $L$ cannot be orthogonal to $L_0$ so $\mathrm{d}(L,L_0) < 1$. There exists a linear application $\varphi_L\colon L_0 \to L_0^\perp$ such that
        \begin{equation}
            L = \Set{x + \varphi_L(x) | x \in L_0}.
        \end{equation}
        and the point of intersection between $L$ and $H$ is $x_L = x_0 + \varphi_L(x_0)$. According to (\ref{grassNorm}), we have
        \begin{equation}
            \mathrm{d}(L,L_0) = \frac{\norm{\varphi_L}}{\sqrt{1 + \norm{\varphi_L}^2}}.
        \end{equation}
        Since $\varphi_L$ is defined on a line, it is easy to compute $\norm{\varphi}$:
        \begin{equation}
            \norm{\varphi_L} = \frac{\abs{\varphi_L(x_0)}}{\abs{x_0}} = \frac{\sqrt{\abs{x_L}^2 - \abs{x_0}^2}}{\abs{x_0}}
        \end{equation}
        and the result follows.
    \end{proof}

    \section{Construction of the Federer--Fleming Projection}\label{appendix_FF}
    We recall a few notions. Let $K$ be a $n$-complex. For any integer $d$, the set $\set{A \in K | \mathrm{dim}(A) = d}$ is denoted by $K^d$. For any subcomplex $L$ of $K$, we define
    \begin{equation}
        U(L) = \bigcup \set{\mathrm{int}(A) | A \in L}.
    \end{equation}
    According the properties of $n$-complexes, this is an open set of $\R^n$. We define the gauge $\zeta^d$ on Borel subsets of $\R^n$ by
    \begin{equation}
        \zeta^d(E) = \int_{G(d,n)} \HH^d(p_V(E)) \, \mathrm{d}V.
    \end{equation}
    For a cell $A$, we define the restriction of this gauge to $A$
    \begin{equation}
        \zeta^d \mres A(E) = \int_{G(\mathrm{aff}(A),d)} \HH^d(p_V(A \cap E)) \, \mathrm{d}V,
    \end{equation}
    where $\mathrm{aff}(A)$ is the affine span of $A$ and $G(\mathrm{aff}(A),d)$ is the set of all $d$-linear planes of $\mathrm{aff}(A)$ centered at an arbitrary point.

    We recall what we want to do. Let us say $K$ as the set of dyadic cells of sidelength $2^{-k}$ which are included in $Q = [-1,1]^n$ but not in $\partial Q$. Let us say that $E$ is a $1$-dimensional subset of $Q$. We perform a radial projection in each cell $A \in L^n$, then in each cell $A \in L^{n-1}, \ldots$ until the cells $A \in L^2$. This operation sends $E$ to
    \begin{equation}
        \abs{K} \setminus \bigcup \set{\mathrm{int}(A) | A \in K\, \mathrm{dim} \, A \geq 2} = \partial Q \cup \abs{K^1}.
    \end{equation}
    Throughout the construction, we choose the centers of projections so that $\HH^d(\phi(E)) \leq C \HH^d(E)$ and $\zeta^d(\phi(E)) \leq C \zeta^d(E)$. However, at the end we will only evaluate $\zeta^d$ on restriction to cells $A \in K^d$ so that $\zeta^d$ reduces to $\HH^d$.

    We prove the construction in two steps. In the first Lemma, we forget about the set $E$ and we just show that making radial projections in the cells of a subcomplex $L$ of $K$ yields a Lipschitz retraction onto $\abs{K} \setminus U(L)$. In the second Lemma, we focus about the choice of projection centers.
    \begin{lem}\label{abstract_FF}
        Let $K$ be a $n$-complex, let $L$ be a subcomplex of $K$ which does not contain $0$-dimensional cells. There exists a Lipschitz function $\phi\colon \abs{K} \to \abs{K}$ satisfying the following properties:
        \begin{enumerate}[label=(\roman*)]
            \item for all $A \in K$, $\phi(A) \subset A$;
            \item $\phi = \mathrm{id}$ in $\abs{K} \setminus U(L)$;
            \item there exists a relative open set $O \subset \abs{K}$ containing $\abs{K} \setminus U(L)$ such that
                \begin{equation}
                    \phi(O) \subset \abs{K} \setminus U(L).
                \end{equation}
        \end{enumerate}
    \end{lem}
    \begin{proof}
        The letter $C$ plays the role of a constant $\geq 1$ that depends on $n$, $\kappa$, $\Gamma$. Its value can increase from one line to another (but a finite number of times). We build by induction a family of locally Lipschitz functions
        \begin{equation}
            (\phi_m)\colon \abs{K} \to \abs{K}
        \end{equation}
        indexed by a decreasing integer $m = n+1, \ldots, 1$. The application $\phi_m$ is obtained by composing $\phi_{m+1}$ with radial projections in the cells $A \in L$ of dimension $m$. We define
        \begin{equation}
            U_m(L) = \bigcup \set{\mathrm{int}(A) | A \in L,\ \mathrm{dim} \, A \geq m}
        \end{equation}
        and we require that
        \begin{enumerate}[label=(\roman*)]
            \item for all $A \in K$, $\phi_m(A) \subset A$,
            \item $\phi_m = \mathrm{id}$ in $\abs{K} \setminus U_m(L)$,
            \item there exists a relative open set $O \subset \abs{K}$ containing $\abs{K} \setminus U_m(L)$ such that
                \begin{equation}
                    \phi(O) \subset \abs{K} \setminus U_m(L).
                \end{equation}
        \end{enumerate}
        The induction starts with $\phi_{n+1} = \mathrm{id}$. Assume that $\phi_{m+1}$ is well defined for some $m \leq  n$. We are going to postcompose $\phi_{m+1}$ with a function $\psi_m$. We define first $\psi_m$ on $\abs{K} \setminus U_{m+1}(L)$. We observe that
        \begin{equation}
            \abs{K} \setminus U_{m+1}(L) \subset \left(\abs{K} \setminus U_m(L)\right) \cup \abs{L^m}
        \end{equation}
        but according to Lemma \ref{complex_lemma}, $\abs{L^m}$ is disjoint from $U_{m+1}(L)$ so
        \begin{equation}
            \abs{K} \setminus U_{m+1}(L) = \left(\abs{K} \setminus U_m(L)\right) \cup \abs{L^m}
        \end{equation}
        We set
        \begin{equation}\label{init_psi_m}
            \psi_m = \mathrm{id}\ \text{in}\ \abs{K} \setminus U_m(L).
        \end{equation}
        and we set $\psi_m$ in each cell $A \in L^m$ to be a radial projection. For $A \in L^m$, let $x_A$ be the center of $A$ and let $\delta_A > 0$ be such that
        \begin{equation}
            A \cap \overline{B}(x_A, \delta_A) \subset \mathrm{int}(A).
        \end{equation}
        We define $\psi_m$ in $A \setminus B(x_A, \delta_A)$ to be the radial projection centered in $x_A$ onto $\partial A$. As $A$ is a face of cube, $\psi_m$ is Lipschitz with a controlled constant: for all $x, y \in A \setminus B(x_A, \delta_A)$,
        \begin{equation}
            \abs{\psi_m(x) - \psi_m(y)} \leq C \mathrm{diam}(A) \delta_A^{-1} \abs{x - y}.
        \end{equation}
        We extend $\psi_m$ as a $C \mathrm{diam}(A) \delta_A^{-1}$-Lipschitz function $\psi_m \colon A \to A$. The function $\psi_m$ is now well defined on $\abs{K} \setminus U_{m+1}(L)$. By construction, $\psi_m$ is Lipschitz in each cell $A \in L^m$ and it is the identity map on $\abs{K} \setminus U_m(L)$. We are going to deduce that $\psi_m$ is Lipschitz on $\abs{K} \setminus U_{m+1}(L)$. Consider $A, B \in L^m$ such that $A \ne B$ and let $x \in A$ and $y \in B$. According to Lemma \ref{complex_lemma}, we have $V_A \cap B = \emptyset$. And by the properties of $n$-complexes, there exists a constant $\kappa \geq 1$ (depending on $n$) such that $V_A(\kappa) \subset V_A$. We deduce that $y \notin V_A(\kappa)$, i.e., $\mathrm{d}(y, \partial A) \leq \kappa \mathrm{d}(y,A)$. Let $z \in \partial A$ be such that $\abs{y - z} = \mathrm{d}(y, \partial A)$, in particular
        \begin{equation}
            \abs{y - z} \leq \kappa \abs{x - y}.
        \end{equation}
        Similarly, one can find $z' \in \partial B$ such that
        \begin{equation}
            \abs{x - z'} \leq \kappa \abs{x - y}.
        \end{equation}
        Using the triangular inequality, we see that
        \begin{equation}
            \abs{x - z}, \abs{y - z'}, \abs{z - z'} \leq C \kappa \abs{x - y}.
        \end{equation}
        Thus
        \begin{align}
            \begin{split}
                \abs{\psi_m(x) - \psi_m(y)} &   \leq \abs{\psi_m(x) - \psi_m(z)} + \abs{\psi_m(z) - \psi_m(z')}\\
                                            &   \qquad + \abs{\psi_m(z') - \psi_m(y)}
            \end{split}\\
                                        &   \leq \abs{\psi_m(x) - \psi_m(z)} + \abs{z - z'} + \abs{\psi_m(z') - \psi_m(y)}\\
                                        &   \leq C \kappa (\mathrm{diam}(A) \delta_A^{-1} + \mathrm{diam}(B) \delta_{B}^{-1} + 1) \abs{x - y}.
        \end{align}
        Next consider $A \in L^m$, $x \in A$ and $y \in \abs{K} \setminus U_m(L)$. We have $V_A(\kappa) \subset V_A \subset U_m(L)$ so $\mathrm{d}(y, \partial A) \leq \kappa \mathrm{d}(y,A)$. Let $z \in \partial A$ be such that $\abs{y - z} = \mathrm{d}(y, \partial A)$, in particular
        \begin{equation}
            \abs{y - z} \leq \kappa \abs{x - y}.
        \end{equation}
        Using the triangular inequality, we see that $\abs{x - z} \leq C\kappa \abs{x - y}$. It follows that
        \begin{align}
            \abs{\psi_m(x) - \psi_m(y)} &   \leq \abs{\psi_m(x) - \psi_m(z)} + \abs{\psi_m(z) - \psi_m(y)}\\
                                        &   \leq \abs{\psi_m(x) - \psi_m(z)} + \abs{z - y}\\
                                        &   \leq C \kappa (\mathrm{diam}(A) \delta_A^{-1} + 1) \abs{x - y}.
        \end{align}
        As $x_A$ is the center of $A$, it is possible to choose $\delta_A$ in such a way that $\mathrm{diam}(A) \delta_A^{-1} \leq C$. This concludes the proof that $\psi_m$ is well defined and $C$-Lipschitz on $\abs{K} \setminus U_{m+1}(L)$. In a future lemma, we will copy the proof but in a situation where we cannot choose $x_A$ to be the center of $A$. Thus we will not be able to control the ratio $\mathrm{diam}(A) \delta_A^{-1}$. 

        In this paragraph, we extend $\psi_m$ over $\abs{K}$ in such way that for each $A \in K$, $\psi_m(A) \subset A$. We present the extension procedure and we will check afterward that $\psi_m$ preserves every face. If $m = n$, we have
        \begin{equation}
            \abs{K} \setminus U_{m+1}(L) = \abs{K}
        \end{equation}
        so $\psi_m$ is already defined over $\abs{K}$. Assume that $m < n$. For $A \in L^{m+1}$, Lemma \ref{complex_lemma} says that
        \begin{equation}
            A \setminus U_{m+1}(L) = A \setminus \mathrm{int}(A) = \partial A
        \end{equation}
        so $\psi_m$ is defined and $C$-Lipschitz on $\partial A$ and we can extend it as a $C$-Lipschitz function $\psi_m \colon A \to A$. We recall that
        \begin{equation}
            \abs{K} \setminus U_{m+2}(L) = \left(\abs{K} \setminus U_{m+1}(L)\right) \cup \abs{L^{m+1}}
        \end{equation}
        so the function $\psi_m$ is now well defined on $\abs{K} \setminus U_{m+2}(L)$. It is is Lipschitz in each cell $A \in L^{m+1}$ and it is Lipschitz on $\abs{K} \setminus U_{m+1}(L)$. We can deduce that it is Lipschitz on $\abs{K} \setminus U_{m+2}(L)$ using the same proof than in the previous paragraph. We continue the process until $\psi_m$ is defined on $\abs{K}$. Let us prove that $\psi_m$ preserves every face $A \in K$. There are two cases to distinguish. If $A \subset \abs{K} \setminus U_m(L)$, then $\psi_m = \mathrm{id}$ on $A$ by (\ref{init_psi_m}). If $A \cap U_m(L) \ne \emptyset$, there exists $B \in L$ such that $\mathrm{dim} \, B \geq m$ and $A \cap \mathrm{int}(B) \ne \emptyset$. By the first point of Lemma \ref{complex_lemma}, $B \subset A$ so $\mathrm{dim} \, A \geq m$ and, by the properties of subcomplexes, $A \in L$. We conclude that $\psi_m$ preserves $A$ by construction.

        We finally define $\phi_m = \psi_m \circ \phi_{m+1}$. By construction, $\phi_m$ is $C$-Lipschitz and satisfies the first and second requirements of the induction. Let us check that it satisfies the third requirement. By assumption, there exists a relative open set $O$ of $\abs{K}$ which contains $\abs{K} \setminus U_{m+1}(L)$ and such that
        \begin{equation}
            \phi_{m+1}(O) \subset \abs{K} \setminus U_{m+1}(L)
        \end{equation}
        We will solve the induction with the set
        \begin{equation}
            O' = O \setminus \phi_{m+1}^{-1}\left(\bigcup\limits_{A \in L^m} A \cap \overline{B}(x_A, \delta_A)\right).
        \end{equation}
        We justify that $O'$ contains $\abs{K} \setminus U_m(L)$. First, we have trivially
        \begin{equation}
            \abs{K} \setminus U_m(L) \subset \abs{K} \setminus U_{m+1}(L) \subset O.
        \end{equation}
        In addition $\phi_{m+1} = \mathrm{id}$ on $\abs{K} \setminus U_m(L)$ and this set is disjoint from
        \begin{equation}
            \bigcup\limits_{A \in L^m} A \cap \overline{B}(x_A, \delta_A) \subset U_m(L).
        \end{equation}
        We justify that $O'$ is open. The family $(A \cap \overline{B}(x_A, \delta_A))_A$ is a locally finite family of closed sets in $\abs{K}$ so its union is relatively closed. Then $O'$ is relatively open in $\abs{K}$ by continuity of $\psi_{m+1}$. Finally, we justify that $\psi_m(O') \subset \abs{K} \setminus U_m(L)$. We have $\phi_{m+1}(O) \subset \abs{K} \setminus U_{m+1}(L)$ and
        \begin{equation}
            \abs{K} \setminus U_{m+1}(L) \subset \left(\abs{K} \setminus U_m(L)\right) \cup \abs{L^m}
        \end{equation}
        so
        \begin{align}
            \phi_m(O')  &   \subset \psi_m \circ \phi_{m+1}(O')\\
                        &   \subset \psi_m \left(\abs{K} \setminus \left(U_{m+1}(L) \cup \bigcup\limits_{A \in L^m} A \setminus \overline{B}(x_A, \delta_A)\right)\right)\\
                        &   \subset \abs{K} \setminus U_m(L)
        \end{align}
        by construction of $\psi_m$.
    \end{proof}

    \begin{lem}[Federer--Fleming projection]\label{lem_FF}
        Let $K$ be $n$-complex. Let $0 \leq d < n$ be an integer and let $E$ be Borel subset of $\abs{K}$ such that $\HH^{d+1}(\abs{K} \cap \overline{E}) = 0$. Then there exists a locally Lipschitz function $\phi\colon \abs{K} \to \abs{K}$ satisfying the following properties:
        \begin{enumerate}[label=(\roman*)]
            \item for all $A \in K$, $\phi(A) \subset A$;
            \item $\phi = \mathrm{id}$ in $\abs{K} \setminus \bigcup \set{\mathrm{int}(A) | \mathrm{dim} \, A > d}$;
            \item there exists a relative open set $O \subset \abs{K}$ such that $E \subset O$ and
                \begin{equation}
                    \phi(O) \subset \abs{K} \setminus \bigcup \set{\mathrm{int}(A) | \mathrm{dim} \, A > d};
                \end{equation}
            \item for all $A \in K$,
                \begin{equation}
                    \HH^d(\phi(A \cap E)) \leq C \HH^d(A \cap E);
                \end{equation}
            \item for all $A \in K^d$,
                \begin{equation}
                    \HH^d(A \cap \phi(E)) \leq C \int_{G(d,n)} \HH^d(p_V(V_A \cap E)) \, \mathrm{d}V,
                \end{equation}
        \end{enumerate}
        where $C \geq 1$ is a constant that depends only on $n$.
    \end{lem}
    \begin{rmk}
        We do not say that for all $A \in K^n$,
        \begin{equation}
            \HH^d(\phi(A \cap E)) \leq C \int_{G(d,n)} \HH^d(p_V(A \cap E)) \, \mathrm{d}V
        \end{equation}
        because $\phi(A \cap E)$ may not be included in the $d$-skeleton of $K$. Thus $\zeta^d$ may not reduce to $\HH^d$ on $\phi(E \cap A)$.
    \end{rmk}
    \begin{proof}
        The letter $C$ plays the role of a constant $\geq 1$ that depends on $n$. Its value can increase from one line to another (but a finite number of times). The principle of the proof is to build by induction a family of locally Lipschitz functions
        \begin{equation}
            (\phi_m)\colon \abs{K} \to \abs{K}
        \end{equation}
        indexed by a decreasing integer $m = n+1, \ldots, d+1$. The application $\phi_m$ is obtained by composing $\phi_{m+1}$ with radial projections in the cells $A \in L$ of dimension $m$. We define
        \begin{equation}
            U_m(K) = \bigcup \set{\mathrm{int}(A) | A \in K,\ \mathrm{dim} \, A \geq m}.
        \end{equation}
        and require that
        \begin{enumerate}[label=(\roman*)]
            \item for all $A \in K$, $\phi_m(A) \subset A$;
            \item $\phi_m = \mathrm{id}$ in $\abs{K} \setminus U_m(K)$;
            \item there exists a relative open set $O \subset \abs{K}$ containing $\abs{K} \setminus U_m(K)$ such that
                \begin{equation}
                    \phi(O) \subset \abs{K} \setminus U_m(K);
                \end{equation}
            \item for all $A \in K$,
                \begin{equation}
                    \HH^d(\phi_m(A \cap E)) \leq C \HH^d(A \cap E);
                \end{equation}
            \item for all $A \in K$,
                \begin{equation}
                    \zeta^d \mres A(\mathrm{int}(A) \cap \phi_m(E)) \leq C \zeta^d(V_A \cap E).
                \end{equation}
        \end{enumerate}
        The proof follows the same scheme as Lemma \ref{abstract_FF} (where the subcomplex $L$ is the set of cells $A \in K$ of dimension $\geq d+1$) but we choose the center of projection wisely. The induction starts with $\phi_{n+1} = \mathrm{id}$. Assume that $\phi_{m+1}$ is well defined for some $d < m \leq n$. We post-compose $\phi_{m+1}$ with a function $\psi_m$ made of radial projections in the cells of $K$ of dimension $m$. Fix $A \in K^m$. For $x \in \mathrm{int}(A)$, let $\psi_x$ be the radial projection onto $\partial A$ centered at $x$. We want a center of projection $x_A \in \frac{1}{2} A$ such that
        \begin{enumerate}
            \item $x_A \notin \overline{\phi_{m+1}(E)}$;
            \item for all $B \in K$ containing $A$,
                \begin{equation}
                    \HH^d(\psi_{x_A}(\mathrm{int}(A) \cap F_B)) \leq C \HH^d(\mathrm{int}(A) \cap F_B)
                \end{equation}
                where $F_B = \phi_{m+1}(B \cap E)$;
            \item for all $B \in K$ included in $\partial A$,
                \begin{equation}
                    \zeta^d \mres B (\psi_{x_A}(\mathrm{int}(A) \cap F)) \leq C \zeta^d \mres A(\mathrm{int}(A) \cap F),
                \end{equation}
                where $F = \phi_{m+1}(E)$.
        \end{enumerate}
        We are going to show that such centers $x_A$ exists. First, we have
        \begin{equation}\label{x_A-0}
            \HH^{d+1}(\overline{\phi_{m+1}(E)}) = 0.
        \end{equation}
        because $\phi_{m+1}$ is locally Lipschitz. This means that the first requirement is satisfied for $\HH^m$-almost every $x_A \in \mathrm{int}(A)$. Now we deal with the second and the third requirements. Let $F_1, \ldots, F_N$ be a Borel subset of $A$ such that $\HH^{d+1}(\overline{F_i}) = 0$. We fix an index $i$ and we apply Lemma \ref{lem_DS1} to $F_i$ in combination with the Markov inequality. We estimate that for $\lambda > 0$,
        \begin{equation}
            \HH^m(\set{x \in \tfrac{1}{2} A \setminus \overline{F_i} | \HH^d(\psi_x(F_i)) \geq \lambda \HH^d(F_i)}) \leq C \lambda^{-1} \mathrm{diam}(A)^m 
        \end{equation}
        and thus
        \begin{multline}
            \HH^m(\set{x \in \tfrac{1}{2} A \setminus \bigcup_i \overline{F_i} | \exists i,\ \HH^d(\psi_x(F_i)) \geq \lambda \HH^d(F_i)}) \\\leq N C \lambda^{-1} \mathrm{diam}(A)^m.
        \end{multline}
        We deduce that for $\lambda$ big enough (depending on $n$, $N$) there exists $x \in \tfrac{1}{2} A \setminus \bigcup \overline{F_i}$ such that for all $i$,
        \begin{equation}
            \HH^d(\psi_x(F_i)) \leq \lambda\HH^d(F_i)).
        \end{equation}
        We apply this reasoning to the sets $F_B$ and $F$. Note that $N \leq 3^n + 1$ because there are at most $3^n$ cells $B$ that contain $A$. The center of projection $x_A$ is chosen. We let $\delta_A > 0$ be such that
        \begin{equation}\label{delta_A_definition}
            A \cap \overline{B}(x_A, \delta_A) \subset \mathrm{int}(A) \setminus \phi_{m+1}(E)
        \end{equation}
        and we define $\psi_m$ in $A \setminus B(x_A, \delta_A)$ to be the radial projection centered in $x_A$ onto $\partial A$. In particular $\psi_m$ is $C\mathrm{diam}(A) \delta_A^{-1}$ Lipschitz and we extend $\psi_m$ as a $C \mathrm{diam}(A) \delta_A^{-1}$-Lipschitz function $\psi_m \colon A \to A$. The construction of $\psi_m$ continues just as in the proof of Lemma \ref{abstract_FF} and $\phi_m$ is finally defined as $\phi_m = \psi_m \circ \phi_{m+1}$.

        It is left to show that for all $B \in K$,
        \begin{equation}\label{goal_phi1}
            \HH^d(\phi_m(B \cap E)) \leq C \HH^d(B \cap E);
        \end{equation}
        and for all $A \in K$,
        \begin{equation}\label{goal_phi2}
            \zeta^d \mres A(\mathrm{int}(A) \cap \phi_m(E)) \leq C \zeta^d(V_A \cap E).
        \end{equation}
        The other requiremements are practically proved in \ref{abstract_FF}. We fix $B \in K$ and we prove that
        \begin{equation}
            \HH^d(\phi_m(B \cap E)) \leq C\HH^d(B \cap E)).
        \end{equation}
        We have
        \begin{align}
            \phi_{m+1}(B \cap E)    &\subset \abs{K} \setminus U_{m+1}\\
                                    &\subset (\abs{K} \setminus U_m) \cup \bigcup_{A \in L^m} \mathrm{int}(A)\label{line_union}
        \end{align}
        We observe that $\phi_{m+1}(B \cap E) \subset B$ so for $A \in L^m$ such that $\mathrm{int}(A) \cap B \ne \emptyset$, we have $A \subset B$. The union in (\ref{line_union}) is in fact indexed by $A \in L^m$ such that $A \subset B$. By construction, we have $\psi_m = \mathrm{id}$ on $\abs{K} \setminus U_{m+1}$ and for all $A \in L^m$ such that $A \subset B$,
        \begin{equation}
            \HH^d(\psi_m(\mathrm{int}(A) \cap \phi_{m+1}(B \cap E))) \leq C \HH^d(\mathrm{int}(A) \cap \phi_{m+1}(B \cap E)).
        \end{equation}
        As $\HH^d$ is $\sigma$-additive on Borel sets and the cells $A \in K^m$ have disjoint interiors, we conclude that 
        \begin{align}
            \HH^d(\phi_m(B \cap E)) &\leq C \HH^d(\phi_{m+1}(B \cap E))\\
                                    &\leq C \HH^d(B \cap E).
        \end{align}

        We fix $A \in K$ and we prove that
        \begin{equation}
            \zeta^d \mres A(\mathrm{int}(A) \cap \phi_m(E)) \leq C \zeta^d(V_A \cap E).
        \end{equation}
        If $\mathrm{dim}(A) > m$ this is trivial because $\phi_m(E) \subset \abs{K} \setminus U_m$. If $\mathrm{dim}(A) < m-1$, this is trivial because $\phi_m = \mathrm{id}$ on $\abs{K} \setminus U_m$ (and we use (\ref{zeta_mres})). We assume that $\mathrm{dim}(A) = m-1$. Here again,
        \begin{align}
            \psi_m^{-1}(\mathrm{int}(A)) \cap \phi_{m+1}(E) &\subset \abs{K} \setminus U_{m+1}\\
                                                            &\subset (\abs{K} \setminus U_m) \cup \bigcup_{B \in K^m} \mathrm{int}(B)\label{line_union2}.
        \end{align}
        For $B \in K^m$ such that $\psi_m^{-1}(\mathrm{int}(A)) \cap \mathrm{int}(B) \ne \emptyset$, we have $\mathrm{int}(A) \cap B \ne \emptyset$, because $\psi_m(B) \subset B$. As a consequence $A \subset B$ but since $\mathrm{dim} \, A < \mathrm{dim} \, B$, we also have $\mathrm{int}(B) \cap A = \emptyset$. This proves that the union in (\ref{line_union2}) is indexed by $B \in K^m$ such that $A \subset \partial B$. By construction, we have $\psi_m = \mathrm{id}$ on $\abs{K} \setminus U_{m+1}$ and for all $B \in K^m$ such that $A \subset \partial B$,
        \begin{align}
            \zeta^d \mres A(\psi_m(\mathrm{int}(B) \cap \phi_{m+1}(E))) &\leq C \zeta^d(\mathrm{int}(B) \cap \phi_{m+1}(E))\\
                                                                        &\leq C \zeta^d(V_B \cap E)\\
                                                                        &\leq C \zeta^d(V_A \cap E).
        \end{align}
        As there are at most $3^n$ cells $B \in K^m$ such that $A \subset \partial B$, we conclude that
        \begin{align}
            \zeta^d(\mathrm{int}(A) \cap \phi_m(E)) \leq C \zeta^d(V_A \cap E).
        \end{align}
    \end{proof}

    In the next Lemma, we provide two conditions to control the Lipschitz constant of $\phi$. The second condition is a weak form of Ahlfors-regularity (see Remark \ref{rmk_af_semi} below the proof)
    \begin{lem}\label{lem_semi}
        We keep the notations of the Lemma \ref{lem_FF}. Let us assume that there exists a constant $\lambda \geq 1$ such that
        \begin{enumerate}[label=(\roman*)]
            \item for all $A, B \in K$ which are not $0$-cells and such that $A \subset B$, we have $\mathrm{diam}(B) \leq \lambda \mathrm{diam}(A)$;
            \item for all bounded subset $S \subset E$, for all radius $\rho > 0$, the set $S$ can be covered by at most $\lambda \rho^{-d} \mathrm{diam}(S)^d$ balls of radius $r$.
        \end{enumerate}
        Then $\phi$ can be taken $C$-Lipschitz where $C$ depends on $n$, $\lambda$.
    \end{lem}
    \begin{proof}
        The letter $C$ plays the role of constant $\geq 1$ that depends on on $n$ and $\lambda$. Given a $C$-Lipschitz $\phi_{m+1}$, we want to build a $C$-Lipschitz $\psi_m$. According to the proof of Lemma \ref{abstract_FF}, it suffices that for all $A \in K^m$,
        \begin{equation}
            \mathrm{diam}(A) \delta_A^{-1} \leq C.
        \end{equation}
        We prove the existence of $x_A \in \frac{1}{2}$ such that
        \begin{enumerate}
            \item $\mathrm{diam}(A) \mathrm{d}(x, \phi_{m+1}(E))^{-1} \leq C$;
            \item for all $B \in K$ containing $A$,
                \begin{equation}
                    \HH^d(\psi_{x_A}(\mathrm{int}(A) \cap F_B)) \leq C \HH^d(\mathrm{int}(A) \cap F_B)
                \end{equation}
                where $F_B = \phi_{m+1}(B \cap E)$;
            \item for all $B \in K$ included in $\partial A$,
                \begin{equation}
                    \zeta^d \mres B (\psi_{x_A}(\mathrm{int}(A) \cap F)) \leq C \zeta^d \mres A(\mathrm{int}(A) \cap F),
                \end{equation}
                where $F = \phi_{m+1}(E)$.
        \end{enumerate}
        Remember that to obtain the second requirement on $x_A$, we have chosen $C$ big enough (depending on $n$) so that for $B \in K$ containing $A$,
        \begin{equation}
            \mathrm{diam}(A)^{-m} \HH^m(\set{x \in \mathrm{int}(A) | \HH^d(\psi_x(F_B)) \geq C \HH^d(F_B)})
        \end{equation}
        is sufficiently small (depending on $n$). We have obtained the third requirement on $x_A$ similarly. Now, we also want $C$ big enough (depending on $n$, $\lambda$) so that
        \begin{equation}
            \mathrm{diam}(A)^{-m} \HH^m(\set{x \in \mathrm{int}(A) | \mathrm{d}(x, \phi_{m+1}(E)) \leq C^{-1} \mathrm{diam}(A)})
        \end{equation}
        is sufficiently small (depending on $n$). Thus, the points $x \in \mathrm{int}(A)$ that do not satisfy all our criteria will have a small $\HH^m$-measure compare to $\HH^m(A)$. Fix $0 < \delta \leq 1$. We recall that
        \begin{equation}
            \mathrm{int}(A) \cap \phi_{m+1}(E) \subset \bigcup_B \phi_{m+1}(B \cap E),
        \end{equation}
        where $\bigcup_B$ is indexed by the cells $B \in K$ containing $A$. For such $B$, we can cover $E \cap B$ by at most $C \delta^{-d}$ balls of radius $\delta \mathrm{diam}(A)$. Since $\phi_{m+1}$ is $C$-Lipschitz and since, by Definition \ref{defi_system}, there are at most $3^n$ cells $B \in K$ containing $A$, the set $\mathrm{int}(A) \cap \phi_{m+1}(E)$ is covered by at most $C\delta^{-d}$ balls of radius $C\delta \mathrm{diam}(A)$. We deduce that
        \begin{multline}
            \HH^m(\set{x \in \mathrm{int}(A) | \mathrm{d}(x, \phi_{m+1}(E)) < C\delta\mathrm{diam}(A)}) \\ \leq C \delta^{m-d} \mathrm{diam}(A)^m.
        \end{multline}
        As $m > d$ and $\delta$ is arbitrary small, we can find $C$ big enough so that
        \begin{equation}
            \mathrm{diam}(A)^{-m} \HH^m(\set{x \in \mathrm{int}(A) | \mathrm{d}(x, \phi_{m+1}(E)) \leq C^{-1} \mathrm{diam}(A)})
        \end{equation}
        is sufficiently small.
    \end{proof}
    \begin{rmk}\label{rmk_af_semi}
        Let $E \subset \R^n$ be a Borel set and $B$ be a ball of radius $R > 0$. We assume that there exists a constant $\lambda \geq 1$ such that for all $x \in E \cap B$, for all $0 \leq r \leq 2R$,
        \begin{equation}\label{AF_semi}
            \lambda^{-1} r^d \leq \HH^d(E \cap B(x,r)) \leq \lambda r^d.
        \end{equation}
        Then there exists a constant $C \geq 1$ (depending on $\lambda$, $d$) such that for all subset $S \subset E \cap B$, for all radius $r > 0$, the set $S$ can be covered by at most $C r^{-d} \mathrm{diam}(S)^d$ balls of radius $r$. Here is the proof.

        We assume that $S$ contains at least one point, denoted by $x_S$. The property is trivial for $r > \mathrm{diam}(S)$ because $S$ is covered by by $B(x_S,r)$. Now we work with $r \leq \mathrm{diam}(S)$. Let $(x_i)$ be a maximal sequence of points in $S$ such that $\abs{x_i - x_j} \geq r$. By maximality, $S$ is covered by the balls $B(x_i,r)$. Then, we estimate the cardinal $N$ of such family. The balls $E \cap B(x_i,\tfrac{1}{2}r)$ are disjoints and included in $E \cap B(x_S, 2\mathrm{diam}(S))$. It follows that
        \begin{equation}
            \sum \HH^d(E \cap B(x_i,\tfrac{1}{2}r)) \leq \HH^d(E \cap B(x_S,2\mathrm{diam}(S))).
        \end{equation}
        We apply (\ref{AF_semi}) and obtain
        \begin{equation}
            (2^d \lambda)^{-1} N r^d \leq 2^d \lambda \mathrm{diam}(S)^d.
        \end{equation}
    \end{rmk}
\end{appendices}

\textsc{Université Paris-Saclay, CNRS, Laboratoire de mathématiques d'Orsay, 91405, Orsay, France.}
\end{document}